\documentclass[12pt,reqno]{amsart}
\usepackage{graphicx}
\usepackage{color}
\usepackage{amsmath,amssymb,amsthm,amsfonts}
\usepackage{graphicx}
\usepackage{color}
\usepackage{multicol}

\usepackage{mathptmx}
\usepackage{enumerate}
\usepackage[numbers,sort&compress]{natbib}
\numberwithin{equation}{section}
\newcommand{\R}{{\mathbb R}}

\newcommand{\be}{\begin{equation}}
\newcommand{\ee}{\end{equation}}

\newcommand{\pt}{\partial}

\newcommand{\ben}{\begin{eqnarray*}}
\newcommand{\enn}{\end{eqnarray*}}

\newcommand{\Om}{\Omega}

\newcommand{\va}{\varepsilon}

\newcommand{\lge}{\langle}
\newcommand{\rge}{\rangle}

\newcommand{\na}{\nabla}
\newcommand{\ol}{\overline}

\newtheorem{theorem}{\textbf Theorem}[section]
\newtheorem{lemma}{\textbf Lemma}[section]
\newtheorem{rem}{\textbf Remark}[section]

\newtheorem{prop}{\textbf Proposition}[section]

\def\endProof{{\hfill$\Box$}}
\begin{document}

\author{Qianqian Hou}
\address{Institute for Advanced Study in Mathematics, Harbin Institute of Technology, Harbin 150001, People's Republic of China}
\email{qianqian.hou@hit.edu.cn}

\title[Boundary layer effects induced by the fluid in a chemotaxis-Navier-Stokes system]{Boundary layer effects induced by the fluid in a chemotaxis-Navier-Stokes system}

\begin{abstract}
This paper is concerned with the boundary layer problem on a chemotaxis-Navier-Stokes system modelling the boundary layer formation of aerobic bacteria in fluids. Completing this system with Neumann boundary conditions on oxygen, we show that gradients of its second solution component in the half plane of $\mathbb{R}^2$ possess boundary layer effects as the oxygen diffusion rate goes to zero. However, neglecting the influence of the fluid, gradients of solutions to the chemotaxis-only subsystem no longer present such boundary layer effects. It seems that the boundary layer effect for the chemotaxis-Navier-Stokes system under Neumann boundary conditions on oxygen is induced by the presence of fluids.
\end{abstract}

\subjclass[2010]{35A01, 35B40, 35K57, 35Q92, 92C17}

\keywords{Boundary layers, chemotaxis, Navier-Stokes equations, asymptotic analysis, vanishing diffusion limit}
\maketitle

\section{Introduction}
\subsection{Background and literature review} Oxytactic bacteria living in water like \emph{Bacillus subtilis} swim up along the oxygen gradients and quickly aggregate in a relatively thin layer below the water surface (cf. \cite{hillesdon-pedley1996, hillesdon-pedley-kessler1995}). The following chemotaxis-Navier-Stokes system has been proposed in \cite{Tuval} to describe the interplay of the bacteria, oxygen and fluids in this process: 
\begin{equation}\label{e00}
\left\{
\begin{aligned}
&m_t+\vec{u}\cdot\na m+\nabla\cdot(m \chi(c)\na c)=D_m \Delta m, \quad \quad\ \ \  \,&(\vec{x},t)\in \Omega\times(0,\infty),
\\
&c_t+\vec{u}\cdot \na c+m f(c)=\va\Delta c,\qquad \quad\qquad\ \ \qquad \ \,&(\vec{x},t)\in \Omega\times(0,\infty),\\
&\vec{u}_t+ \kappa\vec{u}\cdot\na \vec{u}+\nabla p+m\nabla \phi=D\Delta\vec{u},\quad \quad \quad \ \  \ \ \quad &(\vec{x},t)\in \Omega\times(0,\infty),\\
&\na\cdot\vec{u}=0,\qquad\qquad\qquad\quad \quad\qquad\quad\qquad\ \ \ \ \ \   \, \  &(\vec{x},t)\in \Omega\times(0,\infty),\\
\end{aligned}
\right.
\end{equation}
 where $\Om\subset \R^d$ with $d\geq 1$. The unknowns $m(\vec{x},t)$, $c(\vec{x},t)$, $\vec{u}(\vec{x},t)$ and $p(\vec{x},t)$ are the bacteria density, oxygen concentration, fluid velocity and the associated pressure. The positive constants $D_m$, $\va$ and $D$ denote diffusion rates of the bacterial cells, oxygen and velocity, respectively. The first two equations in \eqref{e00} comprise the Keller-Segel model describing the chemotactic movement of bacteria due to the uneven distributions of the oxygen in the fluids with chemotactic intensity $\chi(c)>0$ and oxygen consumption rate $f(c)>0$, where both bacteria and oxygen diffuse and are convected with the fluid. The last two equations in \eqref{e00} are the well-known incompressible Navier-Stokes equations with the additional term $m\nabla \phi$ accounting for the gravity force exerted on the fluids by the bacteria cells, where the given potential $\phi(\vec{x})$ is independent of the temporal variable $t$.
 
When $\Om=\R^2$, under certain structural conditions on $\chi$ and $f$, global weak solutions on system \eqref{e00} with $\kappa=0$ (the chemotaxis-Stokes system) and with $\kappa=1$ (the chemotaxis-Navier-Stokes system) were derived in \cite{duan-lorz-markowich2010} and \cite{liu-lorz2011}, respectively. Such weak solutions were later proved to be unique in \cite{zhang-zheng2014} by taking advantage of a coupling structure of the equations and using the Fourier localization technique. By demonstrating some blow-up criteria for classical solutions of the chemotaxis-Navier-Stokes system, Chae-Kang-Li showed that the global weak solutions derived in \cite{liu-lorz2011} is indeed a classical one upon improving the regularity of initial data (cf. \cite{chae-kang-li2012,chae-kang-li2014}). Relaxing the structural constraints on $\chi$ and $f$, global well-posedness on classical solutions were established in  \cite{chae-kang-li2014} under a smallness assumption on $\|c_0\|_{L^\infty}$ and in \cite{li-li2016} under some technique conditions on $\phi$. Comparing with the two-dimensional case, the problem of well-posedness in the case $\Om=\R^3$ seems to be more delicate, where the results available so far are merely confined to local and global small classical solutions for the chemotaxis-Navier-Stokes system, and global weak solutions on the chemotaxis-Stokes system (cf. \cite{chae-kang-li2012,chae-kang-li2014,
duan-lorz-markowich2010}).

In the case that $\Om$ is a bounded domain in $\R^d$, $d=2,3$ with smooth boundary, the study on well-posedness of \eqref{e00} subject to
 the following boundary conditions
\be\label{bc}
(D_m\na m-\chi(c)\na c)\cdot \vec{n}=0,\qquad \na c\cdot \vec{n}=0,\qquad \vec{u}=\mathbf{0},
\ee
with $\vec{n}$ the outward unit normal to the boundary $\partial \Om$,  was started by Lorz in \cite{lorz2010}, where local weak solutions were constructed in the situation $\chi$ being a constant and $f$ being monotonically increasing with $f(0)=0$. Under the structural hypotheses
 $(\frac{f(s)}{\chi(s)})^{'}>0$, $(\frac{f(s)}{\chi(s)})^{''}\leq 0$ and $(\chi(s)f(s))^{'}\geq 0$, Winkler established global existence of weak solutions in the 3D case for the chemotaxis-Stokes system and of smooth solutions in the 2D case for the chemotaxis-Navier-Stokes system (cf. \cite{winkler2012}). Those smooth solutions in the latter 2D case approach exponentially to the spatially homogeneous steady state $(\bar{m}_0,0,\mathbf{0})$ in the large time limit, where $\bar{m}_0=\frac{1}{|\Om|}\int_\Om m(x,0)dx$ (cf. \cite{winkler2014arma, zhang-li2015}). Global weak solutions for the three-dimensional chemotaxis-Navier-Stokes system were obtained in \cite{winkler2016} under the same structural requirements on $\chi$ and $f$ proposed in \cite{winkler2012}. Such weak solutions enjoy eventual smoothness and stabilize to the spatially uniform equilibria $(\bar{m}_0,0,\mathbf{0})$ as $t$ goes to infinity (cf. \cite{winkler2017}). 
 
 Besides the Neumann boundary conditions exhibited in \eqref{bc}, Dirichlet/Robin boundary conditions on oxygen have been imposed for system \eqref{e00} and the study on its well-posedness with such boundary conditions have been conducted in \cite{braukhoff-lankeit2019, braukhoff2017, wang-xiang2021, wx2022, wang-xiang2024}. Replacing the linear cell diffusion in \eqref{e00} with the nonlinear diffusion $\Delta m^\alpha$, $(\alpha>1)$, one derives the chemotaxis-Navier-Stokes driven by porous medium diffusion. On well-posedness of such systems we refer the reader to \cite{tao-winkler2013, jin2021, jin2024, tian-xiang2023, zheng-ke2022, winkler2018} and the reference therein.  
 \subsection{Goals and motivations} 
 %Based on the fact that the oxygen diffusion rate $\va$ is negligible compared to the bacterial diffusion rate (cf. \cite{KS71b}), we investigate the vanishing oxygen diffusion limit issue on system \eqref{e00} in this paper to uncover the asymptotic behavior of the solutions with respect to $\va$. 
 We emphasize that one of the most significant findings in the experiment conducted by Tuval et al. (cf. \cite{Tuval}) is the boundary layer formation of bacterial cells under the water surface and extensive studies on the boundary layer problem of various chemotaxis systems have been developed to uncover the underlying mechanism of this boundary layer formation. However, the boundary layer problem on the coupled chemotaxis-fluid system \eqref{e00} is lack of investigations. 
  %the analytic results on this topic available so far in the literature are confined to the chemotaxis-only subsystem, derived from \eqref{e00} by eliminating the incompressible Navier-Stokes equations. 
  The goal of the present paper is to make progress on this issue. Specifically, we investigate the boundary layer problem of \eqref{e00}-\eqref{bc} in the half plane  
 $\mathbb{R}^2_{+}=\{(x,y)\in\mathbb{R}^2:\,y>0\}$. In line with the experiment in \cite{Tuval}, we set $\chi(c)=1$, $f(c)=c$ and $\na \phi=(0,\lambda)$ with the gravity constant $\lambda$. The constants $D_m$ and $D$ are chosen to be $1$ without loss of generality. Then system \eqref{e00}-\eqref{bc} reads as: 
\begin{equation}\label{e1}
\left\{
\begin{array}{lll}
m_t+\vec{u}\cdot\na m+\nabla\cdot(m \na c)= \Delta m, \quad \quad\ \ \ & (x,y,t)\in \mathbb{R}^2_{+}\times(0,T),
\\
c_t+\vec{u}\cdot \na c+m c=\va\Delta c,\qquad \quad\qquad\ \ \, \ &(x,y,t)\in \mathbb{R}^2_{+}\times(0,T),\\
\vec{u}_t+ \vec{u}\cdot\na \vec{u}+\nabla p+m(0,\lambda)=\Delta\vec{u},\quad \ \ \ \ \quad & (x,y,t)\in \mathbb{R}^2_{+}\times(0,T),\\
\na\cdot\vec{u}=0,\qquad\qquad\qquad\quad \quad\qquad\ \ \ \  \, \  &(x,y,t)\in \mathbb{R}^2_{+}\times(0,T),\\
(m,c,\vec{u})(x,y,0)=(m_0,c_0,\vec{u}_0)(x,y), \quad &(x,y)\in \mathbb{R}^2_{+}
\end{array}
\right.
\end{equation}
and
\be\label{e2}
\left\{
\begin{array}{lll}
(\na m-m\na c)\cdot\vec{n}=0,\quad \na c\cdot\vec{n}=0, \quad \vec{u}=\mathbf{0}& \ \mathrm{if}\ \ \va>0,\\
(\na m-m\na c)\cdot\vec{n}=0,\quad \vec{u}=\mathbf{0}& \ \mathrm{if}\ \ \va=0,\\
\end{array}
\right.
\ee
where the initial data is imposed and no boundary condition is prescribed for $c$ in the case of $\va=0$, since its boundary value is intrinsically determined by the second equation in \eqref{e1}. From the boundary layer theory (cf. \cite{P}), we know that the inconsistent boundary conditions on $c$ between $\va>0$ and $\va=0$ in \eqref{e2} may induce to a thin layer near the boundary for small $\va>0$, in which the solution component $c$ changes rapidly and to study this boundary layer effect it suffices to investigate the vanishing oxygen diffusion limit issue on \eqref{e1}-\eqref{e2}. 

At the end of this section, we briefly recall the previous results on boundary layer problem of chemotaxis systems. The author and her collaborators showed that $\na c$, gradients of the second solution component to a chemotaxis system with logarithmic sensitivity possesses boundary layer effects in both one-dimensional and two-dimensional cases, under the circumstance that the bacterial cell and the oxygen concentration subject to Dirichlet and Neumann boundary conditions, respectively (cf. \cite{HWZ,HLWW,hou2019convergence}). Results in \cite{HWZ} were extended to the case with time-dependent boundary data (cf. \cite{peng-wang-zhao-zhu2018}). For the same chemotaxis system with no-flux boundary conditions on bacteria and Dirichlet boundary conditions on oxygen, Carrillo-Li-Wang derived the unique stationary boundary spike-layer steady state in the one-dimensional case and justified the asymptotically nonlinear stability of this steady state as $t$ goes to infinity (cf. \cite{carrillo-li-wang2021}). For the chemotaxis system with linear sensitivity, i.e. $\chi(c)=1$, stationary boundary layer solutions under Dirichlet boundary conditions on oxygen in arbitrary bounded domain of $\mathbb{R}^d$ have been constructed in \cite{lee-wang-yang2020} and the results were later extended to the solutions evolved with time in one-dimensional case(cf. \cite{carrillo-hong-wang2024}). Gradients of radially symmetric solutions under robin boundary conditions still possess boundary layer effects (cf. \cite{hou2024}). 

\section{Notation and main results}
\textbf{Notations.}
 \begin{itemize}
\item Without loss of generality, we assume $0\leq\va<1$ since we are concerned with the diffusion limit
     as $\va\rightarrow 0$. We denote by $C$ a generic constant that is independent of $\va$ but depending on $T$.
\item $\mathbb{N}_{+}$ represents the set of positive integers and $\mathbb{N}=\mathbb{N}_{+}\cup \{0\}$. For $z\in (0,\infty)$, we denote $\lge z\rge=\sqrt{z^2+1}$.
\item With $1\leq p\leq \infty$, we use $L^p_{xy}$ and $L^p_{xz}$ to denote the Lebesgue spaces $L^{p}(\mathbb{R}\times \mathbb{R}_{+})$ with respect to $(x,y)$ and $(x,z)$, respectively, with corresponding norms $\|\cdot\|_{L^p_{xy}}$ and $\|\cdot\|_{L^p_{xz}}$.
\item Similarly,
$H^k_{xy}$ and $H^k_{xz}$ for $k\in \mathbb{N}$ represent the Sobolev spaces $W^{k,2}(\mathbb{R}\times \mathbb{R}_+)$
%:=\{f(x,y)\in L^2(\mathbb{R}\times \mathbb{R}_+)\,|\, \underset{0\leq k_1+k_2\leq k}{\sum}
%\|\partial^{k_{1}}_x\partial^{k_2}_y f(x,y)\|_{L^2}<\infty
%\}
with respect to $(x,y)$ and $(x,z)$ respectively, with corresponding norms $\|\cdot\|_{H^k_{xy}}$ and $\|\cdot\|_{H^k_{xz}}$. Without confusion, we still use $H^k_{xy}$ and $L^p_{xy}$ to denote the two-dimensional vector spaces $(H^k_{xy})^2$ and $(L^p_{xy})^2$.
 \item For $k,m\in \mathbb{N}$, we introduce the anisotropic Sobolev space
$$H^k_xH^{m}_z:=\bigg\{f(x,z)\in L^2(\mathbb{R}\times \mathbb{R}_+)\,|\,\underset{ 0\leq l_1\leq k,\,0\leq l_2\leq m
     }{\sum}\|\partial^{l_1}_x\partial^{l_2}_z f(x,z)\|_{L^2_{xz}}<\infty\bigg\}$$
 with norm  $\|\cdot\|_{H^k_xH^m_z}$. Similarly  $H^k_xH^m_y$ will be used if the dependent variable of $f$ is $(x,y)\in \mathbb{R}\times \mathbb{R}_{+}$.
\item For simplicity, we use $\|\cdot\|_{L^q_T \boldsymbol{X}}$ $(1\leq q\leq \infty)$ to denote $\|\cdot\|_{L^q(0,T; \boldsymbol{X})}$ for Banach space $\boldsymbol{X}$.
\item For a function $f(x,y,t)\in C([0,T];H^1_{xy})$ with $(x,y,t)\in \mathbb{R}^2_{+}\times [0,T]$ and $T>0$, we denote $\ol{f}=f(x,0,t)$.
\end{itemize}

\subsection{Equations for boundary and outer layer profiles} Denote by $(m^\va,c^\va,\vec{u}^{\,\va},p^\va)$ the solutions of \eqref{e1}-\eqref{e2} with $\va>0$. To prove our main results, it is required to construct approximated solutions for $(m^\va,c^\va,\vec{u}^{\,\va},p^\va)$ with small $\va>0$. To this end, we employ a formal asymptotic analysis by assuming that $(m^\va,c^\va,\vec{u}^{\,\va},p^\va)$
possesses the following asymptotic expansions with respect to $\va$ for $j\in \mathbb{N}$:
\be\label{b2}
 \begin{split}
 m^{\va}(x,y,t)&=\sum_{j= 0}^\infty\va^{j/2}
 \left[ m^{I,j}(x,y,t)+m^{B,j}(x,z,t)
 \right],\\
 c^{\va}(x,y,t)&=\sum_{j= 0}^\infty\va^{j/2}
 \left[ c^{I,j}(x,y,t)+c^{B,j}(x,z,t)
 \right],\\
 \vec{u}^{\,\va}(x,y,t)&=\sum_{j=0}^\infty\va^{j/2}
 \left[ \vec{u}^{\,I,j}(x,y,t)+\vec{u}^{\,B,j}(x,z,t)
 \right],\\
 p^{\va}(x,y,t)&=\sum_{j= 0}^\infty\va^{j/2}
 \left[ p^{I,j}(x,y,t)+p^{B,j}(x,z,t)
 \right],
 \end{split}
 \ee
where $(x,y,t)\in \mathbb{R}^2_{+}\times (0,\infty)$ and the boundary layer coordinate is defined as:
\be\label{b1}
 z=\frac{y}{\va^{1/2}}, \quad  y \in (0,\infty).
\ee
Each term in \eqref{b2} is assumed to be smooth and the boundary layer profiles $(m^{B,j},c^{B,j},\vec{u}^{\,B,j},p^{B,j})$ enjoy the following basic hypothesis (cf. \cite[Chapter 4]{Holmes2012}, \cite{GG}, \cite{Rousset2005}):
\begin{quote}
\begin{enumerate}[{\bf H}:]
\item[(H)]
$m^{B,j}$, $c^{B,j}$, $\vec{u}^{\,B,j}$ and $p^{B,j}$ decay to zero exponentially as $z\rightarrow \infty$.
%(see \cite[Chapter 4]{Holmes2012}, \cite{GG}, \cite{Rousset2005}).
 \end{enumerate}
\end{quote}
 In order to obtain the initial-boundary value problems for outer layer profiles $(m^{I,j},c^{I,j},\vec{u}^{\,I,j},p^{I,j})$ and boundary layer profiles $(m^{B,j},c^{B,j},\vec{u}^{\,B,j},p^{B,j})$ with $j\geq 0$, the analysis will be split into five steps. In the first step, we deduce initial and boundary values for the outer and boundary layer profiles by inserting \eqref{b2} into the initial and boundary conditions in \eqref{e1}-\eqref{e2}. Equations on these layer profiles are derived in Step 2-Step 4 by substitutions of \eqref{b2} into each equation in \eqref{e1}. Collecting the results obtained in Step 1- Step 4, we obtain the following initial-boundary value problems in \eqref{e3}- \eqref{e15} and the definition of $\xi$ in \eqref{e18}. Detailed derivations on \eqref{e3}-\eqref{e18} are given in appendix. The leading-order outer layer profiles $(m^{I,0},c^{I,0},\vec{u}^{\,I,0},p^{I,0})(x,y,t)$ satisfy the following initial-boundary value problem:

\be\label{e3}
\left\{
\begin{array}{lll}
m^{I,0}_t+\vec{u}^{\,I,0}\cdot \na m^{I,0}+\na\cdot(m^{I,0}\na c^{I,0})=\Delta m^{I,0},\qquad (x,y,t)\in \mathbb{R}^2_{+}\times(0,\infty),\\
c^{I,0}_t+\vec{u}^{\,I,0}\cdot \na c^{I,0}+m^{I,0}c^{I,0}=0,\\
\vec{u}^{\,I,0}_t+\vec{u}^{\,I,0}\cdot \na \vec{u}^{\,I,0}+\na p^{I,0}+m^{I,0}(0,\lambda)
=\Delta \vec{u}^{\,I,0},\\
\na\cdot \vec{u}^{\,I,0}=0,\\
(m^{I,0},c^{I,0},\vec{u}^{\,I,0})(x,y,0)=(m_0,c_0,\vec{u}_0)(x,y),\\
(\pt_y m^{I,0}-m^{I,0}\pt_y c^{I,0})(x,0,t)=0,\quad \vec{u}^{\,I,0}(x,0,t)=\mathbf{0}.
\end{array}
\right.
\ee
Denote $(m^{0},c^{0},\vec{u}^{\,0},p^{0})(x,y,t)$ as the solution of \eqref{e1}-\eqref{e2} with $\va=0$. Then from the uniqueness of solutions, we deduce that 
\ben
(m^{0},c^{0},\vec{u}^{\,0},p^{0})(x,y,t)=(m^{I,0},c^{I,0},\vec{u}^{\,I,0},p^{I,0})(x,y,t).
\enn
The leading-order boundary layer profiles $(m^{B,0},c^{B,0},\vec{u}^{\,B,0},p^{B,0})(x,z,t)$ with $(x,z,t)\in \mathbb{R}^2_{+}\times (0,\infty)$ satisfy:
\be\label{e4}
m^{B,0}(x,z,t)=c^{B,0}(x,z,t)=p^{B,0}(x,z,t)=0,\quad \vec{u}^{\,B,0}(x,z,t)=\mathbf{0}.
\ee
The first-order outer layer profiles $(m^{I,1},c^{I,1},\vec{u}^{\,I,1},p^{I,1})(x,y,t)$ solve:
\be\label{e5}
\left\{
\begin{array}{lll}
m^{I,1}_t+\vec{u}^{\,0}\cdot \na m^{I,1}+\vec{u}^{\,I,1}\cdot \na m^{0}+\na\cdot(m^{0}\na c^{I,1}+m^{I,1}\na c^{0})
=\Delta m^{I,1},\\
%\quad (x,y,t)\in \mathbb{R}^2_{+}\times(0,\infty),\\
c^{I,1}_t+\vec{u}^{\,0}\cdot \na c^{I,1}+\vec{u}^{\,I,1}\cdot \na c^{0}+m^{0}c^{I,1}
+m^{I,1}c^{0}=0,\\
\vec{u}^{\,I,1}_t+\vec{u}^{\,0}\cdot \na \vec{u}^{\,I,1}
+\vec{u}^{\,I,1}\cdot \na \vec{u}^{\,0}+\na p^{I,1}+m^{I,1}(0,\lambda)
=\Delta \vec{u}^{\,I,1},\\
\na\cdot \vec{u}^{\,I,1}=0,\\
(m^{I,1},c^{I,1},\vec{u}^{\,I,1})(x,y,0)=(0,0,\mathbf{0}),\\
(\pt_y m^{I,1}-m^{I,1}\pt_y c^{0}-m^{0}\pt_y c^{I,1})(x,0,t)=0,\quad \vec{u}^{\,I,1}(x,0,t)=\mathbf{0},
\end{array}
\right.
\ee
which, gives rise to
\be\label{e6}
(m^{I,1},c^{I,1},\vec{u}^{\,I,1},p^{I,1})(x,y,t)=(0,0,\mathbf{0},0)
\ee
thanks to the uniqueness of solutions.
The first-order boundary layer profiles fulfill:
\be\label{e9}
\vec{u}^{\,B,1}(x,z,t)=\mathbf{0},\qquad p^{B,1}(x,z,t)=0
\ee
and
\be\label{e7}
\left\{
\begin{array}{lll}
c^{B,1}_t+z \ol{\partial_y u^{0}_2} \partial_z c^{B,1}+\ol{m^{0}}[\ol{c^{0}}+1]c^{B,1}
=\partial_z^2 c^{B,1},\qquad (x,z,t)\in \mathbb{R}^2_{+}\times (0,\infty),\\
c^{B,1}(x,z,0)=0,\\
\partial_z c^{B,1}(x,0,t)=-\ol{\partial_y c^{0}}
\end{array}
\right.
\ee
and
\be\label{e8}
m^{B,1}(x,z,t)=\ol{m^{0}}c^{B,1}(x,z,t).
\ee
%The second-order outer layer profile satisfies
%\be\label{e10}
%\left\{
%\begin{array}{lll}
%m^{I,2}_t+\vec{u}^{\,I,0}\cdot \na m^{I,2}+\vec{u}^{\,I,2}\cdot \na m^{I,0}
%+\na\cdot(m^{I,0}\na c^{I,2})+\na\cdot(m^{I,2}\na c^{I,0})=\Delta m^{I,2},\qquad (x,y,t)\in \mathbb{R}^2_{+}\times(0,T),\\
%c^{I,2}_t+\vec{u}^{\,I,0}\cdot \na c^{I,2}+\vec{u}^{\,I,2}\cdot \na c^{I,0}+m^{I,0}c^{I,2}
%+m^{I,2}c^{I,0}=\Delta c^{I,0},\\
%\vec{u}^{\,I,2}_t+\vec{u}^{\,I,0}\cdot \na \vec{u}^{\,I,2}
%+\vec{u}^{\,I,2}\cdot \na \vec{u}^{\,I,0}+\na p^{I,2}+m^{I,2}(0,\lambda)
%=\Delta \vec{u}^{\,I,2},\\
%\na\cdot \vec{u}^{\,I,2}=0,\\
%(m^{I,2},c^{I,2},\vec{u}^{\,I,2})(x,y,0)=(0,0,\mathbf{0}),\\
%(\pt_y m^{I,2}-m^{I,2}\pt_y c^{I,0}-m^{I,0}\pt_y c^{I,2})(x,0,t)=0,\quad \vec{u}^{\,I,2}(x,0,t)=\mathbf{0}.
%\end{array}
%\right.
%\ee
The second-order boundary layer profiles satisfy:
\be\label{e11}
\vec{u}^{\,B,2}(x,z,t)=\mathbf{0},\qquad p^{B,2}(x,z,t)=\lambda \int_z^\infty m^{B,1}(x,\eta,t)d\eta.
\ee
and
\be\label{e12}
\left\{
\begin{array}{lll}
c^{B,2}_t\!+\!z \ol{\partial_y u^{0}_2}\partial_z c^{B,2}
 \!+\!\ol{m^{0}}[\ol{c^{0}}+1]c^{B,2}
=\partial^2_z c^{B,2}\!+\!\Gamma(x,z,t)\!-\!\ol{c^{0}}\Theta(x,z,t),\\
%\quad (x,z,t)\in \mathbb{R}^2_{+}\times (0,\infty),\\
c^{B,2}(x,z,0)=0,\\
\partial_z c^{B,2}(x,0,t)=0
\end{array}
\right.
\ee
and
\be\label{e13}
\begin{split}
m^{B,2}(x,z,t)=\ol{m^{0}}c^{B,2}(x,z,t)
-\Theta(x,z,t)
\end{split}
\ee
with
\be\label{e14}
\begin{split}
\Gamma(x,z,t):=&-
z\ol{\partial_y u^{0}_1}\partial_x c^{B,1}-\frac{z^2}{2}\ol{\partial^2_y u^{0}_2}\partial_z c^{B,1}
-z \ol{\partial_y m^{0}}c^{B,1}
-z m^{B,1}\ol{\partial_y c^{0}}
-m^{B,1}c^{B,1}
\end{split}
\ee
and
\be\label{e15}
\begin{split}
\Theta(x,z,t):=&\ol{\partial_y c^{0}}\int_z^\infty m^{B,1}(x,\eta,t)d\eta
+\int_z^\infty (m^{B,1}\partial_\eta c^{B,1})(x,\eta,t)d\eta\\
&+\ol{\partial_y m^{0}}\int_z^\infty \eta \partial_\eta c^{B,1}(x,\eta,t)d\eta.
\end{split}
\ee
%The third-order boundary layer profile $\vec{u}^{\,B,3}$ satisfies:
%\be\label{e16}
%\vec{u}^{\,B,3}(x,z,t)=\mathbf{0}.
%\ee
Denote
\be\label{e18}
\begin{split}
\xi(x,z,t)=&-\int_z^\infty \eta \ol{\partial_y m^{0}}\,\partial_\eta c^{B,2}(x,\eta,t)d\eta
-\int_z^\infty \ol{\partial_y c^{0}}\, m^{B,2}(x,\eta,t)d\eta\\
&-\int_z^\infty \frac{\eta^2}{2} \ol{\partial_y^2 m^{0}}\,\partial_\eta c^{B,1}(x,\eta,t)d\eta
-\int_z^\infty \int_\eta^\infty \partial^2_x m^{B,1}(x,s,t)ds\, d\eta\\
&+\int_z^\infty \int_\eta^\infty m^{B,1}_t(x,s,t)ds \,d\eta
+\int_z^\infty \int_\eta^\infty s \ol{\partial_y u^{0}_2}\,\partial_s m^{B,1}(x,s,t)ds \,d\eta\\
&+\int_z^\infty \int_\eta^\infty \partial_x[\ol{m^{0}}\partial_x c^{B,1}(x,s,t)
+m^{B,1}(x,s,t)\ol{\partial_x c^{0}}
]ds \,d\eta\\
&+\int_z^\infty \int_\eta^\infty [m^{B,1}(x,s,t)+s\partial_s m^{B,1}(x,s,t)]
\ol{\partial^2_y c^{0}}ds\, d\eta\\
&-\int_z^\infty (m^{B,1}\partial_\eta c^{B,2}+m^{B,2}\partial_\eta c^{B,1})(x,\eta,t)d\eta\\
&+\int_z^\infty \int_\eta^\infty s\partial_s c^{B,1}(x,s,t)\,
\ol{\partial^2_y m^{0}}ds \,d\eta.
\end{split}
\ee

\subsection{Main results}
To attain the regularity for the solutions presented in Proposition \ref{p2}, it is natural to impose on initial data the following compatibility conditions:
\begin{equation*}
(A)\left\{
\begin{array}{lll}
0=\partial_y c_{0}(x,0)=\partial_y m_{0}(x,0),\\
\mathbf{0}=[\Delta \vec{u}_0-\na p_0-m_0(0,\lambda)](x,0)=\vec{u}_0(x,0),\\
0=[\pt_y(\vec{u}_0\cdot\na c_0+m_0c_0)](x,0),\\
0=[\pt_y(\Delta m_0-\vec{u}_0\cdot\na m_0-\na\cdot(m_0\na c_0))](x,0),\\
0=(\pt_y \vec{u}^{\,0}_t\cdot\na c_0+\pt_y \vec{u}_{0}\cdot\na c^0_t)(x,0,0),\\
0=[\pt_y(\vec{u}^{\,0}_t\cdot\na m_0+\vec{u}_0\cdot\na m^0_t+\na (m^0_t\na c_0+m_0\na c^0_t))](x,0,0),\\
\mathbf{0}=[\Delta \vec{u}^{\,0}_t-\na p_{t0}-m^0_t(0,\lambda)](x,0,0),
\end{array}
\right.
\end{equation*}
where $p_0(x,y)$ solves
\begin{equation*}
\left\{\begin{array}{lll}
\Delta p_0=-\na\cdot[\vec{u}_0\cdot\na \vec{u}_0+m_0(0,\lambda)],\quad (x,y)\in \mathbb{R}^2_{+},\\
\pt_y p_0(x,0)=\pt_y^2u_{02}(x,0)-m_0(x,0)
\end{array}
\right.
\end{equation*}
and
\begin{equation*}
\begin{array}{lll}
\vec{u}^{\,0}_t(x,y,0)=[\Delta \vec{u}_0-\vec{u}_0\cdot\na \vec{u}_0-\na p_0-m_0(0,\lambda)](x,y),\\
c^0_t(x,y,0)=-[\vec{u}_{0}\cdot\na c_0+m_0c_0](x,y),\\
m^0_t(x,y,0)=[\Delta m_0-\vec{u}_0\cdot\na m_0-\na\cdot(m_0\na c_0)](x,y)
\end{array}
\end{equation*}
and $p_{t0}(x,y)$ solves
\begin{equation*}
\left\{\begin{array}{lll}
\Delta p_{t0}=-\na\cdot[\vec{u}^{\,0}_t\cdot\na \vec{u}_0+\vec{u}_0\cdot\na \vec{u}^{\,0}_t+m^0_t(0,\lambda)](x,y,0),\quad (x,y)\in \mathbb{R}^2_{+}\\
\pt_y p_{t0}(x,0)=\pt_y^2u^0_{2t}(x,0,0)-m^0_t(x,0,0).
\end{array}
\right.
\end{equation*}

%\begin{prop}\label{p1}
%Assume that $m_0\in H^2$, $c_0\in H^3$ and $\vec{u}_{0}\in H^3$ satisfy the compatibility conditions in the first two lines of $(A)$. Then for any $\va>0$, there exists a $0<T_\va\leq \infty$, the maximal time of existence and a unique solution $(m^\va,c^\va,\vec{u}^{\,\va},\na p^\va)$ of \eqref{e1}-\eqref{e2} satisfying for any $T<T_\va$
%\ben
%\begin{split}
%&(m^\va,c^\va,\vec{u}^{\,\va},\na p^\va)\in C([0,T];H^2_{xy}\times H^3_{xy}\times H^3_{xy}\times H^1_{xy})\\
%&(m^\va,c^\va,\vec{u}^{\,\va},\na p^\va)\in L^2(0,T;H^3_{xy}\times H^4_{xy}\times H^4_{xy}\times H^2_{xy})\\
%&\qquad(m^\va_t,c^\va_t,\vec{u}^{\,\va}_t)\in C([0,T];L^2_{xy}\times H^1_{xy}\times H^1_{xy}).
%\end{split}
%\enn
%Furthermore, if $T_\va<\infty$, then 
%\ben
%\begin{split}
%\limsup_{t\rightarrow T_\va}&(\|m^\va(t)\|_{H^2_{xy}}+\|c^\va(t)\|_{H^3_{xy}}+\|\vec{u}^{\,\va}(t)\|_{H^3_{xy}})\\
%+&\|m^\va\|_{L^2(0,T_\va;H^3_{xy})}+\|c^\va\|_{L^2(0,T_\va;H^4_{xy})}+\|\vec{u}^{\,\va}\|_{L^2(0,T_\va;H^4_{xy})}
%=\infty.
%\end{split}
%\enn
%\end{prop}

The following regularity results on solutions of \eqref{e1}-\eqref{e2} with $\va=0$ can be proved by using the well-posedness theory on parabolic equations along with Shauder's fixed point theorem. We omit its proof and refer the reader to \cite[pages 388, 540]{Evans} for details.
\begin{prop}\label{p2}
Assume that $m_0\in H^6_{xy}$, $c_0, \vec{u}_{0}\in H^7_{xy}$ fulfills the compatibility conditions $(A)$. Then \eqref{e1}-\eqref{e2} with $\va=0$ admits a unique solution $(m^0,c^0,\vec{u}^{\,0},\na p^0)$, whose maximal time of existence is denoted by $0<T_*\leq \infty$,  satisfying for any $0<T<T_*$
\ben
\begin{split}
&(m^0,c^0,\vec{u}^{\,0},\na p^0)\in C([0,T];H^6_{xy}\times H^7_{xy}\times H^7_{xy}\times H^5_{xy}),\\
&(m^0,c^0,\vec{u}^{\,0},\na p^0)\in L^2(0,T;H^7_{xy}\times H^7_{xy}\times H^8_{xy}\times H^6_{xy})
\end{split}
\enn
and
\ben
\begin{split}
&\pt^j_t m^{0}\in C([0,T];H^{6-2j}_{xy})\cap L^2(0,T;H^{7-2j}_{xy}),\quad j=1,2,3,\\
&\pt^j_t \vec{u}^{\,0}\in C([0,T];H^{7-2j}_{xy})\cap L^2(0,T;H^{8-2j}_{xy}),\quad j=1,2,3,\\
&\pt_t^j c^0\in C([0,T];H^{7-l}_{xy}),\quad l=1,2,\qquad \pt^3_t c^0\in L^2(0,T;H^4_{xy}).\\
\end{split}
\enn
\end{prop}

With proposition \ref{p2} in hand, we are now in the position to state our main results on boundary layer effects of the solutions with small $\va>0$.
\begin{theorem}\label{t1}
Assume that $m_0\in H^6_{xy}$, $c_0, \vec{u}_{0}\in H^7_{xy}$ satisfy compatibility conditions $(A)$. Let $(m^0,c^0,\vec{u}^{\,0},\na p^0)$ be the solution of \eqref{e1}-\eqref{e2} with $\va=0$, whose lifespan is $T_*$. Then for each $0<T<T_*$, there exists a $\va_T>0$ depending on $T$ such that system \eqref{e1}-\eqref{e2} with $\va\in(0,\va_T]$ admits a unique solution 
\ben
(m^\va,c^\va,\vec{u}^{\,\va},\na p^\va)\in C([0,T];H^2_{xy}\times H^3_{xy}\times H^3_{xy}\times H^1_{xy}).
\enn
 Moreover, there exits a constant $C$ independent of $\va$, depending on $T$ such that 
\begin{equation*}
\begin{split}
\|(m^\va-m^0,c^\va-c^0)(x,y,t)\|_{L^\infty([0,T];L^\infty_{xy})}
\leq C\va^{\frac{1}{2}},\quad
\|(\vec{u}^{\,\va}-\vec{u}^{\,0})(x,y,t)\|_{L^\infty([0,T];L^\infty_{xy})}
\leq C\va
\end{split}
\end{equation*}
and 
\begin{equation}\label{a0}
\begin{split}
&\|\pt_xc^\va(x,y,t)-\pt_xc^0(x,y,t)\|_{L^\infty([0,T];L^\infty_{xy})}
\leq C\va^{\frac{1}{8}},\\
&\|\pt_yc^\va(x,y,t)-[\pt_yc^0(x,y,t)+\pt_zc^{B,1}(x,\frac{y}{\sqrt{\va}},t)]\|_{L^\infty([0,T];L^\infty_{xy})}
\leq C\va^{\frac{1}{8}},\\
&\|\na \vec{u}^{\,\va}(x,y,t)-\na \vec{u}^{\,0}(x,y,t)\|_{L^\infty([0,T];L^\infty_{xy})}
\leq C\va^{\frac{3}{4}},
\end{split}
\end{equation}
where $c^{B,1}(x,z,t)$ is the unique solution of \eqref{e7}, derived in Lemma \ref{l10}. 
\end{theorem}

\begin{rem} From \eqref{a0} we know that $\pt_y c$ possesses boundary layer effects and it is natural to expect 
$\pt_y m$ also possesses such boundary layer effects due to the chemotactic interaction between $m$ and $c$. Indeed, one can prove that 
\begin{equation*}
\|\pt_ym^\va(x,y,t)-[\pt_ym^0(x,y,t)+\pt_zm^{B,1}(x,\frac{y}{\sqrt{\va}},t)]\|_{L^\infty([0,T];L^\infty_{xy})}
\leq C\va^{\frac{1}{4}},
\end{equation*} 
upon improving the regularity on initial data, e.g. $m_0\in H^8_{xy}$, $c_0, \vec{u}_{0}\in H^9_{xy}$ and requiring further compatibility conditions. Here the $m^{B,1}$ is given in \eqref{e8},
\end{rem}

\begin{rem} 
Neglecting the influence of fluids, i.e. letting $\vec{u}=\na p=\mathbf{0}$ in \eqref{e1}-\eqref{e2} we derive a chemotaxis-only subsystem. From the second equation of this chemotaxis system with $\va=0$ and the first line of $(A)$, we have
\begin{equation*}
\pt_y c^0(x,0,t)=\pt_yc_0(x,0)e^{-\int_0^t[m^0(c^0+1)](x,0,\tau)d\tau}=0,
\end{equation*}
which, substituted into \eqref{e7} gives rise to $\pt_z c^{B,1}(x,0,t)=0$ and thus 
$c^{B,1}(x,z,t)\equiv0$, thanks to the uniqueness of solutions. Inserting $c^{B,1}(x,z,t)\equiv0$ into the second inequality of \eqref{a0}, one gets
\be\label{a00}
\|\pt_yc^\va(x,y,t)-\pt_yc^0(x,y,t)\|_{L^\infty([0,T];L^\infty_{xy})}
\leq C\va^{\frac{1}{8}},
\ee
which, indicates that $\pt_y c$ no longer possess boundary layer effects in the chemotaxis-only subsystem setting. Comparing \eqref{a0} to \eqref{a00} we conclude that the boundary layer effect on $\pt_y c$ for the chemotaxis-Navier-Stokes system \eqref{e1} under boundary conditions \eqref{e2} is induced by the presence of fluids. 
\end{rem}
The remaining parts of this paper is organized as follows. In Section 3, we show the well-posedness on system \eqref{e7}-\eqref{e18}, and derive certain regularities on their solutions. Section 4 is devoted to proving Theorem \ref{t1}. In Subsection 4.1, we first construct approximation solutions for $(m^\va,c^\va,\vec{u}^{\,\va},\na p^\va)$ and then demonstrate well-posedness on the initial-boundary value problem of the remainders between $(m^\va,c^\va,\vec{u}^{\,\va},\na p^\va)$ and the approximation solutions. Based on these well-posedness results, we prove Theorem \ref{t1} in Subsection 4.2.  

\section{Estimates on boundary layer profiles}
To assert the well-posedness on solutions of \eqref{e7} and \eqref{e12}, we introduce the following auxiliary initial-boundary value problem
\be\label{e25}
\left\{
\begin{array}{lll}
\varphi_t+z \ol{\partial_y u^{0}_2}\partial_z \varphi
 +\ol{m^{0}}(\ol{c^{0}}+1)\varphi
=\partial^2_z \varphi+\rho,\qquad (x,z,t)\in \mathbb{R}^2_{+}\times(0,\infty),\\
\ \varphi(x,z,0)=0,\\
\ \partial_z \varphi(x,0,t)=0.
\end{array}
\right.
\ee
For system \eqref{e25}, we have the following result.
\begin{lemma}\label{l9}
Let $k_0,\, k_1,\, k_2\in\mathbb{N}_{+}$ and $0<T<\infty$. Suppose
\be\label{a13}
\begin{split}
&\pt_t^j\ol{\partial_y u^{0}_2},\,\, \pt_t^j[\ol{m^{0}}(\ol{c^{0}}+1)] \in L^2(0,T;H^{k_j}_{x}),\qquad
\lge z\rge^{l}\pt^j_t\rho\in L^2(0,T;H^{k_j}_{x}L^2_z)
\end{split}
\ee
for $j=0,1$ and each $l\in\mathbb{N}$.
Then \eqref{e25} admits a unique solution $\varphi$ defined on $\mathbb{R}^2_{+}\times [0,T]$ fulfilling
\be\label{a14}
\begin{split}
&\lge z\rge^{l}\pt^j_t\varphi\in C([0,T];H^{k_j}_{x}H^1_z),\qquad \lge z\rge^{l}\pt^{j+1}_t\varphi\in L^2(0,T;H^{k_j}_{x}L^2_z)
\end{split}
\ee
for $j=0,1$ and each $l\in\mathbb{N}$. Assume further that 
\be\label{a15}
\begin{split}
&\pt^2_t \ol{\partial_y u^{0}_2},\,\, \pt^2_t[\ol{m^{0}}(\ol{c^{0}}+1)] \in L^2(0,T;H^{k_2}_{x}),\qquad
\lge z\rge^{l}\pt^2_t\rho\in L^2(0,T;H^{k_2}_xL^2_{z})
\end{split}
\ee
for each $l\in\mathbb{N}$. Then 
\be\label{a16}
\begin{split}
&\lge z\rge^{l}\pt^2_t\varphi\in C([0,T];H^{k_2}_{x}H^1_z),\qquad \lge z\rge^{l}\pt^{3}_t\varphi\in L^2(0,T;H^{k_2}_xL^2_{z})
\end{split}
\ee
for each $l\in\mathbb{N}$.
\end{lemma}
\begin{proof}
For fixed $l\in\mathbb{N}$, testing the first equation of \eqref{e25} with $\lge z\rge^{2l} \varphi$ and using integration by parts, one gets
\be\label{a70}
\begin{split}
&\frac{1}{2}\frac{d}{dt}\|\lge z\rge^l \varphi\|_{L^2_{xz}}^2+\|\lge z\rge^l\pt_z\varphi\|_{L^2_{xz}}^2\\
=&-2l\int_0^\infty\int_{-\infty}^\infty \lge z\rge^{2l-2} z \,\varphi\,\pt_z\varphi dxdy
+\int_0^\infty\int_{-\infty}^\infty \lge z\rge^{2l}  \rho\,\varphi dxdy\\
&-\int_0^\infty\int_{-\infty}^\infty \lge z\rge^{2l} z[\ol{\partial_y u^{0}_2}\pt_z\varphi]\,\varphi dxdy
-\int_0^\infty\int_{-\infty}^\infty \lge z\rge^{2l}  [\ol{m^{0}}(\ol{c^{0}}+1)\varphi]\,\varphi dxdy
,
\end{split}
\ee
where 
\ben
\begin{split}
&-2l\int_0^\infty\int_{-\infty}^\infty \lge z\rge^{2l-2} z \,\varphi\,\pt_z\varphi dxdy
+\int_0^\infty\int_{-\infty}^\infty \lge z\rge^{2l}  \rho\,\varphi dxdy\\
\leq& \frac{1}{2}\|\lge z\rge^l\pt_z\varphi\|_{L^2_{xz}}^2+(l^2+1)\|\lge z\rge^l\varphi\|_{L^2_{xz}}^2
+\|\lge z\rge^l\rho\|_{L^2_{xz}}^2.
\end{split}
\enn
It follows from integration by parts and the Sobolev embedding inequality that
\ben
\begin{split}
&-\int_0^\infty\int_{-\infty}^\infty \lge z\rge^{2l} z[\ol{\partial_y u^{0}_2}\pt_z\varphi]\,\varphi dxdy
-\int_0^\infty\int_{-\infty}^\infty \lge z\rge^{2l}  [\ol{m^{0}}(\ol{c^{0}}+1)\varphi]\,\varphi dxdy\\
=&l\int_0^\infty\int_{-\infty}^\infty \lge z\rge^{2l-2} z^2\ol{\partial_y u^{0}_2}\,\varphi^2 dxdy
+\frac{1}{2}\int_0^\infty\int_{-\infty}^\infty \lge z\rge^{2l}\ol{\partial_y u^{0}_2}\,\varphi^2 dxdy\\
&-\int_0^\infty\int_{-\infty}^\infty \lge z\rge^{2l}  [\ol{m^{0}}(\ol{c^{0}}+1)\varphi]\,\varphi dxdy\\
\leq & (l+1)\|\ol{\partial_y u^{0}_2}\|_{L^\infty_x}\|\lge z\rge^l\varphi\|_{L^2_{xz}}^2
+\|\ol{m^{0}}(\ol{c^{0}}+1)\|_{L^\infty_x}\|\lge z\rge^l\varphi\|_{L^2_{xz}}^2\\
\leq & C(\|\ol{\partial_y u^{0}_2}\|_{H^1_x}+\|\ol{m^{0}}(\ol{c^{0}}+1)\|_{H^1_x})\|\lge z\rge^l\varphi\|_{L^2_{xz}}^2.
\end{split}
\enn
Substituting the above two estimates into \eqref{a70}, using Gronwall's inequality and \eqref{a13}, we obtain
\be\label{a71}
\begin{split}
\|\lge z\rge^l \varphi\|_{L^\infty_TL^2_{xz}}^2+\|\lge z\rge^l\pt_z\varphi\|_{L^2_TL^2_{xz}}^2\leq C
\end{split}
\ee
for each $l\in\mathbb{N}$.

Based on \eqref{a71}, we next prove that
\be\label{b38}
\begin{split}
\|\lge z\rge^l \varphi\|_{L^\infty_TH^{k_0}_xL^2_{z}}^2+\|\lge z\rge^l \pt_z\varphi\|_{L^2_TH^{k_0}_xL^2_{z}}^2\leq C
\end{split}
\ee
holds true for each $l\in\mathbb{N}$, by the argument of induction. To this end, we assume that
\be\label{a72}
\begin{split}
\|\lge z\rge^l\varphi\|_{L^\infty_TH^{j}_xL^2_{z}}^2+\|\lge z\rge^l \pt_z\varphi\|_{L^2_TH^{j}_xL^2_{z}}^2\leq C
\end{split}
\ee
holds true for each $0\leq j\leq k_0-1$ and $l\in\mathbb{N}$, and show that \eqref{a72} holds true for $j=k_0$. 
Indeed, we apply $\pt^{k_0}_x$ to the first equation of \eqref{e25} and then take the $L^2_{xz}$ inner product of the resulting equation with $\lge z\rge^{2l} \pt^{k_0}_x\varphi$ for $l\in \mathbb{N}$ to have
\be\label{b39}
\begin{split}
&\frac{1}{2}\frac{d}{dt}\|\lge z\rge^l \pt^{k_0}_x\varphi\|_{L^2_{xz}}^2+\|\lge z\rge^l \pt^{k_0}_x\pt_z\varphi\|_{L^2_{xz}}^2\\
=&-2l\int_0^\infty\int_{-\infty}^\infty \lge z\rge^{2l-2} z \pt^{k_0}_x\varphi\,\pt^{k_0}_x\pt_z\varphi dxdy
-\int_0^\infty\int_{-\infty}^\infty \lge z\rge^{2l} z\pt^{k_0}_x [\ol{\partial_y u^{0}_2}\pt_z\varphi]\,\pt^{k_0}_x\varphi dxdy\\
&-\int_0^\infty\int_{-\infty}^\infty \lge z\rge^{2l} \pt^{k_0}_x [\ol{m^{0}}(\ol{c^{0}}+1)\varphi]\,\pt^{k_0}_x\varphi dxdy
+\int_0^\infty\int_{-\infty}^\infty \lge z\rge^{2l} \pt^{k_0}_x \rho\,\pt^{k_0}_x\varphi dxdy\\
:=&I_1+I_2+I_3+I_4.
\end{split}
\ee
It follows from the Cauchy-Schwarz inequality that
\ben
\begin{split}
I_1
\leq 2l\|\lge z\rge^{l}\pt^{k_0}_x\pt_z\varphi\|_{L^2_{xz}}\|\lge z\rge^{l}\pt^{k_0}_x\varphi\|_{L^2_{xz}}
\leq \frac{1}{8}\|\lge z\rge^{l}\pt^{k_0}_x\pt_z\varphi\|_{L^2_{xz}}^2+8l^2\|\lge z\rge^{l}\pt^{k_0}_x\varphi\|_{L^2_{xz}}^2
\end{split}
\enn
and that
\ben
\begin{split}
I_4
\leq \|\lge z\rge^{l}\pt^{k_0}_x\varphi\|_{L^2_{xz}}^2+\|\lge z\rge^{l}\pt^{k_0}_x\rho\|_{L^2_{xz}}^2.
\end{split}
\enn
Integration by parts and the Sobolev embedding inequality lead to
\ben
\begin{split}
I_2=
&l\int_0^\infty\int_{-\infty}^\infty \ol{\partial_y u^{0}_2}\,\lge z\rge^{2l-2}z^2(\pt^{{k_0}}_x\varphi)^2 dxdy
+\frac{1}{2} \int_0^\infty\int_{-\infty}^\infty \lge z\rge^{2l} \ol{\partial_y u^{0}_2}\,(\pt^{{k_0}}_x\varphi)^2 dxdy\\
&-\int_0^\infty\int_{-\infty}^\infty \pt^{k_0}_x\ol{\partial_y u^{0}_2}\,\lge z\rge^{2l}z\pt_z\varphi\,\pt^{{k_0}}_x\varphi dxdy
-\sum_{i=1}^{{k_0}-1} \int_0^\infty\int_{-\infty}^\infty \lge z\rge^{2l} z\,(\pt^i_x \ol{\partial_y u^{0}_2})\,(\pt^{k_0-i}_x\pt_z\varphi)\,\pt^{k_0}_x\varphi dxdy\\
\leq & (l+\frac{1}{2})\|\ol{\partial_y u^{0}_2}\|_{L^\infty_x}\|\lge z\rge^{l}\pt^{{k_0}}_x\varphi\|_{L^2_{xz}}^2
+\|\pt^{k_0}_x\ol{\partial_y u^{0}_2}\|_{L^2_x}\|\lge z\rge^{l+1}\pt_z\varphi\|_{L^\infty_xL^2_{z}}\|\lge z\rge^{l}\pt^{{k_0}}_x\varphi\|_{L^2_{xz}}\\
&+\sum_{i=1}^{{k_0}-1}\|\pt^i_x \ol{\partial_y u^{I,0}_2}\|_{L^\infty_x}\|\lge z\rge^{l+1}\pt^{{k_0}-i}_x\pt_z\varphi\|_{L^2_{xz}}
\|\lge z\rge^1\pt^{{k_0}}_x\varphi\|_{L^2_{xz}}
\\
\leq &C(l+\frac{1}{2})\|\ol{\partial_y u^{0}_2}\|_{H^1_x}\|\lge z\rge^{l}\pt^{{k_0}}_x\varphi\|_{L^2_{xz}}^2
+\frac{1}{2}\|\lge z\rge^{l+1}\pt_z\varphi\|_{H^1_xL^2_{z}}^2\\
&+C\|\ol{\partial_y u^{0}_2}\|_{H^{k_0}_x}^2\|\lge z\rge^{l}\pt^{{k_0}}_x\varphi\|_{L^2_{xz}}^2
+C({k_0}-1)^2\|\lge z\rge^{l+1}\pt_z\varphi\|_{H^{{k_0}-1}_xL^2_{z}}^2.
\end{split}
\enn 
The Sobolev embedding inequality entails that
\ben
\begin{split}
I_3=&-\sum_{i=1}^{{k_0}-1} \int_0^\infty\int_{-\infty}^\infty \lge z\rge^{2l} \,\pt^i_x [\ol{m^{0}}(\ol{c^{0}}+1)]\,(\pt^{{k_0}-i}_x\varphi)\,\pt^{k_0}_x\varphi dxdy\\
&-\int_0^\infty\int_{-\infty}^\infty \lge z\rge^{2l}\,[\ol{m^{0}}(\ol{c^{0}}+1)]\,(\pt^{{k_0}}_x\varphi)^2 dxdy
-\int_0^\infty\int_{-\infty}^\infty \lge z\rge^{2l}\,\pt_x^{k_0}[\ol{m^{0}}(\ol{c^{0}}+1)]\,\varphi\,\pt^{{k_0}}_x\varphi dxdy\\
\leq &\sum_{i=1}^{{k_0}-1}\|\pt^i_x [\ol{m^{0}}(\ol{c^{0}}+1)]\|_{L^\infty_x}\|\lge z\rge^{l}\,\pt^{{k_0}-i}_x\varphi\|_{L^2_{xz}}
\|\lge z\rge^{l}\,\pt^{{k_0}}_x\varphi\|_{L^2_{xz}}\\
&+\|\ol{m^{0}}(\ol{c^{0}}+1)\|_{L^\infty_x}
\|\lge z\rge^{l}\,\pt^{{k_0}}_x\varphi\|_{L^2_{xz}}^2
+\|\pt^{k_0}_x\ol{m^{0}}(\ol{c^{0}}+1)\|_{L^2_x}
\|\lge z\rge^{l}\,\varphi\|_{L^\infty_xL^2_{z}}\|\lge z\rge^{l}\,\pt^{{k_0}}_x\varphi\|_{L^2_{xz}}\\
\leq &({k_0}-1)^2
\|\lge z\rge^{l}\,\varphi\|_{H^{{k_0}-1}_xL^2_{z}}^2 
+C\| [\ol{m^{0}}(\ol{c^{0}}+1)]\|_{H^{k_0}_x}^2
\|\lge z\rge^{l}\,\pt_x^{k_0}\varphi\|_{L^2_{xz}}^2
+\|\lge z\rge^{l}\,\varphi\|_{H^{1}_xL^2_{z}}^2 \\ 
&+C\| [\ol{m^{0}}(\ol{c^{0}}+1)]\|_{H^1_x}
\|\lge z\rge^{l}\,\pt_x^{k_0}\varphi\|_{L^2_{xz}}^2. 
\end{split}
\enn
Substituting the above estimates for $I_1$-$I_4$ into \eqref{b39}, using Gronwall's inequality, \eqref{a13} and the assumption \eqref{a72}, we get
\ben
\begin{split}
\|\lge z\rge^l \pt^{k_0}_x\varphi\|_{L^\infty_TL^2_{xz}}^2+\|\lge z\rge^l \pt^{k_0}_x\pt_z\varphi\|_{L^2_TL^2_{xz}}^2\leq C,
\end{split}
\enn
which, along with \eqref{a72} gives \eqref{b38}.

With $0\leq j \leq {k_0}$, applying $\pt^j_x$ to the first equation of \eqref{e25} and then taking the $L^2_{xz}$ inner product of the resulting equation with $\lge z\rge^{2l} \pt^j_x\varphi_t$ for $l\in \mathbb{N}$ to have
\be\label{b40}
\begin{split}
&\frac{1}{2}\frac{d}{dt}\|\lge z\rge^l \pt^j_x\pt_z\varphi\|_{L^2_{xz}}^2+\|\lge z\rge^l \pt^j_x\varphi_t\|_{L^2_{xz}}^2\\
=&-2l\int_0^\infty\int_{-\infty}^\infty \lge z\rge^{2l-2} z \pt^j_x\pt_z\varphi\,\pt^j_x\varphi_t dxdy
-\int_0^\infty\int_{-\infty}^\infty \lge z\rge^{2l} z\pt^j_x [\ol{\partial_y u^{0}_2}\pt_z\varphi]\,\pt^j_x\varphi_t dxdy\\
&-\int_0^\infty\int_{-\infty}^\infty \lge z\rge^{2l} \pt^j_x [\ol{m^{0}}(\ol{c^{0}}+1)\varphi]\,\pt^j_x\varphi_t dxdy
+\int_0^\infty\int_{-\infty}^\infty \lge z\rge^{2l} \pt^j_x \rho\,\pt^j_x\varphi_t dxdy\\
:=&Q_1+Q_2+Q_3+Q_4.
\end{split}
\ee
It follows from the Sobolev embedding inequality and Cauchy-Schwarz inequality that
\ben
\begin{split}
Q_2
\leq & \|\pt^j_x\ol{\partial_y u^{0}_2}\|_{L^2_x}\|\lge z\rge^{l+1}\pt_z\varphi\|_{L^\infty_xL^2_{z}}
\|\lge z\rge^l\pt^{j}_x\varphi_t\|_{L^2_{xz}}\\
&+\sum_{i=0}^{j-1}\|\pt^i_x \ol{\partial_y u^{0}_2}\|_{L^\infty_x}\|\lge z\rge^{l+1}\pt^{j-i}_x\pt_z\varphi\|_{L^2_{xz}}
\|\lge z\rge^l\pt^{j}_x\varphi_t\|_{L^2_{xz}}
\\
\leq &C({k_0}+1)^2\|\ol{\partial_y u^{0}_2}\|_{H^{k_0}_x}^2\|\lge z\rge^{l+1}\pt_z\varphi\|_{H^{k_0}_{x}L^2_z}^2
+\frac{1}{6}\|\lge z\rge^{l}\pt^{j}_x\varphi_t\|_{L^2_{xz}}^2
\end{split}
\enn
and that
\ben
\begin{split}
Q_3
\leq&  C({k_0}+1)^2\|\ol{m^{0}}(\ol{c^{0}}+1)\|_{H^{k_0}_x}^2\|\lge z\rge^{l}\varphi\|_{H^{k_0}_{x}L^2_z}^2
+\frac{1}{6}\|\lge z\rge^{l}\pt^{j}_x\varphi_t\|_{L^2_{xz}}^2.
\end{split}
\enn
The Cauchy-Schwarz inequality entails that
\ben
\begin{split}
Q_1+Q_4
\leq &6l^2\|\lge z\rge^{l}\pt_z\varphi\|_{H^{k_0}_{x}L^2_z}^2+6\|\lge z\rge^{l}\rho\|_{H^{k_0}_{x}L^2_z}^2
+\frac{1}{6}\|\lge z\rge^{l}\pt^{j}_x\varphi_t\|_{L^2_{xz}}^2.
\end{split}
\enn
Inserting the above estimates for $Q_1-Q_4$ into \eqref{b40}, one gets
\be\label{b41}
\begin{split}
&\frac{d}{dt}\|\lge z\rge^l \pt^j_x\pt_z\varphi\|_{L^2_{xz}}^2+\|\lge z\rge^l \pt^j_x\varphi_t\|_{L^2_{xz}}^2\\
\leq &
C[({k_0}+1)^2\|\ol{\partial_y u^{0}_2}\|_{H^{k_0}_x}^2+l^2]\|\lge z\rge^{l+1}\pt_z\varphi\|_{H^{k_0}_{x}L^2_z}^2
+6\|\lge z\rge^{l}\rho\|_{H^{k_0}_{x}L^2_z}^2\\
&+C({k_0}+1)^2\|\ol{m^{0}}(\ol{c^{0}}+1)\|_{H^{k_0}_x}^2\|\lge z\rge^{l}\varphi\|_{H^{k_0}_{x}L^2_z}^2
.
\end{split}
\ee
Summing \eqref{b41} from $j=0$ to $j={k_0}$, then employing Gronwall's inequality to the resulting inequality and using \eqref{a13}, \eqref{b38} to derive
\be\label{b42}
\begin{split}
\|\lge z\rge^l \pt_z\varphi\|_{L^\infty_TH^{k_0}_xL^2_{z}}^2+\|\lge z\rge^l\varphi_t\|_{L^2_TH^{k_0}_xL^2_{z}}^2\leq C
\end{split}
\ee
holds true for each $l\in \mathbb{N}$.

Differentiating the first equation of \eqref{e25} with respect to $t$ and applying $\pt^j_x$ with $0\leq j\leq k_1$ to the resulting equality, one gets
\be\label{e26}
\begin{split}
\pt^j_x\varphi_{tt}+z \pt^j_x[\ol{\partial_y u^{0}_{2t}}\partial_z \varphi]+z \pt^j_x[\ol{\partial_y u^{0}_2}\partial_z \varphi_t]
=&\pt^j_x\partial^2_{z} \varphi_t+\pt_x^j\rho_t -\pt^j_x\{[\ol{m^{0}}(\ol{c^{0}}+1)]_t\varphi\}\\
&-\pt^j_x\{\ol{m^{0}}(\ol{c^{0}}+1)\varphi_t\}
.
\end{split}
\ee
For $l\in \mathbb{N}$, testing \eqref{e26} with $\lge z\rge^{2l} \pt^j_x\varphi_{t}$ in $L^2_{xz}$ and by a similar argument used in attaining \eqref{b38}, we obtain 
\be\label{b43}
\begin{split}
\|\lge z\rge^l \varphi_t\|_{L^\infty_TH^{k_1}_xL^2_{z}}^2+\|\lge z\rge^l \pt_z\varphi_t\|_{L^2_TH^{k_1}_xL^2_{z}}^2\leq C
\end{split}
\ee
holds true for each $l\in\mathbb{N}$. Testing \eqref{e26} with $\lge z\rge^{2l} \pt^j_x\varphi_{tt}$ in $L^2_{xz}$ and employing a similar argument used in deriving \eqref{b42}, one has 
\be\label{b44}
\begin{split}
\|\lge z\rge^l \pt_z\varphi_t\|_{L^\infty_TH^{k_1}_xL^2_{z}}^2+\|\lge z\rge^l \varphi_{tt}\|_{L^2_TH^{k_1}_xL^2_{z}}^2\leq C.
\end{split}
\ee

Assume further that \eqref{a15} holds true. Applying $\pt_t$ and $\pt^j_x$ to \eqref{e26} with $0\leq j\leq k_2$ and taking the $L^2_{xz}$ inner product of the resulting equation with $\lge z\rge^{2l} \pt^j_x\varphi_{tt}$ and $\lge z\rge^{2l} \pt^j_x\varphi_{ttt}$ respectively, then using similar arguments in obtaining \eqref{b38} and \eqref{b42} to have
\be\label{b45}
\begin{split}
\|\lge z\rge^l \varphi_{tt}\|_{L^\infty_TH^{k_2}L^2_{z}}^2+\|\lge z\rge^l \pt_z\varphi_{tt}\|_{L^2_TH^{k_2}_xL^2_{z}}^2\leq C
\end{split}
\ee
and 
\be\label{b46}
\begin{split}
\|\lge z\rge^l \pt_z\varphi_{tt}\|_{L^\infty_TH^{k_2}_xL^2_{z}}^2+\|\lge z\rge^l \varphi_{ttt}\|_{L^2_TH^{k_2}_xL^2_{z}}^2\leq C.
\end{split}
\ee
Collecting \eqref{b38}, \eqref{b42}, \eqref{b43} and \eqref{b44}, one derives \eqref{a14}. \eqref{a16} follows from \eqref{b45} and \eqref{b46}.  The proof is completed.

\end{proof}

Based on the above results, we next establish the well-posedness of \eqref{e7}.
\begin{lemma}\label{l10}
Let $(m^0,c^0,\vec{u}^{\,0},\na p^0)$ be the solution derived in Proposition \ref{p2}, whose lifespan is $T_*$. Suppose $0<T<T_*$. Then system \eqref{e7} admits a unique solution $c^{B,1}$ on $[0,T]$, which, along with the $m^{B,1}$ defined in \eqref{e8} fulfills
\be\label{b54}
\begin{split}
&\lge z\rge^l c^{B,1},\,\, \lge z\rge^l m^{B,1}\in C([0,T];H^{5}_xL^2_z)\cap C([0,T];H^{4}_xH^1_z)\cap C([0,T];H^{3}_xH^3_z),\\
&\lge z\rge^l c^{B,1}_t,\,\, \lge z\rge^l m^{B,1}_t\in C([0,T];H^{4}_xL^2_z)\cap C([0,T];H^{3}_xH^1_z)\cap C(0,T;H^{2}_xH^3_z)\cap L^2(0,T;H^{3}_xH^2_z),\\
&\lge z\rge^l \pt^2_t c^{B,1},\,\,\lge z\rge^l \pt^2_tm^{B,1}\in C([0,T];H^{2}_xH^1_z)\cap L^2(0,T;H^{3}_xL^2_z),\\
&\lge z\rge^l \pt^3_t c^{B,1},\,\,\lge z\rge^l\pt^3_t m^{B,1}\in L^2(0,T;H^2_xL^2_{z}),
\end{split}
\ee
for each $l\in \mathbb{N}$.
Furthermore, the $p^{B,2}$ defined in \eqref{e11} satisfies
\be\label{b70}
\begin{split}
&\lge z\rge^l p^{B,2}\in C([0,T];H^{5}_xH^1_z)\cap C([0,T];H^{4}_xH^2_z)\cap C([0,T];H^{3}_xH^4_z),\\
&\lge z\rge^l p^{B,2}_t\in C([0,T];H^{4}_xH^1_z)\cap C([0,T];H^{3}_xH^2_z)\cap C([0,T];H^{2}_xH^4_z)
%\cap L^2(0,T;H^{4}_xH^2_z)
\end{split}
\ee
for each $l\in \mathbb{N}$.
\end{lemma}
\begin{proof}
From Proposition \ref{p2} and the trace theorem, we deduce that
\be\label{a1}
\begin{split}
\pt_t^j\ol{\pt_y c^0}\in L^2(0,T;H^{5-j}_x)
\end{split}
\ee
for $j=0,1,2,3$. Let $S(x,z,t)$ be the solution of the following system 
\be\label{e0}
\left\{
\begin{array}{lll}
S_t
=\partial_z^2 S,\\
S(x,z,0)=0,\\
\partial_z S(x,0,t)=-\ol{\partial_y c^{0}}.
\end{array}
\right.
\ee
Then it follows from the standard well-posedness theory on parabolic systems and \eqref{a1} that
\be\label{a22}
\begin{split}
&\lge z\rge^l\pt^j_t S\in C([0,T];H^{5-j}_xL^2_z)\cap C([0,T];H^{4-j}_xH^1_z)\cap L^2(0,T;H^{5-j}_xH^1_z),\quad j=0,1,2,\\
&\lge z\rge^l\pt^3_t S\in C([0,T];H^{2}_xL^2_z)\cap L^2(0,T;H^{2}_xH^1_z)
\end{split}
\ee
for each $l\in \mathbb{N}$. Denote $\tilde{c}^{B,1}(x,z,t)=c^{B,1}(x,z,t)-S(x,z,t)$. Then from \eqref{e7} and \eqref{e0}, one knows that $\tilde{c}^{B,1}$ solves
\be\label{e10}
\left\{
\begin{array}{lll}
\tilde{c}^{B,1}_t+z \ol{\partial_y u^{0}_2} \partial_z \tilde{c}^{B,1}+\ol{m^{0}}(\ol{c^{0}}+1)\tilde{c}^{B,1}
=\partial_z^2 \tilde{c}^{B,1}+\Phi,\\
\tilde{c}^{B,1}(x,z,0)=0,\\
\partial_z \tilde{c}^{B,1}(x,0,t)=0,
\end{array}
\right.
\ee
where 
\ben
\Phi(x,z,t)=-z \ol{\partial_y u^{0}_2} \partial_z S(x,z,t)-\ol{m^{0}}(\ol{c^{0}}+1)S(x,z,t).
\enn
From Proposition \ref{p2} and the trace theorem, one gets
%\be\label{a23}
%\begin{split}
%&\pt^i_t \ol{\pt_y u^{0}_2},\,\,\pt_t^i[\ol{m^0}(\ol{c^0}+1)]\in L^2(0,T;H^{4-2i}_x),\quad i=0,1,2,
%\end{split}
%\ee
\be\label{b55}
\begin{split}
&\pt^j_t \ol{\pt_y u^{0}_2},\,\,\pt_t^j[\ol{m^0}(\ol{c^0}+1)]\in C([0,T];H^{5-2j}_x)\cap L^2(0,T;H^{6-2j}_x) ,\quad j=0,1,2.\\
\end{split}
\ee
which, along with \eqref{a22} leads to
\be\label{a31}
\begin{split}
&\lge z\rge^l \Phi\in L^2(0,T;H^{5}_xL^2_z),\quad \lge z\rge^l\pt_t \Phi\in L^2(0,T;H^{4}_xL^2_z),\quad
\lge z\rge^l\pt^2_t \Phi\in L^2(0,T;H^{2}_xL^2_z),
\end{split}
\ee
for each $l\in \mathbb{N}$. With \eqref{b55} and \eqref{a31} in hand, we apply Lemma \ref{l9} with $k_0=5$, $k_1=4$, $k_2=2$ to \eqref{e10} to have
\be\label{a47}
\begin{split}
&\lge z\rge^{l}\tilde{c}^{B,1}\in C([0,T];H^{5}_{x}H^1_z),\quad \lge z\rge^{l}\tilde{c}^{B,1}_t\in C([0,T];H^{4}_{x}H^1_z)\cap L^2(0,T;H^{5}_{x}L^2_z),\\
&\lge z\rge^{l}\pt_t^2\tilde{c}^{B,1}\in C([0,T];H^{2}_{x}H^1_z)\cap L^2(0,T;H^{4}_{x}L^2_z),\quad \lge z\rge^{l}\pt_t^3\tilde{c}^{B,1}\in L^2(0,T;H^{2}_{x}L^2_z)
\end{split}
\ee
for each $l\in\mathbb{N}$, which, along with \eqref{a22} gives rise to
\be\label{a26}
\begin{split}
&\lge z\rge^l\pt^i_t c^{B,1}\in C([0,T];H^{5-i}_xL^2_z)\cap C([0,T];H^{4-i}_xH^1_z),\quad i=0,1,\\
&\lge z\rge^l \pt^2_tc^{B,1}\in C([0,T];H^{2}_xH^1_z)\cap L^2(0,T;H^{3}_xL^2_z),\qquad \lge z\rge^l\pt^3_t c^{B,1}\in  L^2(0,T;H^2_xL^2_{z})
\end{split}
\ee
for each $l\in \mathbb{N}$.
 It follows from \eqref{e7}, \eqref{a26} and \eqref{b55} that
\be\label{b52}
\begin{split}
\|\lge z\rge^{l}c^{B,1}\|_{L^\infty_T H^{3}_x H^2_z}
\leq& \|\lge z\rge^l c^{B,1}_t\|_{L^\infty_T H^{3}_x L^2_z}
+\|\ol{\pt_y u^{0}_2}\|_{L^\infty_T H^{3}_x}\|\lge z\rge^{l+1}c^{B,1}\|_{L^\infty_T H^{3}_x H^1_z}\\
&+\|[\ol{m^{0}}(\ol{c^{0}}+1)]\|_{L^\infty_T H^{3}_x}
\|\lge z\rge^{l}c^{B,1}\|_{L^\infty_T H^{3}_x L^2_z}
\\
\leq &C,
\end{split}
\ee
which, along with \eqref{e7}, \eqref{a26} and \eqref{b55} gives
\be\label{b47}
\begin{split}
\|\lge z\rge^{l}c^{B,1}\|_{L^\infty_T H^{3}_x H^3_z}
\leq& \|\lge z\rge^lc^{B,1}_t\|_{L^\infty_T H^{3}_x H^1_z}
+\|\ol{\pt_y u^{0}_2}\|_{L^\infty_T H^{3}_x}\|\lge z\rge^{l+1}c^{B,1}\|_{L^\infty_T H^{3}_x H^2_z}\\
&+\|[\ol{m^{I,0}}(\ol{c^{0}}+1)]\|_{L^\infty_T H^{3}_x}
\|\lge z\rge^{l}c^{B,1}\|_{L^\infty_T H^{3}_x H^1_z}
\\
\leq &C.
\end{split}
\ee
Differentiating \eqref{e7} with respect to $t$ and applying a similar argument used in attaining \eqref{b52} and \eqref{b47} to the resulting equation and using \eqref{a26} and \eqref{b55} to have
\be\label{b53}
\|\lge z\rge^{l}c^{B,1}_t\|_{L^\infty_T H^{2}_x H^3_z}+\|\lge z\rge^{l}c^{B,1}_t\|_{L^2_T H^{3}_x H^2_z}\leq C.
\ee
Collecting \eqref{a26}, \eqref{b47}, \eqref{b53} and using \eqref{e8}, we derive \eqref{b54}. \eqref{b70} follows from \eqref{e11} and \eqref{b54}. The proof is completed.

\end{proof}
 The well-posedness on \eqref{e12} is as follows.
\begin{lemma}\label{l11}
Let $(m^0,c^0,\vec{u}^{\,0},\na p^0)$ be the solution of derived in Proposition \ref{p2} and $0<T<T_*$. Let $c^{B,1}$ and $m^{B,1}$ be the solutions derived in Lemma \ref{l10}. Then system \eqref{e12} admits a unique solution $c^{B,2}$ on $[0,T]$, which along with the $m^{B,2}$ defined in \eqref{e13} satisfying
\be\label{b62}
\begin{split}
&\lge z\rge^{l} c^{B,2},\,\,\lge z\rge^{l} m^{B,2}\in C([0,T];H^{4}_{x}H^1_z)\cap C([0,T];H^{2}_{x}H^3_z),\\
& \lge z\rge^{l} c^{B,2}_t,\,\,\lge z\rge^{l} m^{B,2}_t\in C([0,T];H^{2}_{x}H^1_z)\cap L^2(0,T;H^{3}_{x}L^2_z)\cap L^2(0,T;H^{2}_{x}H^2_z),\\
&\lge z\rge^{l} \pt^2_t c^{B,2},\,\,\lge z\rge^{l} \pt^2_t m^{B,2} \in L^2(0,T;H^{2}_{x}L^2_z)
\end{split}
\ee
for each $l\in\mathbb{N}$. Furthermore, the $\xi$ in \eqref{e18} fulfills
\be\label{b63}
\begin{split}
\lge z\rge^{l} \xi\in C([0,T];H^3_{x}H^2_z),\quad 
%\lge z\rge^{l} \xi_t\in L^2(0,T;H^2_{xz}),
\lge z\rge^{l} \xi_t\in C([0,T];H^2_xH^2_{z}),\quad
\lge z\rge^{l}\pt^2_t\xi \in L^2(0,T;L^2_{xz})
\end{split}
\ee
for each $l\in\mathbb{N}$.
\end{lemma}
\begin{proof}
From Proposition \ref{p2} and the trace theorem, we have
\be\label{b50}
\begin{split}
\pt^i_t \ol{\pt_y m^{0}},\,\,\pt^i_t \ol{\pt_y c^{0}},\,\,\pt^i_t \ol{\pt_y u^{0}_1},\,\,\pt^i_t \ol{\pt_y^2 u^{0}_2}\in C([0,T];H^{4-2i}_x),\quad i=0,1.
\end{split}
\ee
%and
%\be\label{b97}
%\begin{split}
%\pt^i_t \ol{\pt_y c^{0}},\,\,\pt^i_t \ol{\pt_y m^{0}}\in C([0,T];H^{3-i}_x),\quad i=0,1.
%\end{split}
%\ee
It follows from \eqref{e14}, the Sobolev embedding inequality, \eqref{b54} and \eqref{b50} that
\be\label{b51}
\begin{split}
&\|\lge z\rge^l \Gamma\|_{L^\infty_TH^{4}_x L^2_z}\\
\leq& \|\ol{\partial_y u^{0}_1}\|_{L^\infty_TH^4_x}\|\lge z\rge^{l+1}c^{B,1}\|_{L^\infty_TH^5_xL^2_z}
+\|\ol{\partial^2_y u^{0}_2}\|_{L^\infty_TH^4_x}\|\lge z\rge^{l+2}c^{B,1}\|_{L^\infty_TH^4_xH^1_z}\\
&+\|\ol{\partial_y m^{I,0}}\|_{L^\infty_TH^4_x}\|\lge z\rge^{l}c^{B,1}\|_{L^\infty_TH^4_xL^2_z}
+\|\lge z\rge^{l+1}m^{B,1}\|_{L^\infty_TH^4_xL^2_z}\|\ol{\partial_y c^{0}}\|_{L^\infty_TH^4_x}\\
&+C\|\lge z\rge^{l}m^{B,1}\|_{L^\infty_TH^4_xH^1_z}\|c^{B,1}\|_{L^\infty_TH^4_xH^1_z}\\
\leq &C
\end{split}
\ee
for each $l\in \mathbb{N}$. By a similar argument used in attaining \eqref{b51}, one gets from \eqref{b54} and \eqref{b50} that
\be\label{b56}
\begin{split}
\|\lge z\rge^l \Gamma_t\|_{L^\infty_TH^{2}_x L^2_z}\leq C
\end{split}
\ee
for each $l\in \mathbb{N}$. For fixed $l\in \mathbb{N}$, it follows from the H\"{o}lder's inequality and \eqref{b54} that
\be\label{b57}
\begin{split}
&\big\|\lge z\rge^l\int_z^\infty m^{B,1}(x,\eta,t)d\eta \big\|_{L^\infty_TH^4_xL^2_z}\\
=&\sum_{j=0}^{4}\big\|\lge z\rge^l\int_z^\infty\pt^j_x m^{B,1}(x,\eta,t)d\eta \big\|_{L^\infty_TL^2_xL^2_z}\\
\leq & \sum_{j=0}^{4}\big\|\lge z\rge^{-1}\big[\int_z^\infty\lge \eta\rge^{2(l+1)}|\pt^j_x m^{B,1}(x,\eta,t)|^2 d\eta\big]^{\frac{1}{2}} \big\|_{L^\infty_TL^2_xL^2_z}\\ 
%\leq & \big\{\int_0^\infty (z^2+1)^{-1}dz\big\}^{\frac{1}{2}}\times\sum_{j=0}^{3}\sup_{0<t\leq T }\big\{\int_0^\infty\int_0^\infty\lge \xi\rge^{2(l+2)}|\pt^j_x m^{B,1}(x,\xi,t)|^2 d\xi dx \big\}^{\frac{1}{2}}\\ 
\leq & C\|\lge z\rge^{l+1} m^{B,1}\|_{L^\infty_TH^4_xL^2_z}.
\end{split}
\ee
In a similar fashion as above, one can estimate the other terms in \eqref{e15} to deduce that
\be\label{b58}
\|\lge z\rge^l\ol{c^0}\Theta\|_{L^\infty_TH^4_xH^1_z}+\|\lge z\rge^l (\ol{c^0}\Theta)_t\|_{L^\infty_TH^{2}_x H^1_z}\leq C
\ee
for each $l\in \mathbb{N}$.
With \eqref{b55}, \eqref{b51}-\eqref{b58} in hand, we employ Lemma \ref{l9} with $k_0=4$, $k_1=2$ to  \eqref{e12}, and derive that
\be\label{b59}
\begin{split}
\lge z\rge^{l}\pt^i_t c^{B,2}\in C([0,T];H^{4-2i}_{x}H^1_z),\quad \lge z\rge^{l}\pt^{i+1}_t c^{B,2}\in L^2(0,T;H^{4-2i}_{x}L^2_z), \quad i=0,1,
\end{split}
\ee
for each $l\in\mathbb{N}$. Moreover, \eqref{b54}, \eqref{b50} and similar arguments used in obtaining \eqref{b51} and \eqref{b58} further entail that
\ben
\|\lge z\rge^l\Gamma\|_{L^\infty_TH^3_xH^1_z}+\|\lge z\rge^l\ol{c^0}\Theta\|_{L^\infty_TH^3_xH^3_z}\leq C,\qquad \forall \,l\in \mathbb{N},
\enn
which, in conjunction with the first equation in \eqref{e12}, \eqref{b59} and a similar argument used in deriving \eqref{b47} gives
\be\label{b60}
\lge z\rge^{l}c^{B,2}\in C([0,T];H^{2}_{x}H^3_z),\qquad \forall\, l\in\mathbb{N}.
\ee
Differentiating the first equation in \eqref{e12} with respect to $t$, employing a similar argument used in attaining \eqref{b47} to the resulting equation and using \eqref{b56} and \eqref{b58}, one gets
\be\label{b61}
\lge z\rge^{l}c^{B,2}_t\in L^2(0,T;H^{2}_{x}H^2_z),\qquad \forall\, l\in\mathbb{N}.
\ee
\eqref{b62} follows from \eqref{b59}-\eqref{b61} and \eqref{e13}. Using \eqref{b54}, \eqref{b62} and applying a similar argument used in deriving \eqref{b57} to each term in \eqref{e18}, one gets \eqref{b63}. The proof is finished.

\end{proof}
\section{Proof of the main results}
This section is devoted to proving Theorem \ref{t1}. In Subsection 4.1, we first construct approximation solutions for $(m^\va,c^\va,\vec{u}^{\,\va},p^\va)$, then establish the well-posedness on remainders between $(m^\va,c^\va,\vec{u}^{\,\va},p^\va)$ and the approximation solutions. With the results derived in Subsection 4.1, we give the proof of Theorem \ref{t1} in Subsection 4.2. 
\subsection{Estimates on the remainders}
The approximation solutions are defined as:
\be\label{e21}
\begin{split}
M^a(x,y,t)=&m^{0}(x,y,t)+\va^{\frac{1}{2}}m^{B,1}(x,\frac{y}{\sqrt{\va}},t)
+\va m^{B,2}(x,\frac{y}{\sqrt{\va}},t)+\va^{\frac{3}{2}}\xi(x,\frac{y}{\sqrt{\va}},t),\\
C^a(x,y,t)=&c^{0}(x,y,t)+\va^{\frac{1}{2}}c^{B,1}(x,\frac{y}{\sqrt{\va}},t)
+\va c^{B,2}(x,\frac{y}{\sqrt{\va}},t),\\
%\vec{U}^{\,a}(x,y,t)=&\vec{u}^{\,I,0}(x,y,t),\\
P^{\,a}(x,y,t)=&p^{0}(x,y,t)+\va p^{B,2}(x,\frac{y}{\sqrt{\va}},t)
\end{split}
\ee
and the remainders are given as:
\be\label{e22}
\begin{split}
&M^\va(x,y,t)=\va^{-\frac{1}{2}}[m^\va (x,y,t)-M^a (x,y,t)],\qquad C^\va(x,y,t)=\va^{-\frac{1}{2}}[c^\va (x,y,t)-C^a (x,y,t)],\\
&\vec{U}^{\,\va}(x,y,t)=\va^{-\frac{1}{2}}[\vec{u}^{\,\va} (x,y,t)-\vec{u}^{\,0} (x,y,t)],\qquad \, \, \,\, P^\va(x,y,t)=\va^{-\frac{1}{2}}[p^\va (x,y,t)-P^{\,a} (x,y,t)].
\end{split}
\ee
Substituting \eqref{e22} into \eqref{e1}-\eqref{e2}, we deduce that the remainders $(M^\va,C^\va,\vec{U}^{\,\va},P^\va)(x,y,t)$ fulfill the following initial-boundary value problem:
\be\label{e23}
\left\{
\begin{array}{lll}
M^\va_t+\va^{\frac{1}{2}}\vec{U}^{\,\va}\cdot \nabla M^\va-\Delta M^\va
=-\vec{U}^{\,\va}\cdot \nabla M^a-\vec{u}^{\,0}\cdot \nabla M^\va+\va^{-\frac{1}{2}}f^\va\\
\qquad \qquad \qquad \qquad \qquad \quad \,\,\,\, -\nabla\cdot [\va^{\frac{1}{2}}M^\va \nabla C^\va+M^\va \nabla C^a+M^a \nabla C^\va],\\
C^\va_t+\va^{\frac{1}{2}}\vec{U}^{\,\va}\cdot \nabla C^\va-\va \Delta C^\va=-\vec{U}^{\,\va}\cdot \nabla C^a-\vec{u}^{\,0}\cdot \nabla C^\va+\va^{-\frac{1}{2}}g^\va\\
\qquad \qquad \qquad \qquad \qquad \quad\,\,\,-\va^{\frac{1}{2}}M^\va C^\va -M^\va C^a-M^a C^\va,\\
\vec{U}^{\,\va}_t+\va^{\frac{1}{2}}\vec{U}^{\,\va}\cdot \nabla \vec{U}^{\,\va}+\vec{U}^{\,\va}\cdot \nabla \vec{u}^{\,0}+\vec{u}^{\,0}\cdot \nabla \vec{U}^{\,\va}+\nabla P^\va+M^\va (0,\lambda)=\Delta \vec{U}^{\,\va}+\va^{-\frac{1}{2}}\vec{h}^{\,\va},\\
\nabla \cdot \vec{U}^{\,\va}=0,\\
(M^\va,C^\va,\vec{U}^{\,\va})(x,y,0)=(0,0,\mathbf{0}),\\
\partial_y M^\va(x,0,t)
=-\va^{\frac{1}{2}} \pt_z\xi(x,0,t),\quad
\partial_y C^\va(x,0,t)=0,\quad
\vec{U}^{\,\va}(x,0,t)=\mathbf{0},
\end{array}
\right.
\ee
where
\be\label{e24}
\begin{split}
f^\va(x,y,t):=&\Delta M^a(x,y,t)-M^a_t(x,y,t)-(\vec{u}^{\,0}\cdot\nabla M^a)(x,y,t)-\nabla\cdot(M^a\nabla C^a)(x,y,t),\\
g^\va (x,y,t):=&\va \Delta C^a(x,y,t)-C^a_t(x,y,t)-(\vec{u}^{\,0}\cdot\nabla C^a)(x,y,t)-(M^a C^a)(x,y,t),\\
\vec{h}^{\,\va}(x,y,t):=&\Delta \vec{u}^{\,0}(x,y,t)-\vec{u}^{\,0}_t(x,y,t)
-(\vec{u}^{\,0}\cdot\nabla \vec{u}^{\,0})(x,y,t)\\
&-\nabla P^{\,a}(x,y,t)-(0,\lambda M^a(x,y,t)).
\end{split}
\ee

We next give the derivation of the boundary conditions in \eqref{e23}. By the boundary conditions in \eqref{e7}, \eqref{e12} and $\pt_y c^{\va}(x,0,t)=0$, one deduces that
\be\label{b32}
\pt_y C^{a}(x,0,t)=0,\qquad \pt_y C^{\va}(x,0,t)=0.
\ee
It follows from \eqref{b33}, \eqref{b25}, \eqref{e6}, the boundary conditions in \eqref{e3}, \eqref{e7} and \eqref{e12} that
\ben
\begin{split}
&\pt_y m^{0}(x,0,t)+\pt_zm^{B,1}(x,0,t)=0,\\
& \pt_zm^{B,2}(x,0,t)=m^{0}(x,0,t)\pt_zc^{B,2}(x,0,t)+m^{B,1}(x,0,t)[\pt_yc^{0}+\pt_z c^{B,1}](x,0,t)=0, 
\end{split}
\enn
which, along with \eqref{e2} gives rise to
\be\label{b37}
\begin{split}
\pt_yM^{\va}(x,0,t)
=&-\va^{-\frac{1}{2}}\pt_y M^a(x,0,t)\\
&=-\va^{-\frac{1}{2}}[\pt_y m^{0}+\pt_zm^{B,1}](x,0,t)-\pt_zm^{B,2}(x,0,t)-\va^{\frac{1}{2}}\pt_z\xi(x,0,t)\\
&=-\va^{\frac{1}{2}}\pt_z\xi(x,0,t).
\end{split}
\ee
Collecting \eqref{b32} and \eqref{b37}, one derives the boundary conditions for $M^\va$ and $C^\va$.

The main result of this subsection is as follows. 
\begin{prop}\label{p3}
Suppose that $m_0\in H^6_{xy}$, $c_0, \vec{u}_{0}\in H^7_{xy}$ fulfill compatibility conditions $(A)$. Let $(m^0,c^0,\vec{u}^{\,0},$ $\na p^0)$ and $T_*$ be derived in Proposition \ref{p2} and $0<T<T_*$. Then there exists a constant $\va_T$ depending on $T$, given in \eqref{b48}, such that for each $0<\va<\va_T$, system \eqref{e23} admits a unique solution \be\label{b91}
(M^\va, C^\va, \vec{U}^{\,\va},\na P^\va)\in C([0,T];H^2_{xy}\times H^3_{xy} \times H^3_{xy}\times H^1_{xy})
\ee
 satisfying
\be\label{b86}
\begin{split}
\|M^\va\|_{L^\infty_TL^\infty_{xy}}+\|C^\va\|_{L^\infty_TL^\infty_{xy}}+\|\vec{U}^{\,\va}\|_{L^\infty_TH^2_{xy}}\leq C\va^{\frac{1}{8}}, 
\end{split}
\ee
and 
\be\label{b87}
\begin{split}
\|\vec{U}^{\,\va}\|_{L^\infty_TL^\infty_{xy}}\leq C\va^{\frac{1}{2}},\qquad
\|\na C^\va\|_{L^\infty_TL^\infty_{xy}}\leq C\va^{-\frac{3}{8}}, \qquad \|\na\vec{U}^{\,\va}\|_{L^\infty_TL^\infty_{xy}}\leq C\va^{\frac{1}{4}}.
\end{split}
\ee
\end{prop}

Existence and uniqueness of solutions to system \eqref{e23} with regularity \eqref{b91} follows from the standard method used in \cite{Wang-Zhao13}, we omit its proof for brevity and proceed to deriving a series of \emph{a priori} estimate for the solutions in the following Lemma \ref{l13} - Lemma \ref{l8} to attain \eqref{b86} and \eqref{b87}. To this end, we next estimate $f^\va$, $g^\va$ and $\vec{h}^{\,\va}$. 

\begin{lemma}\label{l1} Let the assumptions in Proposition \ref{p3} hold. Then there exists a constant $C$ independent of $\va$, such that
\ben
\|f^\va\|_{L^\infty_TL^2_{xy}}\leq C\va^{\frac{5}{4}},\qquad \|f^\va_t\|_{L^2_TL^2_{xy}}\leq C\va^{\frac{5}{4}}.
\enn
\end{lemma}
\begin{proof}
By the change of variables $z=\frac{y}{\sqrt{\va}}$ and direct computations, one deduces that
\ben
\begin{split}
\Delta M^a(x,y,t)=& \Delta m^{0}+\va^{-\frac{1}{2}}\partial^2_z m^{B,1}+\partial^2_z m^{B,2}+\va^{\frac{1}{2}}\partial^2_z \xi+\va^{\frac{1}{2}}\partial^2_x m^{B,1}
\\
&+\va \partial^2_x m^{B,2}+\va^{\frac{3}{2}}\partial^2_x \xi
\end{split}
\enn
and
\ben
\begin{split}
-\nabla\cdot(&M^a\nabla C^a)\\
=&-\nabla\cdot (m^{0}\nabla c^{0})
-[\partial_y m^{0}\partial_z c^{B,1}
+m^{0}\partial^2_z c^{B,2}
+\partial_z m^{B,1}\partial_y c^{0}
+\partial_z(m^{B,1}\partial_z c^{B,1})]\\
&-\va^{-\frac{1}{2}}m^{0}\partial^2_z c^{B,1}
-\va^{\frac{1}{2}}[\partial_x(m^{0}\partial_x c^{B,1}+m^{B,1}\partial_x c^{0})
+\partial_y m^{0}\partial_z c^{B,2}
+m^{B,1}\partial^2_y c^{0}]\\
&-\va^{\frac{1}{2}}[\partial_z m^{B,2}\partial_y c^{0}
+\partial_z(m^{B,1}\partial_z c^{B,2}
+m^{B,2}\partial_z c^{B,1})
]
-\va(m^{B,2}\partial_y^2 c^{0}
+\partial_z \xi\partial_y c^{0})
\\
&-\va \partial_x(m^{0}\partial_x c^{B,2}
+m^{B,2}\partial_x c^{0}+m^{B,1}\partial_x c^{B,1})
-\va
\partial_z(\xi\partial_z c^{B,1}
+m^{B,2}\partial_z c^{B,2}
)
\\
&-\va^{\frac{3}{2}}\partial_x(m^{B,1}\partial_x c^{B,2}
+m^{B,2}\partial_x c^{B,1}
+\xi\partial_x c^{0}
)
-\va^\frac{3}{2}\pt_z(\xi\partial_z c^{B,2})\\
&-\va^{\frac{3}{2}}
 \xi\partial^2_y c^{0}
-\va^{2}\partial_x(m^{B,2}\partial_x c^{B,2}+\xi\partial_x c^{B,1}
)
-\va^{\frac{5}{2}}\pt_x(\xi\partial_x c^{B,2})
\end{split}
\enn
and
\ben
\begin{split}
-u^{\,0}\cdot\nabla M^a=&-u^{\,0}\cdot \nabla m^{0}
-u^{0}_2 \partial_z m^{B,1}
-\va^{\frac{1}{2}}[u^{0}_1\partial_xm^{B,1}+u^{0}_2\partial_z m^{B,2}]\\
&-\va (u^{0}_1\partial_x m^{B,2}
+ u^{0}_2\partial_z \xi)
-\va^{\frac{3}{2}}u^{0}_1\partial_x \xi
\end{split}
\enn
and
\ben
-M^a_t=-m^{0}_t-\va^{\frac{1}{2}}m^{B,1}_t-\va m^{B,2}_t-\va^{\frac{3}{2}}\xi_t.
\enn
Substituting the above four identities into \eqref{e24} and using 
\eqref{b29}, \eqref{b30}, \eqref{b28} and the change of variables $z=\frac{y}{\sqrt{\va}}$, one gets after rearrangement that
\ben
\begin{split}
f^{\va}=&\va^{-\frac{1}{2}}(\ol{m^{0}}-m^{0}+y\ol{\pt_ym^{0}}
+\frac{y^2}{2}\ol{\pt_y^2m^{0}})\pt_z^2c^{B,1}
+(\ol{\pt_ym^{0}}-\pt_ym^{0}
+y\ol{\pt_y^2m^{0}})\pt_zc^{B,1}\\
&+(\ol{m^{0}}-m^{0}+y\ol{\pt_ym^{0}}
)\pt_z^2c^{B,2}
+\pt_z m^{B,1}(\ol{\pt_yc^{0}}-\pt_yc^{0}
+y\ol{\pt_y^2c^{0}})\\
&+(\ol{u^{0}_2}-u^{0}_2+y\ol{\pt_yu^{0}_2}
)\pt_zm^{B,1}+\va^{\frac{1}{2}}\pt_x[m^{B,1}(\ol{\pt_xc^{0}}-\pt_xc^{0})+(\ol{m^{0}}-m^{0})\pt_x c^{B,1}]\\
&+\va^{\frac{1}{2}}(\ol{\pt_ym^{0}}-\pt_ym^{0})\pt_zc^{B,2}
+\va^{\frac{1}{2}}\pt_z m^{B,2}(\ol{\pt_yc^{0}}-\pt_yc^{0})
+\va^{\frac{1}{2}}m^{B,1}(\ol{\pt_y^2c^{0}}-\pt_y^2c^{0})\\
&
+\va^{\frac{1}{2}}(\ol{u^{0}_1}-u^{0}_1)\pt_x m^{B,1}
+\va^{\frac{1}{2}}(\ol{u^{0}_2}-u^{0}_2)\pt_z m^{B,2}
-\va(m^{B,2}\partial_y^2 c^{0}
+\partial_z \xi\partial_y c^{0})\\
&-\va \partial_x(m^{0}\partial_x c^{B,2}
+m^{B,2}\partial_x c^{0}+m^{B,1}\partial_x c^{B,1})
-\va
\partial_z(\xi\partial_z c^{B,1}
+m^{B,2}\partial_z c^{B,2}
)\\
&+\va \partial^2_x m^{B,2}-\va (u^{0}_1\partial_x m^{B,2}
+ u^{0}_2\partial_z \xi)-\va m^{B,2}_t+\va^{\frac{3}{2}}\partial^2_x \xi
-\va^{\frac{3}{2}}
 \partial_z (
\xi\partial_z c^{B,2})
\\
&-\va^{\frac{3}{2}}\partial_x(m^{B,1}\partial_x c^{B,2}
+m^{B,2}\partial_x c^{B,1}
+\xi\partial_x c^{0}
)
-\va^{\frac{3}{2}}
 \xi\partial^2_y c^{0}
 -\va^{\frac{3}{2}}\xi_t\\
 &-\va^{\frac{3}{2}}u^{0}_1\partial_x \xi
 -\va^{2}\partial_x(m^{B,2}\partial_x c^{B,2}+\xi\partial_x c^{B,1}
)
-\va^{\frac{5}{2}}\partial_x(
\xi\partial_x c^{B,2})
\\
:=&\sum_{i=1}^{9}K_i,
\end{split}
\enn
where $K_i$ represents the entirety of the $i$-th line in the above expression. By the change of variables $y=\va^{1/2}z$, Taylor's formula, the Sobolev embedding inequality, Proposition \ref{p2} and Lemma \ref{l10}, one gets
\be\label{a2}
\begin{split}
\|&(\ol{\pt_ym^{0}}-\pt_ym^{0}
+y\ol{\pt_y^2m^{0}})\pt_zc^{B,1}\|_{L^\infty_TL^2_{xy}}\\
=&
\va\Big\|\frac{\pt_y m^{0}(x,y,t)-\pt_y m^{0}(x,0,t)-y\pt_y^2 m^{0}(x,0,t)}{y^2}\cdot z^2\pt_zc^{B,1}\Big\|_{L^\infty_TL^2_{xy}}
\\
\leq &2 \va \|\pt_y^3 m^{0}\|_{L^\infty_TL^\infty_{xy}}\,\|z^2\pt_zc^{B,1}\|_{L^\infty_TL^2_{xy}}\\
\leq &C \va^{\frac{5}{4}} \| m^{0}\|_{L^\infty_TH^5_{xy}}\,\|\lge z\rge^2 c^{B,1}\|_{L^\infty_TL^2_{x}H^1_z}\\
\leq &C\va^{\frac{5}{4}}.
\end{split}
\ee
One can estimate the remaining part in $K_1$ by a similar argument used in deriving \eqref{a2} to deduce that
\ben
\begin{split}
\|K_1\|_{L^\infty_TL^2_{xy}}\leq& C\va^{\frac{5}{4}} \| m^{0}\|_{L^\infty_TH^5_{xy}}\|\lge z\rge^3c^{B,1} \|_{L^\infty_TL^2_{x}H^2_z}
+C\va^{5/4} \| m^{0}\|_{L^\infty_TH^5_{xy}}\|\lge z\rge^2 c^{B,1}\|_{L^\infty_TL^2_{x}H^1_z}\\
\leq &C\va^{\frac{5}{4}}.
\end{split}
\enn
Applying the arguments used in estimating $\|K_1\|_{L^\infty_TL^2_{xy}}$ to $K_2$- $K_4$ we get
\ben
\|K_2\|_{L^\infty_TL^2_{xy}}+\|K_3\|_{L^\infty_TL^2_{xy}}
+\|K_4\|_{L^\infty_TL^2_{xy}}
\leq C\va^{\frac{5}{4}}.
\enn
We next estimate the third term in $K_5$. By the change of variables $y=\va^{1/2}z$, Soblev embedding inequality, Proposition \ref{p2} and Lemma \ref{l11}, we have
\be\label{a3}
\begin{split}
\va\|m^{B,2}\partial_y^2 c^{0}+\pt_z\xi \pt_yc^0\|_{L^\infty_TL^2_{xy}}
\leq& \va\|m^{B,2}\|_{L^\infty_TL^2_{xy}}\|\partial_y^2 c^{0}\|_{L^\infty_TL^\infty_{xy}}+\va\|\pt_z \xi\|_{L^\infty_TL^2_{xy}}\|\partial_y c^{0}\|_{L^\infty_TL^\infty_{xy}}\\
\leq& C\va^{\frac{5}{4}}\| m^{B,2}\|_{L^\infty_TL^2_{xz}}\|c^{0}\|_{L^\infty_TH^4_{xy}}
+C\va^{\frac{5}{4}}\| \xi\|_{L^\infty_TL^2_{z}H^1_z}\|c^{0}\|_{L^\infty_TH^3_{xy}}\\
\leq &C\va^{\frac{5}{4}}.
\end{split}
\ee
One can estimate the other terms in $K_5$ by a similar argument used in deriving \eqref{a2} to deduce that
\ben
\va^{\frac{1}{2}}\|(\ol{u^{0}_1}-u^{0}_1)\pt_x m^{B,1}\|_{L^\infty_TL^2_{xy}}
+\va^{\frac{1}{2}}\|(\ol{u^{0}_2}-u^{0}_2)\pt_z m^{B,2}\|_{L^\infty_TL^2_{xy}}
\leq C\va^{\frac{5}{4}}, 
\enn
which, along with \eqref{a3} gives rise to
\be\label{a73}
\|K_5\|_{L^\infty_TL^2_{xy}}\leq C\va^{\frac{5}{4}}.
\ee
Similar arguments used in attaining \eqref{a3} and the assumption $0<\va<1$ lead to
\ben
\|K_6\|_{L^\infty_TL^2_{xy}}+\|K_7\|_{L^\infty_TL^2_{xy}}
+\|K_8\|_{L^\infty_TL^2_{xy}}+\|K_{9}\|_{L^\infty_TL^2_{xy}}
\leq C\va^{5/4}.
\enn
Collecting the above estimates for $K_1$- $K_{9}$, one immediately gets
\be\label{a57}
\|f^\va\|_{L^\infty_TL^2_{xy}}\leq C\va^{\frac{5}{4}}.
\ee

By a similar argument used in deriving \eqref{a2}, one gets
\be\label{b64}
\begin{split}
\|&\pt_t[(\ol{\pt_ym^{0}}-\pt_ym^{0}
+y\ol{\pt_y^2m^{0}})\pt_zc^{B,1}]\|_{L^2_TL^2_{xy}}\\
\leq &2 \va \|\pt_y^3 m^{0}_t\|_{L^2_TL^\infty_{xy}}\,\|z^2\pt_zc^{B,1}\|_{L^\infty_TL^2_{xy}}+2 \va \|\pt_y^3 m^{0}\|_{L^\infty_TL^\infty_{xy}}\,\|z^2\pt_zc^{B,1}_t\|_{L^2_TL^2_{xy}}\\
\leq &C \va^{\frac{5}{4}} \| m^{0}_t\|_{L^2_TH^5_{xy}}\,\|\lge z\rge^2 c^{B,1}\|_{L^\infty_TL^2_{x}H^1_z}
+C \va^{\frac{5}{4}} \| m^{0}\|_{L^\infty_TH^5_{xy}}\,\|\lge z\rge^2 c^{B,1}_t\|_{L^2_TL^2_{x}H^1_z}\\
\leq &C\va^{\frac{5}{4}}.
\end{split}
\ee
In a similar fashion in attaining \eqref{b64}, we deduce that
\ben
\|\pt_t K_i\|_{L^2_TL^2_{xy}}\leq C\va^{\frac{5}{4}},\quad i=1,2,\cdots,9,
\enn
which, gives rise to
\be\label{b65}
\|f^\va_t\|_{L^2_TL^2_{xy}}\leq \sum_{i=1}^{9}\|\pt_t K_i\|_{L^2_TL^2_{xy}}\leq C\va^{\frac{5}{4}}.
\ee
\eqref{b65}, along with \eqref{a57} completes the proof.

\end{proof}

\begin{lemma}\label{l2}
Let the assumptions in Proposition \ref{p3} hold. Then
there exists a constant $C$ independent of $\va$, such that
\ben
\|g^\va\|_{L^\infty_TL^2_{xy}}\leq C\va,\qquad \|\nabla g^\va\|_{L^\infty_TL^2_{xy}}\leq C\va,\qquad
\|\nabla g^\va_t\|_{L^2_TL^2_{xy}}\leq C\va^{\frac{3}{4}}.
\enn
\end{lemma}
\begin{proof}
Using \eqref{e3} and the change of variables $z=\frac{y}{\sqrt{\va}}$, we derive from \eqref{e24} that
\ben
\begin{split}
g^\va=&\va^{\frac{1}{2}}\pt_z^2c^{B,1}-\va^{\frac{1}{2}}c^{B,1}_t-u^{0}_2\pt_zc^{B,1}
-\va^{\frac{1}{2}}u^{0}_1\pt_xc^{B,1}-\va^{\frac{1}{2}}u^{0}_2\pt_zc^{B,2}
-\va^{\frac{1}{2}}m^{0}c^{B,1}-\va^{\frac{1}{2}}m^{B,1}c^{0}\\
&+\va\pt_z^2c^{B,2}-\va c^{B,2}_t-\va u^{0}_1\pt_xc^{B,2}
-\va m^{0}c^{B,2}- \va m^{B,1}c^{B,1}
-\va m^{B,2}c^{0}\\
&+\va \Delta c^{0}+\va^{\frac{3}{2}}(\pt^2_x c^{B,1}
-m^{B,1}c^{B,2}-m^{B,2}c^{B,1}-\xi c^{0})\\
&+\va^2(\pt^2_x c^{B,2}
- m^{B,2}c^{B,2}-\xi c^{B,1})
-\va^{\frac{5}{2}} \xi c^{B,2}\\
:=&\sum_{i=1}^4 P_i,
\end{split}
\enn
where $P_i$ represents the entirety of the $i$-th line. From \eqref{e7}, \eqref{e12}, the boundary conditions in \eqref{e3} and the change of variables $z=\frac{y}{\sqrt{\va}}$, one deduces that
\be\label{b66}
\begin{split}
P_1+P_2=&-(u^{0}_2-\ol{u^{0}_2}-y\ol{\pt_y u^{0}_2}-\frac{y^2}{2}\ol{\pt_y^2 u^{0}_2})\pt_z c^{B,1}
-\va^{\frac{1}{2}}(m^{0}-\ol{m^{0}}-y\ol{\pt_y m^{0}}) c^{B,1}\\
&-\va^{\frac{1}{2}}m^{B,1}(c^{0}-\ol{c^{0}}-y\ol{\pt_y c^{0}})
-\va^{\frac{1}{2}}(u^{0}_1-\ol{u^{0}_1}-y\ol{\pt_y u^{0}_1})\pt_x c^{B,1}\\
&-\va^{\frac{1}{2}}(u^{0}_2-\ol{u^{0}_2}-y\ol{\pt_y u^{0}_2})\pt_z c^{B,2}
-\va (u^{0}_1-\ol{u^{0}_1})\pt_xc^{B,2}\\
&-\va (m^{0}-\ol{m^{0}})c^{B,2}-\va m^{B,2}(c^{0}-\ol{c^{0}}),
\end{split}
\ee
where
\be\label{a6}
\begin{split}
\|&-(u^{0}_2-\ol{u^{0}_2}-y\ol{\pt_y u^{0}_2}-\frac{y^2}{2}\ol{\pt_y^2 u^{0}_2})\pt_z c^{B,1}\|_{L^\infty_TL^2_{xy}}\\
\leq &6\va^{\frac{3}{2}}\Big\|
\frac{u^{0}_2-\ol{u^{0}_2}-y\ol{\pt_y u^{0}_2}-\frac{y^2}{2}\ol{\pt_y^2 u^{0}_2}}{y^3}
\Big\|_{L^\infty_TL^\infty_{xy}}\,\|z^3 \pt_z c^{B,1}\|_{L^\infty_TL^2_{xy}}\\
\leq & C\va^{\frac{7}{4}} \|\na^3 u^{0}_2\|_{L^\infty_TL^\infty_{xy}}\,\|\lge z \rge^3 \pt_z c^{B,1}\|_{L^\infty_TL^2_{xz}}\\
\leq & C\va^{\frac{7}{4}} \|u^{0}_2\|_{L^\infty_TH^5_{xy}}\,\|\lge z \rge^3 c^{B,1}\|_{L^\infty_TL^2_{x}H^1_z}\\
\leq & C\va^{\frac{7}{4}}.
\end{split}
\ee
By a similar argument used in deriving \eqref{a6}, one can estimate the other terms in $P_1+P_2$ to deduce that
\be\label{a7}
\begin{split}
\|P_1+P_2\|_{L^\infty_TL^2_{xy}}
\leq  C\va^{\frac{7}{4}}.
\end{split}
\ee
Similar arguments used in attaining \eqref{a3} and the assumption $0<\va <1$ lead to
\be\label{a8}
\|P_3\|_{L^\infty_TL^2_{xy}}
\leq  C\va,\qquad
\|P_4\|_{L^\infty_TL^2_{xy}}
\leq  C\va^{\frac{9}{4}}.
\ee
Collecting \eqref{a7} and \eqref{a8}, we obtain
\be\label{a63}
\|g^\va\|_{L^\infty_TL^2_{xy}}
\leq  C\va.
\ee
A direct computation and similar arguments used in deriving \eqref{a7}-\eqref{a8} and the assumption $0<\va <1$ yield
\ben
\|\na(P_1+P_2)\|_{L^\infty_TL^2_{xy}}
\leq  C\va^{\frac{5}{4}},\qquad \|\na P_3\|_{L^\infty_TL^2_{xy}}
\leq  C\va,\qquad \|\na P_4\|_{L^\infty_TL^2_{xy}}
\leq  C\va^{\frac{7}{4}}.
\enn
Then it follows from the above estimates that
\be\label{a62}
\|\na g^\va\|_{L^\infty_TL^2_{xy}}
\leq  C\va.
\ee

We proceed to estimate $\|\na g^\va_t\|_{L^2_TL^2_{xy}}$. 
Indeed, it follows from Taylor's formula, the change of variables $z=\frac{y}{\sqrt{\va}}$, Sobolev embedding inequality, Proposition \ref{p2} and Lemma \ref{l10} that
\be\label{b67}
\begin{split}
&\|\pt_y\pt_t(u^{0}_2-\ol{u^{0}_2}-y\ol{\pt_y u^{0}_2}-\frac{y^2}{2}\ol{\pt_y^2 u^{0}_2})\pt_z c^{B,1}\|_{L^2_TL^2_{xy}}\\
=&\|\pt_t(\pt_yu^{0}_2-\ol{\pt_y u^{0}_2}-y\ol{\pt_y^2 u^{0}_2})\pt_z c^{B,1}\|_{L^2_TL^2_{xy}}\\
\leq &
\va^{\frac{1}{2}}\left\|\frac{(\pt_t\pt_yu^{0}_2-\ol{\pt_t\pt_yu^{0}_2})}{y}\right\|_{L^2_TL^\infty_{xy}}\|z\pt_zc^{B,1}\|_{L^\infty_TL^2_{xy}}
+\va^{\frac{1}{2}}\|\ol{\pt_t\pt_y^2 u^{0}_2}\|_{L^2_TL^\infty_x}\|z\pt_zc^{B,1}\|_{L^\infty_TL^2_{xy}}
\\
\leq &C\va^{\frac{3}{4}}\|\na\pt_t\pt_yu^{0}_2\|_{L^2_TL^\infty_{xy}}
\|\lge z\rge c^{B,1}\|_{L^\infty_TL^2_{x}H^1_z}
+C\va^{\frac{3}{4}}\|\pt_tu^{0}_2\|_{L^2_TH^4_{xy}}
\|\lge z\rge c^{B,1}\|_{L^\infty_TL^2_{x}H^1_z}
\\
\leq &C\va^{\frac{3}{4}}\|\pt_t u^{0}_2\|_{L^2_TH^4_{xy}}
\|\lge z\rge c^{B,1}\|_{L^\infty_TL^2_{x}H^1_z}
\\
\leq &C\va^{\frac{3}{4}}.
\end{split}
\ee
By a similar argument used in attaining \eqref{b67}, one deduces that 
\ben
\|\pt_x\pt_t(u^{0}_2-\ol{u^{0}_2}-y\ol{\pt_y u^{0}_2}-\frac{y^2}{2}\ol{\pt_y^2 u^{0}_2})\pt_z c^{B,1}\|_{L^2_TL^2_{xy}}
\leq &C\va^{\frac{3}{4}},
\enn
which, along with \eqref{b67} yields
\be\label{b68}
\|\na\pt_t(u^{0}_2-\ol{u^{0}_2}-y\ol{\pt_y u^{0}_2}-\frac{y^2}{2}\ol{\pt_y^2 u^{0}_2})\pt_z c^{B,1}\|_{L^2_TL^2_{xy}}
\leq C\va^{\frac{3}{4}}.
\ee
In a similar fashion in attainting \eqref{b68}, one can estimate the other terms in $\|\na\pt_t[(u^{0}_2-\ol{u^{0}_2}-y\ol{\pt_y u^{0}_2}-\frac{y^2}{2}\ol{\pt_y^2 u^{0}_2})\pt_z c^{B,1}]\|_{L^2_TL^2_{xy}}$ to get
\be\label{b69}
\|\na\pt_t[(u^{0}_2-\ol{u^{0}_2}-y\ol{\pt_y u^{0}_2}-\frac{y^2}{2}\ol{\pt_y^2 u^{0}_2})\pt_z c^{B,1}]\|_{L^2_TL^2_{xy}}
\leq C\va^{\frac{3}{4}}.
\ee
By an analogous argument used in deriving \eqref{b69}, we estimate the other terms in $\na\pt_t(P_1+P_2)$, $\na\pt_tP_3$ and $\na\pt_tP_4$ to find that
\ben
\|\na\pt_t(P_1+P_2)\|_{L^2_{T}L^2_{xy}}+\|\na\pt_tP_3\|_{L^2_{T}L^2_{xy}}+\|\na\pt_tP_4\|_{L^2_{T}L^2_{xy}}\leq  C\va^{\frac{3}{4}}.
\enn
Thus,
\ben
\|\na g^\va_t\|_{L^2_TL^2_{xy}}
\leq  C\va^{\frac{3}{4}},
\enn
which, in conjunction with \eqref{a62} and \eqref{a63} completes the proof.

\end{proof}

\begin{lemma}\label{l3} Suppose that the assumptions in Proposition \ref{p3} holds true. 
Then there exists a constant $C$ independent of $\va$ such that
\ben
\|\vec{h}^{\,\va}\|_{L^\infty_TL^2_{xy}}\leq  C\va^{\frac{5}{4}},\qquad \|\vec{h}^{\,\va}_t\|_{L^\infty_TL^2_{xy}}\leq C\va^{\frac{5}{4}},\qquad \|\vec{h}^{\,\va}\|_{L^\infty_TH^1_{xy}}\leq C\va^{\frac{3}{4}}.
\enn
\end{lemma}
\begin{proof}
By \eqref{e3}, \eqref{e11} and the change of variables $z=\frac{y}{\sqrt{\va}}$, we derive from \eqref{e24} that
\be\label{a9}
\begin{split}
\vec{h}^{\,\va}=&\Delta \vec{u}^{\,0}-\pt_t\vec{u}^{\,0}
-\vec{u}^{\,0}\cdot\na \vec{u}^{\,0}-\na p^{0}-m^{0}(0,\lambda)
-(\va\pt_x p^{B,2},\va^{\frac{1}{2}}\pt_z p^{B,2})\\
&-\va^{\frac{1}{2}}m^{B,1}(0,\lambda)
-\va m^{B,2}(0,\lambda)
-\va^{\frac{3}{2}}\xi(0,\lambda)\\
=&-(\va\pt_x p^{B,2}, \lambda\va m^{B,2}+\lambda\va^{\frac{3}{2}}  \xi).
\end{split}
\ee
A direct computation along with the change of variables $z=\frac{y}{\sqrt{\va}}$ and Lemma \ref{l10}-Lemma \ref{l11} leads to
\ben
\begin{split}
\|-\pt_x p^{B,2}\|_{L^\infty_TL^2_{xy}}=\va^{\frac{1}{4}}\|\pt_x p^{B,2}\|_{L^\infty_TL^2_{xz}}
\leq C\va^{\frac{1}{4}},\qquad
\|m^{B,2}\|_{L^\infty_TL^2_{xy}}
\leq C\va^{\frac{1}{4}},\qquad \|\xi\|_{L^\infty_TL^2_{xy}}
\leq C\va^{\frac{1}{4}}.
\end{split}
\enn
%and
%\ben
%\|m^{B,2}\|_{L^\infty_TL^2_{xy}}
%\leq C\va^{\frac{1}{4}},\qquad \|\xi\|_{L^\infty_TL^2_{xy}}
%\leq C\va^{\frac{1}{4}}.
%\enn
Substituting the above estimates into \eqref{a9} and using the assumption $0<\va<1$, we obtain
\be\label{a42}
\|\vec{h}^{\,\va}\|_{L^\infty_TL^2_{xy}}
\leq C\va^{\frac{5}{4}}.
\ee
By a similar argument in deriving \eqref{a42}, one gets
\be\label{b71}
\|\vec{h}_t^{\,\va}\|_{L^\infty_TL^2_{xy}}
\leq C\va^{\frac{5}{4}}.
\ee
From \eqref{a9} and the change of variables $z=\frac{y}{\sqrt{\va}}$, we know that
\be\label{a43}
\begin{split}
&\pt_y \vec{h}^{\,\va}=-(\va^{\frac{1}{2}}\pt_x \pt_z p^{B,2}, \lambda\va^{\frac{1}{2}} \pt_zm^{B,2}+\lambda\va  \pt_z \xi),\\
&\pt_x \vec{h}^{\,\va}=-(\va\pt_x^2 p^{B,2}, \lambda\va \pt_xm^{B,2}+\lambda\va^{\frac{3}{2}}  \pt_x \xi).
\end{split}
\ee
By the assumption $0<\va<1$ and a similar argument used in attaining \eqref{a42}, we estimate each term in \eqref{a43} to deduce that
\ben
\begin{split}
\|\na \vec{h}^{\,\va}\|_{L^\infty_T L^2_{xy}}
\leq& \|\pt_y \vec{h}^{\,\va}\|_{L^\infty_T L^2_{xy}}+\|\pt_x \vec{h}^{\,\va}\|_{L^\infty_T L^2_{xy}}\\
\leq& \va^{\frac{3}{4}}(\| p^{B,2}\|_{L^\infty_T H^2_{xz}}+\lambda\| m^{B,2}\|_{L^\infty_T H^1_{xz}})
+\lambda\va^{\frac{5}{4}}\|\xi\|_{L^\infty_T H^1_{xz}}\\
\leq& C\va^{\frac{3}{4}},
\end{split}
\enn
which, along with \eqref{a42} and \eqref{b71} gives the desired estimate. 
The proof is finished.

\end{proof}

\begin{lemma}\label{l13} Let the assumptions in Proposition \ref{p3} hold.
Then there exists a constant $C$ independent of $\va$, such that
\be\label{a24}
\begin{split}
\frac{d}{dt}\|M^\va\|_{L^2_{xy}}^2+\frac{5}{4}\|\na M^\va\|_{L^2_{xy}}^2
\leq&
C\va \|M^\va\|_{L^\infty_{xy}}^2
\|\na C^{\va}\|_{L^2_{xy}}^2
+C(\|M^\va\|_{L^2_{xy}}^2+\|\vec{U}^{\,\va}\|_{L^2_{xy}}^2+\|\na C^{\va}\|_{L^2_{xy}}^2)\\
&+C\va+\va^{-1}\|f^\va\|_{L^2_{xy}}^2.
\end{split}
\ee
\end{lemma}
\begin{proof}
Testing the first equation of \eqref{e23} with $M^\va$ in $L^2_{xy}$, using integration by parts and \eqref{b32} to get
\be\label{a21}
\begin{split}
\frac{1}{2}&\frac{d}{dt}\|M^\va\|_{L^2_{xy}}^2+\|\na M^\va\|_{L^2_{xy}}^2\\
=&-\va^{\frac{1}{2}}\int_0^\infty\int_{-\infty}^\infty \vec{U}^{\,\va}\cdot \na M^\va M^\va dxdy
-\int_0^\infty\int_{-\infty}^\infty \vec{U}^{\,\va}\cdot \na M^a M^\va dxdy\\
&-\int_0^\infty\int_{-\infty}^\infty \vec{u}^{\,0}\cdot \na M^\va M^\va dxdy
+\va^{\frac{1}{2}}\int_0^\infty\int_{-\infty}^\infty M^{\va} \na C^\va \cdot\na M^\va dxdy\\
&+\int_0^\infty\int_{-\infty}^\infty M^{\va} \na C^a \cdot\na M^\va dxdy
+\int_0^\infty\int_{-\infty}^\infty M^{a} \na C^\va \cdot\na M^\va dxdy
\\
&+\va^{\frac{1}{2}}\int_{-\infty}^\infty \pt_z\xi(x,0,t)\cdot M^\va (x,0,t) dx
+\va^{-\frac{1}{2}}\int_0^\infty\int_{-\infty}^\infty M^{\va} f^\va dxdy\\
:=&\sum_{i=1}^{8}G_i.
\end{split}
\ee
It follows from integration by parts, the facts $\na \cdot \vec{U}^{\,\va}=\na\cdot \vec{u}^{\,0}=0$ and $\vec{U}^{\,\va}(x,0,t)=\vec{u}^{\,0}(x,0,t)=\mathbf{0}$ that
\ben
\begin{split}
G_1=\frac{1}{2}\va^{\frac{1}{2}}\int_0^\infty \int_{-\infty}^{\infty} \na \cdot \vec{U}^{\,\va}|M^\va|^2 \,dxdy=0,
\qquad G_3=\frac{1}{2}\int_0^\infty \int_{-\infty}^{\infty} \na \cdot \vec{u}^{\,0}|M^\va|^2 \,dxdy=0.
\end{split}
\enn
By the Sobolev embedding inequality, change of variables $z=\frac{y}{\sqrt{\va}}$, the assumption $0<\va<1$, Proposition \ref{p2} and Lemma \ref{l10}-Lemma \ref{l11}, one gets
\be\label{a35}
\begin{split}
\|\na M^a(t)\|_{L^4_{xy}}\leq & \|\na m^{0}\|_{L^\infty_TL^4_{xy}}+\va^{\frac{1}{8}}(\va^{\frac{1}{2}}\|\pt_xm^{B,1}\|_{L^\infty_TL^4_{xz}}
+\|\pt_zm^{B,1}\|_{L^\infty_TL^4_{xz}}
+\va\|\pt_xm^{B,2}\|_{L^\infty_TL^4_{xz}})\\
&+\va^{\frac{1}{8}}(\va^{\frac{1}{2}}\|\pt_zm^{B,2}\|_{L^\infty_TL^4_{xz}}
+\va^{\frac{3}{2}}\|\pt_x \xi\|_{L^\infty_TL^4_{xz}}+\va\|\pt_z\xi\|_{L^\infty_TL^4_{xz}})\\
\leq &C(\|m^{0}\|_{L^\infty_TH^2_{xy}}+\|m^{B,1}\|_{L^\infty_TH^2_{xz}}
+\|m^{B,2}\|_{L^\infty_TH^2_{xz}}+\|\xi\|_{L^\infty_TH^2_{xz}})\\
\leq &C
\end{split}
\ee
and
\be\label{a34}
\begin{split}
\|\na C^a(t)\|_{L^\infty_{xy}}\leq C(\|c^{0}\|_{L^\infty_TH^3_{xy}}+\|c^{B,1}\|_{L^\infty_TH^3_{xz}}
+\|c^{B,2}\|_{L^\infty_TH^3_{xz}})
\leq C
\end{split}
\ee
and
\be\label{b72}
\begin{split}
\|M^a(t)\|_{L^\infty_{xy}}
\leq& C(\|m^{0}\|_{L^\infty_TH^2_{xy}}+\|m^{B,1}\|_{L^\infty_TH^2_{xz}}
+\|m^{B,2}\|_{L^\infty_TH^2_{xz}}+\|\xi\|_{L^\infty_TH^2_{xz}})
\leq C,\\
\| C^a(t)\|_{L^\infty_{xy}}\leq& C(\|c^{0}\|_{L^\infty_TH^2_{xy}}+\|c^{B,1}\|_{L^\infty_TH^2_{xz}}
+\|c^{B,2}\|_{L^\infty_TH^2_{xz}})
\leq C
\end{split}
\ee
for each $t\in [0,T]$.
Then it follows from \eqref{a35} that
\ben
\begin{split}
G_2
\leq \|\vec{U}^{\,\va}\|_{L^2_{xy}}
\|\na M^a\|_{L^4_{xy}}
\|M^\va\|_{L^4_{xy}}
\leq&C \|\vec{U}^{\,\va}\|_{L^2_{xy}}
\|M^\va\|_{L^2_{xy}}^{\frac{1}{2}}\|\na M^\va\|_{L^2_{xy}}^{\frac{1}{2}}\\
\leq&\frac{1}{16}\|\na M^\va\|_{L^2_{xy}}^2+\|M^\va\|_{L^2_{xy}}^2+C \|\vec{U}^{\,\va}\|_{L^2_{xy}}^2
.
\end{split}
\enn
The Cauchy-Schwarz inequality yields that
\ben
\begin{split}
G_4
\leq
\frac{1}{16}\|\na M^\va\|_{L^2_{xy}}^2
+C\va\|M^\va\|_{L^\infty_{xy}}^2\|\na C^\va\|_{L^2_{xy}}^2.
\end{split}
\enn
It follows from the Sobolev embedding inequality and \eqref{b72} that
\ben
\begin{split}
G_5\leq& \|M^\va\|_{L^2_{xy}}
\|\na C^a\|_{L^\infty_{xy}}\|\na M^\va\|_{L^2_{xy}}
\leq \frac{1}{16}\|\na M^\va\|_{L^2_{xy}}^2
+C\|M^\va\|_{L^2_{xy}}^2,\\
G_6
\leq&
\|\na M^\va\|_{L^2_{xy}}
\|M^a\|_{L^\infty_{xy}}\|\na C^\va\|_{L^2_{xy}}
\leq 
\frac{1}{16}\|\na M^\va\|_{L^2_{xy}}^2
+C\|\na C^\va\|_{L^2_{xy}}^2
.
\end{split}
\enn
We deduce from the Sobolev embedding inequality that
\be\label{b34}
\begin{split}
G_7
\leq \va^{\frac{1}{2}}\int_{-\infty}^\infty \|\pt_z\xi\|_{L^\infty_z}\cdot\| M^\va\|_{L^\infty_y} dx
\leq& \va^{\frac{1}{2}}\int_{-\infty}^\infty \|\xi\|_{H^2_z}\cdot\| M^\va\|_{H^1_y} dx\\
\leq& C\va^{\frac{1}{2}}\|\xi\|_{L^2_xH^2_{z}}\|M^\va\|_{H^1_{xy}}\\
\leq &\frac{1}{16}\|\na M^\va\|_{L^2_{xy}}^2+\|M^\va\|_{L^2_{xy}}^2+C\va\|\xi\|_{L^2_xH^2_{z}}^2.
\end{split}
\ee
A direct computation gives
\ben
G_8\leq \|M^\va\|_{L^2_{xy}}^2+\va^{-1}\|f^\va\|_{L^2_{xy}}^2.
\enn
Inserting the above estimates for $G_1$-$G_8$ into \eqref{a21} and using Lemma \ref{l11}, one derives \eqref{a24}.
The proof is finished.

\end{proof}

Denote
$\vec{V}^{\va}=\na C^\va$. Then from the second equation of \eqref{e23}, we have
\be\label{a10}
\begin{split}
&\vec{V}^{\,\va}_t+\va^{\frac{1}{2}}\na (\vec{U}^{\,\va}\cdot\vec{V}^{\,\va})+\na (\vec{U}^{\,\va}\cdot\na C^{a})
+\na (\vec{u}^{\,0}\cdot\vec{V}^{\,\va})+\va^{\frac{1}{2}}\na M^\va C^\va +\va^{\frac{1}{2}}M^\va \vec{V}^{\,\va}
\\
=&\va \Delta \vec{V}^{\,\va}+\va^{-\frac{1}{2}}\na g^\va -M^\va \na C^a-\na M^\va C^a -\na M^{a} C^\va -M^a\vec{V}^{\,\va}.
\end{split}
\ee
For $\vec{V}^{\,\va}$ we have the following estimates.
\begin{lemma} Suppose that the assumptions in Proposition \ref{p3} holds true. 
Then there is a constant $C$ independent of $\va$, such that
\be\label{a20}
\begin{split}
&\frac{d}{dt}\|\vec{V}^{\,\va}\|_{L^2_{xy}}^2+\va \|\na\vec{V}^{\,\va}\|_{L^2_{xy}}^2\\
\leq 
&\|\na \vec{U}^{\,\va}\|_{L^2_{xy}}^2
  +C(\|\na\vec{U}^{\,\va}\|_{L^2_{xy}}^2+\va^{\frac{1}{2}} \|M^\va \|_{L^\infty_{xy}}+\va
  \|C^\va\|_{L^2_{xy}}^2+1)\|\vec{V}^{\,\va}\|_{L^2_{xy}}^2
  \\
  &+\frac{1}{4}\|\na M^\va\|_{L^2_{xy}}^2+C(\|M^{\va} \|_{L^2_{xy}}^2 +\|\vec{U}^{\,\va}\|_{L^2_{xy}}^2+\|C^{\va} \|_{L^2_{xy}}^2)+\va^{-1}\|\na g^\va\|_{L^2_{xy}}^2.
\end{split}
\ee
\end{lemma}
\begin{proof}
Taking the $L^2_{xy}$ inner product of \eqref{a10} with $\vec{V}^{\,\va}$ and integrating the resulting equality over $\mathbb{R}^2_{+}$, then using integration by parts, one gets
\be\label{a11}
\begin{split}
&\frac{1}{2}\frac{d}{dt}\|\vec{V}^{\,\va}\|_{L^2_{xy}}^2+\va \|\na\vec{V}^{\,\va}\|_{L^2_{xy}}^2\\
=&-\va\int_{-\infty}^\infty \pt_y V^\va_1(x,0,t)V^\va_1(x,0,t) dx-\va\int_{-\infty}^\infty \pt_y V^\va_2(x,0,t)V^\va_2(x,0,t) dx\\
&-\va^{\frac{1}{2}}\int_{0}^\infty \int_{-\infty}^\infty\na (\vec{U}^{\,\va}\cdot\vec{V}^{\,\va})\cdot \vec{V}^{\,\va}\,dxdy
-\int_{0}^\infty \int_{-\infty}^\infty\na (\vec{U}^{\,\va}\cdot\na C^{a})\cdot \vec{V}^{\,\va}\,dxdy\\
&-\int_{0}^\infty \int_{-\infty}^\infty\na (\vec{u}^{\,0}\cdot\vec{V}^{\,\va})\cdot \vec{V}^{\,\va}\,dxdy
-\va^{\frac{1}{2}}\int_{0}^\infty \int_{-\infty}^\infty\na M^\va \cdot \vec{V}^{\,\va}C^\va\,dxdy\\
&-\va^{\frac{1}{2}}\int_{0}^\infty \int_{-\infty}^\infty M^\va  \vec{V}^{\,\va}\cdot \vec{V}^{\,\va}\,dxdy
-\int_{0}^\infty \int_{-\infty}^\infty M^\va  \na C^a\cdot \vec{V}^{\,\va}\,dxdy\\
&-\int_{0}^\infty \int_{-\infty}^\infty \na M^\va  \cdot \vec{V}^{\,\va}C^a\,dxdy
-\int_{0}^\infty \int_{-\infty}^\infty \na M^a  \cdot \vec{V}^{\,\va}C^\va\,dxdy\\
&-\int_{0}^\infty \int_{-\infty}^\infty M^a  \vec{V}^{\,\va}\cdot \vec{V}^{\,\va}\,dxdy
+\va^{-\frac{1}{2}}\int_{0}^\infty \int_{-\infty}^\infty \na g^\va  \cdot \vec{V}^{\,\va}\,dxdy\\
&:=\sum_{i=1}^{12}H_i.
\end{split}
\ee
By the definition of $\vec{V}^{\,\va}$ and the boundary conditions in \eqref{e23}, one gets
\ben
\begin{split}
H_1
=&-\va\int_{-\infty}^\infty \pt_y \pt_x C^\va(x,0,t)\pt_xC^\va(x,0,t) dx=0,\\
H_2
=&-\va\int_{-\infty}^\infty \pt_y^2 C^\va(x,0,t)\pt_y C^\va(x,0,t) dx=0.
\end{split}
\enn
By integration by parts, the fact $\na \cdot \vec{U}^{\,\va}=0$ and Sobolev embedding inequality we have
\ben
\begin{split}
 H_3
 =&
 -\va^{\frac{1}{2}}\int_{0}^\infty \int_{-\infty}^\infty (\vec{V}^{\,\va}\cdot\na\vec{U}^{\,\va})\cdot
  \vec{V}^{\,\va}\,dxdy
  +\frac{1}{2}\va^{\frac{1}{2}}\int_{0}^\infty \int_{-\infty}^\infty (\na\cdot\vec{U}^{\,\va})
  |\vec{V}^{\,\va}|^{2}\,dxdy\\
  \leq & \va^{\frac{1}{2}}\|\na\vec{U}^{\,\va}\|_{L^2_{xy}}\|\vec{V}^{\,\va}\|_{L^4_{xy}}^2\\
  \leq &\va^{\frac{1}{2}}\|\na\vec{U}^{\,\va}\|_{L^2_{xy}}(\|\vec{V}^{\,\va}\|_{L^2_{xy}}^2
  +\|\vec{V}^{\,\va}\|_{L^2_{xy}}\|\na\vec{V}^{\,\va}\|_{L^2_{xy}})
  \\
  \leq &\frac{1}{2}\va\|\na\vec{V}^{\,\va}\|_{L^2_{xy}}^2
  +C(\|\na\vec{U}^{\,\va}\|_{L^2_{xy}}^2+1)\|\vec{V}^{\,\va}\|_{L^2_{xy}}^2.
\end{split}
\enn
Similarly, it follows from the fact $\na\cdot \vec{u}^{\,0}=0$ and the Sobolev embedding inequality that
\ben
\begin{split}
H_5
=-\int_{0}^\infty \int_{-\infty}^\infty (\vec{V}^{\,\va}\cdot\na\vec{u}^{\,0})\cdot \vec{V}^{\,\va}\,dxdy
\leq \|\na\vec{u}^{\,0}\|_{L^\infty_{xy}}\|\vec{V}^{\,\va}\|_{L^2_{xy}}^2
\leq C\|\vec{u}^{\,0}\|_{H^3_{xy}}\|\vec{V}^{\,\va}\|_{L^2_{xy}}^2.
\end{split}
\enn
A direct computation gives
\be\label{a19}
\begin{split}
H_4
=&-\int_{0}^\infty \int_{-\infty}^\infty(\na \vec{U}^{\,\va}\,\vec{V}^{\,\va}) \cdot\na C^{a}\,dxdy
-\int_{0}^\infty \int_{-\infty}^\infty (\vec{U}^{\,\va}\cdot\na\na C^{a})\cdot \vec{V}^{\,\va}\,dxdy\\
=&-\int_{0}^\infty \int_{-\infty}^\infty(\na \vec{U}^{\,\va}\,\vec{V}^{\,\va}) \cdot\na C^{a}\,dxdy
-\int_{0}^\infty \int_{-\infty}^\infty (\vec{U}^{\,\va}\cdot\na\na c^{0})\cdot \vec{V}^{\,\va}\,dxdy\\
&-\va^{\frac{1}{2}}\int_{0}^\infty \int_{-\infty}^\infty (\vec{U}^{\,\va}\cdot\na\na c^{B,1})\cdot \vec{V}^{\,\va}\,dxdy
-\va\int_{0}^\infty \int_{-\infty}^\infty (\vec{U}^{\,\va}\cdot\na\na c^{B,2})\cdot \vec{V}^{\,\va}\,dxdy.\\
\end{split}
\ee
We next estimate each term in \eqref{a19}. 
One deduces from \eqref{a34} that
\be\label{a36}
\begin{split}
-\int_{0}^\infty \int_{-\infty}^\infty(\na \vec{U}^{\,\va}\vec{V}^{\,\va}) \cdot\na C^{a}\,dxdy
\leq &\frac{1}{4}\|\na \vec{U}^{\,\va}\|_{L^2_{xy}}^2 +\|\na C^{a}\|_{L^\infty_{xy}}^2
\|\vec{V}^{\,\va}\|_{L^2_{xy}}^2\\
\leq &\frac{1}{4}\|\na \vec{U}^{\,\va}\|_{L^2_{xy}}^2 +C
\|\vec{V}^{\,\va}\|_{L^2_{xy}}^2.
\end{split}
\ee
The Sobolev embedding inequality entails that
\be\label{a76}
\begin{split}
-\int_{0}^\infty \int_{-\infty}^\infty (\vec{U}^{\,\va}\cdot\na\na c^{0})\cdot \vec{V}^{\,\va}\,dxdy
\leq& \|\vec{U}^{\,\va}\|_{L^2_{xy}}\|\na^2 c^{0}\|_{L^\infty_{xy}}\|\vec{V}^{\,\va}\|_{L^2_{xy}}\\
\leq& C\|c^0\|_{H^4_{xy}}(\|\vec{U}^{\,\va}\|_{L^2_{xy}}^2+\|\vec{V}^{\,\va}\|_{L^2_{xy}}^2).
\end{split}
\ee
It follows from the change of variables $z=\frac{y}{\sqrt{\va}}$, the fact $\vec{U}^{\,\va}(x,0,t)=\mathbf{0}$ and the H\"{o}lder inequality that
\be\label{a74}
\begin{split}
&-\va^{\frac{1}{2}}\int_{0}^\infty \int_{-\infty}^\infty U^\va_2(x,y,t)\pt^2_y c^{B,1}(x,\frac{y}{\sqrt{\va}},t) V^\va_2(x,y,t)  dxdy\\
=&-\va^{-\frac{1}{2}} \int_{-\infty}^\infty \int_{0}^\infty U^\va_2(x,y,t)\pt^2_z c^{B,1}(x,\frac{y}{\sqrt{\va}},t) V^\va_2 (x,y,t) dydx\\
=&-\va^{-\frac{1}{2}}\int_{-\infty}^\infty\int_{0}^\infty  [\int_0^y\pt_\eta U^\va_2(x,\eta,t)d\eta]\pt^2_z c^{B,1}(x,\frac{y}{\sqrt{\va}},t) V^\va_2 (x,y,t) dydx\\
\leq &\va^{-\frac{1}{2}}\int_{-\infty}^\infty\int_{0}^\infty  \Big\{\int_0^y[\pt_\eta U^\va_2(x,\eta,t)]^2d\eta\Big\}^{\frac{1}{2}}y^{\frac{1}{2}}|\pt^2_z c^{B,1}(x,\frac{y}{\sqrt{\va}},t)| \,|V^\va_2 (x,y,t)| dydx\\
\leq &\va^{-\frac{1}{2}}\Big\{\int_{-\infty}^\infty\int_{0}^\infty  [\pt_\eta U^\va_2(x,\eta,t)]^2d\eta dx\Big\}^{\frac{1}{2}}\\
&\times \Big\{\int_{-\infty}^\infty\Big[\int_{0}^\infty y^{\frac{1}{2}}|\pt^2_z c^{B,1}(x,\frac{y}{\sqrt{\va}},t)| \,|V^\va_2 (x,y,t)| dy\Big]^2 dx\Big\}^{\frac{1}{2}},\\
\end{split}
\ee
where
\ben
\begin{split}
\Big[\int_{0}^\infty y^{\frac{1}{2}}|\pt^2_z c^{B,1}(x,\frac{y}{\sqrt{\va}},t)| \,|V^\va_2 (x,y,t)| dy\Big]^2
\leq &\int_0^\infty y |\pt_z^2 c^{B,1}(x,\frac{y}{\sqrt{\va}},t)|^2 dy\cdot \int_0^\infty
|V^\va_2 (x,y,t)|^2 dy\\
=&\va \int_0^\infty z |\pt_z^2 c^{B,1}(x,z,t)|^2 dz\cdot \int_0^\infty
|V^\va_2 (x,y,t)|^2 dy.
\end{split}
\enn
Substituting the above estimate into \eqref{a74} and using the Sobolev embedding inequality to have
\be\label{a18}
\begin{split}
&-\va^{\frac{1}{2}}\int_{0}^\infty \int_{-\infty}^\infty U^\va_2(x,y,t)\pt^2_yc^{B,1}(x,\frac{y}{\sqrt{\va}},t) V^\va_2(x,y,t)  dxdy\\
\leq &\Big\{\int_{-\infty}^\infty\int_{0}^\infty  [\pt_\eta U^\va_2(x,\eta,t)]^2d\eta dx\Big\}^{\frac{1}{2}}\\
&\times \Big\{\int_{-\infty}^\infty \Big[\int_0^\infty z |\pt_z^2 c^{B,1}(x,z,t)|^2 dz\cdot \int_0^\infty
|V^\va_2 (x,y,t)|^2 dy\Big] dx\Big\}^{\frac{1}{2}}\\
\leq& \|\na \vec{U}^{\,\va}\|_{L^2_{xy}}\| \vec{V}^{\,\va}\|_{L^2_{xy}}\Big(\int_0^\infty z \|\pt_z^2 c^{B,1}(x,z,t)\|_{L^\infty_x}^2 dz\Big)^{\frac{1}{2}}\\
\leq & \frac{1}{32}\|\na \vec{U}^{\,\va}\|_{L^2_{xy}}^2
+C\|\lge z\rge c^{B,2}\|_{H^1_xH^2_z}^2
\| \vec{V}^{\,\va}\|_{L^2_{xy}}^2.
\end{split}
\ee
It follows from the change of variable $z=\frac{y}{\sqrt{\va}}$, the Sobolev embedding inequality and the assumption $0<\va<1$ that
\be\label{a37}
\begin{split}
&-\va^{\frac{1}{2}}\int_{0}^\infty \int_{-\infty}^\infty U^{\va}_2(x,y,t)\pt_x\pt_y c^{B,1}(x,\frac{y}{\sqrt{\va}},t) V^{\va}_1(x,y,t)\,dxdy\\
=& -\int_{0}^\infty \int_{-\infty}^\infty U^{\va}_2(x,y,t)\,\pt_x\pt_z  c^{B,1}(x,\frac{y}{\sqrt{\va}},t)\, V^{\va}_1(x,y,t)\,dxdy\\
\leq & \|\vec{U}^{\,\va}\|_{L^4_{xy}}\|\pt_x\pt_z  c^{B,1}\|_{L^4_{xy}}\|\vec{V}^{\,\va}\|_{L^2_{xy}}\\
= & \va^{\frac{1}{4}}\|\vec{U}^{\,\va}\|_{L^4_{xy}}\|\pt_x\pt_z  c^{B,1}\|_{L^4_{xz}}\|\vec{V}^{\,\va}\|_{L^2_{xy}}\\
\leq & \frac{1}{32}\|\na \vec{U}^{\,\va}\|_{L^2_{xy}}^2
+C \|\pt_x\pt_zc^{B,1}\|_{H^1_{xz}}^2\|\vec{V}^{\,\va}\|_{L^2_{xy}}^2.
\end{split}
\ee
Similar arguments used in attaining \eqref{a37} further lead to
\be\label{a75}
\begin{split}
&-\va^{\frac{1}{2}}\int_{0}^\infty \int_{-\infty}^\infty U^{\va}_1(x,y,t)[\pt_x^2 c^{B,1}(x,\frac{y}{\sqrt{\va}},t) V^{\va}_1(x,y,t)
\,dxdy\\
&-\va^{\frac{1}{2}}\int_{0}^\infty \int_{-\infty}^\infty U^{\va}_1(x,y,t)\pt_x\pt_y  c^{B,1}(x,\frac{y}{\sqrt{\va}},t) V^{\va}_2(x,y,t)\,dxdy\\
\leq & \frac{1}{16}\|\na \vec{U}^{\,\va}\|_{L^2_{xy}}^2
+C (\|\pt^2_xc^{B,1}\|_{H^1_{xz}}^2+\|\pt_x\pt_zc^{B,1}\|_{H^1_{xz}}^2)\|\vec{V}^{\,\va}\|_{L^2_{xy}}^2.
\end{split}
\ee
Collecting \eqref{a18}-\eqref{a75}, we arrive at
\be\label{a38}
\begin{split}
&-\va^{\frac{1}{2}}\int_{0}^\infty \int_{-\infty}^\infty (\vec{U}^{\,\va}\cdot\na\na c^{B,1})\cdot \vec{V}^{\,\va}\,dxdy\\
\leq& \frac{1}{8}\|\na \vec{U}^{\,\va}\|_{L^2_{xy}}^2
+C(\|\lge z\rge c^{B,1}\|_{H^3_xH^1_z}^2+\|\lge z\rge c^{B,1}\|_{H^1_xH^2_z}^2)
\| \vec{V}^{\,\va}\|_{L^2_{xy}}^2.
\end{split}
\ee
Employing a similar argument used in deriving \eqref{a37} and using the assumption $0<\va<1$, one gets
\ben
-\va\int_{0}^\infty \int_{-\infty}^\infty (\vec{U}^{\,\va}\cdot\na\na c^{B,2})\cdot \vec{V}^{\,\va}\,dxdy
\leq \frac{1}{8}\|\na \vec{U}^{\,\va}\|_{L^2_{xy}}^2
+C(\|\lge z\rge c^{B,2}\|_{H^3_xH^1_z}^2+\|\lge z\rge c^{B,2}\|_{H^1_xH^2_z}^2)
\| \vec{V}^{\,\va}\|_{L^2_{xy}}^2,
\enn
which, along with \eqref{a19}-\eqref{a76} and \eqref{a38} gives rise to
\ben
\begin{split}
H_4
\leq & \frac{1}{2}\|\na \vec{U}^{\,\va}\|_{L^2_{xy}}^2
+C(\|c^{0}\|_{H^4_{xy}}+1)
\|\vec{V}^{\,\va}\|_{L^2_{xy}}^2
+C\|c^{0}\|_{H^4_{xy}}\|\vec{U}^{\,\va}\|_{L^2_{xy}}^2\\
&+C(\|\lge z\rge c^{B,1}\|_{H^3_xH^1_z}^2+\|\lge z\rge c^{B,1}\|_{H^1_xH^2_z}^2+\|\lge z\rge c^{B,2}\|_{H^3_xH^1_z}^2+\|\lge z\rge c^{B,2}\|_{H^1_xH^2_z}^2)
\|\vec{V}^{\,\va}\|_{L^2_{xy}}^2.
\end{split}
\enn
A direct computation and the Cauchy-Schwarz inequality lead to
\ben
\begin{split}
H_6+H_{12}
\leq 
 \frac{1}{16}\|\na M^\va\|_{L^2_{xy}}^2
+C\va\|C^\va\|_{L^\infty_{xy}}^2\|\vec{V}^{\,\va}\|_{L^2_{xy}}^2+\|\vec{V}^{\,\va}\|_{L^2_{xy}}^2+\va^{-1}\|\na g^\va\|_{L^2_{xy}}^2
\end{split}
\enn
and
\ben
\begin{split}
H_7
\leq \va^{\frac{1}{2}} \|M^\va \|_{L^\infty_{xy}}\|\vec{V}^{\,\va} \|_{L^2_{xy}}^2
.
\end{split}
\enn
It follows from \eqref{a34} and \eqref{b72} that
\ben
\begin{split}
H_8
\leq \|\na C^a\|_{L^\infty_{xy}}(\|M^\va\|_{L^2_{xy}}^2+\|\vec{V}^{\,\va}\|_{L^2_{xy}}^2)
\leq C(\|M^\va\|_{L^2_{xy}}^2+\|\vec{V}^{\,\va}\|_{L^2_{xy}}^2)
\end{split}
\enn
and that
\ben
\begin{split}
H_9
\leq \|\na M^\va\|_{L^2_{xy}}\|C^a\|_{L^\infty_{xy}}
\|\vec{V}^{\,\va}\|_{L^2_{xy}}
\leq \frac{1}{16}\|\na M^\va\|_{L^2_{xy}}^2+C
\|\vec{V}^{\,\va}\|_{L^2_{xy}}^2
\end{split}
\enn
and that
\ben
\begin{split}
H_{11}
\leq \|M^a\|_{L^\infty_{xy}}\|\vec{V}^{\,\va}\|_{L^2_{xy}}^2\leq C\|\vec{V}^{\,\va}\|_{L^2_{xy}}^2
.
\end{split}
\enn
\eqref{a35} and the Sobolev embedding inequality entail that
\ben
\begin{split}
H_{10}
\leq \|\na M^a\|_{L^4_{xy}}\|\vec{V}^{\,\va}\|_{L^2_{xy}}
\|C^{\va}\|_{L^4_{xy}}
\leq C\|\vec{V}^{\,\va}\|_{L^2_{xy}}^{\frac{3}{2}}
\|C^{\va}\|_{L^2_{xy}}^{\frac{1}{2}}
\leq \|C^{\va}\|_{L^2_{xy}}^2
+C\|\vec{V}^{\,\va}\|_{L^2_{xy}}^2.
\end{split}
\enn
Substituting the above estimates on $H_1$-$H_{12}$ into \eqref{a11} and using Proposition \ref{p2}, Lemma \ref{l10} and Lemma \ref{l11}, we obtain \eqref{a20}. The proof is finished.

\end{proof}

\begin{lemma}
Let the assumptions in Proposition \ref{p3} hold. Then
there exists a constant $C$ independent of $\va$, such that
\be\label{a25}
\begin{split}
&\frac{d}{dt}(\|C^\va\|^2_{L^2_{xy}}+\|\vec{U}^{\,\va}\|_{L^2_{xy}}^2)+\va \|\na C^\va\|_{L^2_{xy}}^2+2\|\na \vec{U}^{\,\va}\|^2_{L^2_{xy}}\\
\leq& C\va^{\frac{1}{2}}\|M^\va\|_{L^\infty_{xy}}\|C^\va\|_{L^2_{xy}}^2+C(\|M^\va\|_{L^2_{xy}}^2+\|C^\va\|_{L^2_{xy}}^2+\|\vec{U}^{\,\va}\|_{L^2_{xy}}^2)\\
&+\va^{-1}\|g^\va\|_{L^2_{xy}}^2+\va^{-1}\|\vec{h}^{\,\va}\|_{L^2_{xy}}^2.
\end{split}
\ee
\end{lemma}
\begin{proof}
Taking the $L^2_{xy}$ inner product of the second equations of \eqref{e23} with $C^\va$, using integration by parts, \eqref{a34} and \eqref{b72} , we have
\ben
\begin{split}
&\frac{1}{2}\frac{d}{dt}\|C^\va\|^2_{L^2_{xy}}+\va \|\na C^\va\|^2_{L^2_{xy}}\\
=&
-\int_0^\infty \int_{-\infty}^\infty \vec{U}^{\,\va}\cdot\na C^a C^\va dxdy
-\va^{\frac{1}{2}}\int_0^\infty \int_{-\infty}^\infty M^{\va}C^{\va} C^\va dxdy
-\int_0^\infty \int_{-\infty}^\infty M^{\va}C^{a} C^\va dxdy\\
&-\int_0^\infty \int_{-\infty}^\infty M^{a}C^{\va} C^\va dxdy
+\va^{-\frac{1}{2}}\int_0^\infty \int_{-\infty}^\infty g^\va C^\va dxdy\\
\leq &\|\na C^a\|_{L^\infty_{xy}}\|\vec{U}^{\,\va}\|_{L^2_{xy}}
\|C^{\va}\|_{L^2_{xy}}
+\va^{\frac{1}{2}}\| M^\va\|_{L^\infty_{xy}}\|C^{\va}\|_{L^2_{xy}}^2
+\| C^a\|_{L^\infty_{xy}}\|M^{\va}\|_{L^2_{xy}}
\|C^{\va}\|_{L^2_{xy}}\\
&+\| M^a\|_{L^\infty_{xy}}\|C^{\va}\|_{L^2_{xy}}^2
+\va^{-\frac{1}{2}}\|C^\va\|_{L^2_{xy}}\|g^\va\|_{L^2_{xy}}\\
\leq& C\va^{\frac{1}{2}}\| M^\va\|_{L^\infty_{xy}}\|C^{\va}\|_{L^2_{xy}}^2 +C(\|M^\va\|_{L^2_{xy}}^2+\|C^\va\|_{L^2_{xy}}^2+\|\vec{U}^{\,\va}\|_{L^2_{xy}}^2)
+\va^{-1}\|g^\va\|_{L^2_{xy}}^2.
\end{split}
\enn

Testing the third equation of \eqref{e23} with $\vec{U}^{\,\va}$ in $L^2_{xy}$, using integration by parts and the Sobolev embedding inequality, one gets
\ben
\begin{split}
&\frac{1}{2}\frac{d}{dt}\|\vec{U}^{\,\va}\|_{L^2_{xy}}^2+\|\na\vec{U}^{\,\va}\|_{L^2_{xy}}^2\\
=&-\int_0^\infty\int_{-\infty}^\infty (\vec{U}^{\,\va}\cdot \na \vec{u}^{\,0})\cdot \vec{U}^{\,\va}dxdy -\lambda\int_0^\infty\int_{-\infty}^\infty M^\va U^\va_2 dxdy
+\va^{-\frac{1}{2}}\int_0^\infty\int_{-\infty}^\infty\vec{h}^{\,\va}\cdot \vec{U}^{\,\va}dxdy\\
\leq & \|\na \vec{u}^{\,0}\|_{L^\infty_{xy}}\|\vec{U}^{\,\va}\|_{L^2_{xy}}^2
+\lambda\|\vec{U}^{\,\va}\|_{L^2_{xy}}\|M^{\va}\|_{L^2_{xy}}
+\va^{-\frac{1}{2}}\|\vec{U}^{\,\va}\|_{L^2_{xy}}\|h^{\va}\|_{L^2_{xy}}\\
\leq & C \|\vec{u}^{\,0}\|_{H^3_{xy}}\|\vec{U}^{\,\va}\|_{L^2_{xy}}^2
+\lambda(\|\vec{U}^{\,\va}\|_{L^2_{xy}}^2+\|M^{\va}\|_{L^2_{xy}}^2)
+\|\vec{U}^{\,\va}\|_{L^2_{xy}}^2+\va^{-1}\|\vec{h}^{\,\va}\|_{L^2_{xy}}^2.
\end{split}
\enn
Collecting the above two estimates and using Proposition \ref{p2}, one gets \eqref{a25}.
The proof is completed.

\end{proof}

\begin{lemma}\label{l4} Suppose that the assumptions in Proposition \ref{p3} hold true. Assume further that
\ben
 \|M^\va\|_{L^\infty_TL^\infty_{xy}}+\|C^\va\|_{L^\infty_TL^\infty_{xy}}+\|\vec{U}^{\,\va}\|_{L^\infty_TH^2_{xy}}<1.
 \enn
Then
there exists a constant $C$ independent of $\va$, such that
\ben
\begin{split}
&\|M^\va\|_{L^\infty_TL^2_{xy}}+\|C^\va\|_{L^\infty_TL^2_{xy}}
+\|\vec{V}^{\,\va}\|_{L^\infty_TL^2_{xy}}
+\|\vec{U}^{\,\va}\|_{L^\infty_TL^2_{xy}}\\
&+\|\na M^\va\|_{L^2_TL^2_{xy}}
+\va^{\frac{1}{2}}\|\vec{V}^{\,\va}\|_{L^2_TH^1_{xy}}
+\|\na \vec{U}^{\,\va}\|_{L^2_TL^2_{xy}}\\
\leq& C\va^{\frac{1}{2}}.
\end{split}
\enn
\end{lemma}
\begin{proof}
Adding \eqref{a25} and \eqref{a20} to \eqref{a24} and using the assumptions $\|M^\va\|_{L^\infty_TL^\infty_{xy}}<1$, $\|C^\va\|_{L^\infty_TL^\infty_{xy}}<1$, $\|\vec{U}^{\,\va}\|_{L^\infty_TH^2_{xy}}<1$ and $0<\va<1$, one gets
\be\label{a28}
\begin{split}
&\frac{d}{dt}(\|M^\va\|_{L^2_{xy}}^2+\|C^\va\|_{L^2_{xy}}^2+\|\vec{U}^{\,\va}\|_{L^2_{xy}}^2
+\|\vec{V}^{\,\va}\|_{L^2_{xy}}^2)\\
&+\|\na M^\va\|_{L^2_{xy}}^2+\va \|\na C^\va\|_{L^2_{xy}}^2+\|\na\vec{U}^{\,\va}\|_{L^2_{xy}}^2
+\va\|\na\vec{V}^{\,\va}\|_{L^2_{xy}}^2\\
\leq &C(\|M^\va\|_{L^2_{xy}}^2+\|C^\va\|_{L^2_{xy}}^2+\|\vec{U}^{\,\va}\|_{L^2_{xy}}^2+\|\vec{V}^{\,\va}\|_{L^2_{xy}}^2)\\
  &+\va^{-1}\|f^\va\|_{L^2_{xy}}^2+\va^{-1}\|g^\va\|_{L^2_{xy}}^2+\va^{-1}\|\na g^\va\|_{L^2_{xy}}^2+\va^{-1}\|\vec{h}^{\,\va}\|_{L^2_{xy}}^2+C\va.
\end{split}
\ee
Then it follows from the Gronwall's inequality and Lemma \ref{l1}- Lemma \ref{l3} that
\be\label{a27}
\begin{split}
\|M^\va(t)\|_{L^2_{xy}}^2+\|C^\va(t)\|_{L^2_{xy}}^2+\|\vec{U}^{\,\va}(t)\|_{L^2_{xy}}^2
+\|\vec{V}^{\,\va}(t)\|_{L^2_{xy}}^2
\leq C\va
\end{split}
\ee
for all $t\in [0,T].$ Integrating \eqref{a28} over $(0,T)$ and using \eqref{a27}, we obtain the desired estimate. The proof is finished.

\end{proof}

\begin{lemma}\label{l12} Let the assumptions in Lemma \ref{l4} hold. Then there exists a constant $C_4$ independent of $\va$, such that
\ben
\begin{split}
\|\vec{U}^{\,\va}_t\|_{L^\infty_TL^2_{xy}}+\|\vec{U}^{\,\va}_t\|_{L^2_TH^1_{xy}}
+\|\na P^\va\|_{L^\infty_TL^2_{xy}}+\|\vec{U}^{\,\va}\|_{L^\infty_TH^2_{xy}}
\leq C_4\va^{\frac{1}{2}}
\end{split}
\enn
and that
\ben
\|\na P^\va\|_{L^2_TH^1_{xy}}+\|\vec{U}^{\,\va}\|_{L^2_TH^3_{xy}}
\leq  C_4\va^{\frac{1}{4}}.
\enn
\end{lemma}
\begin{proof}
%From the second equation of \eqref{e23}, \eqref{a54}, \eqref{a53}, \eqref{a34} and the Sobolev embedding inequality, one gets
%\be\label{a67}
%\begin{split}
%\|&C^\va_t\|_{L^2_{T}L^2_{xy}}\\
%\leq & CT^\frac{1}{2} (\va^{\frac{1}{2}}\|\vec{U}^{\,\va}\|_{L^\infty_{T}H^2_{xy}}\|\na C^\va\|_{L^\infty_{T}L^2_{xy}}
%+\|\vec{U}^{\,\va}\|_{L^\infty_{T}L^2_{xy}}\|\na C^a\|_{L^\infty_{T}L^\infty_{xy}}
%+\|\vec{u}^{\,I,0}\|_{L^\infty_{T}H^2_{xy}}\|\na C^\va\|_{L^\infty_{T}L^2_{xy}})\\
%& +CT^\frac{1}{2} (\va^{\frac{1}{2}}\|M^{\va}\|_{L^\infty_{T}L^\infty_{xy}}\|C^\va\|_{L^\infty_{T}L^2_{xy}}
%+ \|C^a\|_{L^\infty_{T}L^\infty_{xy}}\| M^\va\|_{L^\infty_{T}L^2_{xy}}+ \|M^a\|_{L^\infty_{T}L^\infty_{xy}}\| C^\va\|_{L^\infty_{T}L^2_{xy}})\\
%&+CT^\frac{1}{2} (\va\|\na^2C^{\va}\|_{L^\infty_{T}L^2_{xy}}+\va^{-\frac{1}{2}} \|g^\va\|_{L^\infty_{T}L^2_{xy}})\\
%\leq & C\va^{\frac{1}{2}}.
%\end{split}
%\ee
Taking the $L^2_{xy}$ inner product of the third equation of \eqref{e23} with $\vec{U}^{\,\va}_t$, one gets from integration by parts and the Sobolev embedding inequality that
\ben
\begin{split}
&\frac{1}{2}\frac{d}{dt}\|\na\vec{U}^{\,\va}\|_{L^2_{xy}}^2+\|\vec{U}^{\,\va}_t\|_{L^2_{xy}}^2\\
=&-\int_0^\infty \int_{-\infty}^\infty (\va^{\frac{1}{2}}\vec{U}^{\,\va}\cdot\na\vec{U}^{\,\va}+\vec{U}^{\,\va}\cdot\na\vec{u}^{\,0}
+\vec{u}^{\,0}\cdot\na\vec{U}^{\,\va}-\va^{-\frac{1}{2}}\vec{h}^{\,\va})\cdot \vec{U}^{\,\va}_tdxdy
-\lambda\int_0^\infty \int_{-\infty}^\infty M^\va \,U^{\va}_{2t}dxdy\\
\leq &\frac{1}{2}\|\vec{U}^{\,\va}_t\|_{L^2_{xy}}^2
+C(\va^{\frac{1}{2}}\|\vec{U}^{\,\va}\|_{L^\infty_{xy}}^2\|\na\vec{U}^{\,\va}\|_{L^2_{xy}}^2
+\|\vec{U}^{\,\va}\|_{L^4_{xy}}^2\|\na\vec{u}^{\,0}\|_{L^4_{xy}}^2
+\|\vec{u}^{\,0}\|_{L^\infty_{xy}}^2\|\na\vec{U}^{\,\va}\|_{L^2_{xy}}^2)\\
&+C(\lambda^2 \|M^\va\|_{L^2_{xy}}^2+C\va^{-1}\|\vec{h}^{\,\va}\|_{L^2_{xy}}^2)
\\
\leq &\frac{1}{2}\|\vec{U}^{\,\va}_t\|_{L^2_{xy}}^2
+C\va^{\frac{1}{2}}(\|\vec{U}^{\,\va}\|_{H^2_{xy}}^2+\|\vec{u}^{\,0}\|_{H^2_{xy}}^2)\|\na\vec{U}^{\,\va}\|_{L^2_{xy}}^2
+C(\lambda^2 \|M^\va\|_{L^2_{xy}}^2+\va^{-1}\|\vec{h}^{\,\va}\|_{L^2_{xy}}^2).
\end{split}
\enn
Applying the Gronwall's inequality to the above estimate, using the assumptions $\|\vec{U}^{\,\va}\|_{L^\infty_TH^2_{xy}}^2< 1$, $0<\va<1$, Proposition \ref{p2}, Lemma \ref{l3} and Lemma \ref{l4}, one gets
\be\label{a30}
\begin{split}
\|\na\vec{U}^{\,\va}\|_{L^\infty_TL^2_{xy}}+\|\vec{U}^{\,\va}_t\|_{L^2_TL^2_{xy}}
\leq C_1\va^{\frac{1}{2}},
\end{split}
\ee
with the constant $C_1$ independent of $\va$, depending on $T$.

Differentiating the third equation in \eqref{e23} with respect to $t$ and testing the resulting equation with $\vec{U}^{\,\va}_{t}$ in $L^2_{xy}$, then using integration by parts to have
%\be\label{a32}
%\begin{split}
%\vec{U}^{\,\va}_{tt}+\va^{\frac{1}{2}}\vec{U}^{\,\va}\cdot \nabla \vec{U}^{\,\va}_t
%=&\Delta \vec{U}^{\,\va}_t+\va^{-\frac{1}{2}}\vec{h}^{\,\va}_t
%-\va^{\frac{1}{2}}\vec{U}^{\,\va}_t\cdot \nabla \vec{U}^{\,\va}-\vec{U}^{\,\va}\cdot \nabla \vec{u}^{\,I,0}_t-\vec{U}^{\,\va}_t\cdot \nabla \vec{u}^{\,I,0}\\
%&-\vec{u}^{\,I,0}\cdot \nabla \vec{U}^{\,\va}_t-\vec{u}^{\,I,0}_t\cdot \nabla \vec{U}^{\,\va}
%-\nabla P^\va_t-M^\va_t (0,\lambda).
%\end{split}
%\ee
\be\label{a33}
\begin{split}
&\frac{1}{2}\frac{d}{dt}\|\vec{U}^{\,\va}_t\|_{L^2_{xy}}^2+\|\na\vec{U}^{\,\va}_t\|_{L^2_{xy}}^2\\
=&-\int_0^\infty\int_{-\infty}^\infty (\va^{\frac{1}{2}}\vec{U}^{\,\va}_t\cdot\na \vec{U}^{\,\va}
+\vec{U}^{\,\va}\cdot\na \vec{u}^{\,0}_t+\vec{U}^{\,\va}_t\cdot\na \vec{u}^{\,0}+\vec{u}^{\,0}_t\cdot\na \vec{U}^{\,\va})\cdot\vec{U}^{\,\va}_tdxdy\\
&-\lambda\int_0^\infty\int_{-\infty}^\infty M^\va_t U^{\va}_{2t} dxdy
+\va^{-\frac{1}{2}}\int_0^\infty\int_{-\infty}^\infty \vec{h}^{\,\va}_t\cdot\vec{U}^{\,\va}_{t} dxdy\\
:=&\sum_{i=1}^3 J_i.
\end{split}
\ee
It follows from the Sobolev embedding inequality and the assumption $0<\va<1$ that
\ben
\begin{split}
J_1
\leq&\va^{\frac{1}{2}}\|\na \vec{U}^{\,\va}\|_{L^2_{xy}}\|\vec{U}^{\,\va}_t\|_{L^4_{xy}}^2
+\|\na \vec{u}^{\,0}\|_{L^\infty_{xy}}\|\vec{U}^{\,\va}_t\|_{L^2_{xy}}^2 
+(\|\na \vec{u}^{\,0}_t\|_{L^\infty_{xy}}\|\vec{U}^{\,\va}\|_{L^2_{xy}}
+\|\vec{u}^{\,0}_t\|_{L^\infty_{xy}}\|\na \vec{U}^{\,\va}\|_{L^2_{xy}})
\|\vec{U}^{\,\va}_t\|_{L^2_{xy}}
\\
\leq&\va^{\frac{1}{2}}\|\na \vec{U}^{\,\va}\|_{L^2_{xy}}\|\vec{U}^{\,\va}_t\|_{H^1_{xy}}^2
+\|\vec{u}^{\,0}\|_{H^3_{xy}}\|\vec{U}^{\,\va}_t\|_{L^2_{xy}}^2 
+(\|\vec{u}^{\,0}_t\|_{H^3_{xy}}\|\vec{U}^{\,\va}\|_{L^2_{xy}}
+\|\vec{u}^{\,0}_t\|_{H^2_{xy}}\|\na \vec{U}^{\,\va}\|_{L^2_{xy}})
\|\vec{U}^{\,\va}_t\|_{L^2_{xy}}
\\
\leq &\frac{1}{4}\|\na\vec{U}^{\,\va}_t\|_{L^2_{xy}}^2
+C(\|\vec{u}^{\,0}_t\|_{H^3_{xy}}^2+\|\vec{u}^{\,0}\|_{H^3_{xy}}^2+1)\|\vec{U}^{\,\va}_t\|_{L^2_{xy}}^2
+\|\vec{U}^{\,\va}\|_{H^1_{xy}}^2
.
\end{split}
\enn
The Cauchy-Schwarz inequality yields
\ben
J_3\leq \va^{-1}\|h^\va_t\|_{L^2_{xy}}^2+\|\vec{U}^{\,\va}_t\|_{L^2_{xy}}^2.
\enn
It follows from the first equation of \eqref{e23}, integration by parts, \eqref{a35}-\eqref{b72} and the Soboloev embedding inequality that
\ben
\begin{split}
J_2
=&\lambda\int_0^\infty\int_{-\infty}^\infty U^{\va}_{2t}\,( \va^{\frac{1}{2}}\vec{U}^{\,\va}\cdot \nabla M^\va+\vec{U}^{\,\va}\cdot \nabla M^a+\vec{u}^{\,0}\cdot \nabla M^\va
-\va^{-\frac{1}{2}}f^\va) dxdy\\
&+\lambda\int_0^\infty\int_{-\infty}^\infty \na U^{\va}_{2t}\cdot\na M^\va dxdy
-\lambda\int_0^\infty\int_{-\infty}^\infty \na U^{\va}_{2t}\cdot [\va^{\frac{1}{2}}M^\va \nabla C^\va+M^\va \nabla C^a+M^a \nabla C^\va]dxdy\\
\leq &\lambda \| \vec{U}^{\,\va}_t\|_{L^2_{xy}}(\va^{\frac{1}{2}}\|\vec{U}^{\,\va}\|_{L^\infty_{xy}}\|\nabla M^\va\|_{L^2_{xy}}+\|\vec{U}^{\,\va}\|_{L^4_{xy}}\|\na M^a\|_{L^4_{xy}}+\|\vec{u}^{\,0}\|_{L^\infty_{xy}}\|\na M^{\va}\|_{L^2_{xy}}
+\va^{-\frac{1}{2}}\|f^\va\|_{L^2_{xy}})\\
&+\lambda\|\na \vec{U}^{\,\va}_t\|_{L^2_{xy}}(\|\na M^{\va}\|_{L^2_{xy}}
+\va^{\frac{1}{2}}\|M^{\va}\|_{L^\infty_{xy}}\|\na C^{\va}\|_{L^2_{xy}}
+\|M^{\va}\|_{L^2_{xy}}\|\na C^{a}\|_{L^\infty_{xy}}
+\|M^a\|_{L^\infty_{xy}}\|\na C^{\va}\|_{L^2_{xy}})
\\
\leq &C\lambda^2 \| \vec{U}^{\,\va}_t\|_{L^2_{xy}}^2+C(\va\|\vec{U}^{\,\va}\|_{H^2_{xy}}^2\|\nabla M^\va\|_{L^2_{xy}}^2+\|\vec{U}^{\,\va}\|_{H^1_{xy}}^2+\|\vec{u}^{\,0}\|_{H^2_{xy}}^2
\|\na M^{\va}\|_{L^2_{xy}}^2)
+\va^{-1}\|f^\va\|_{L^2_{xy}}^2\\
&+\frac{1}{4}\|\na \vec{U}^{\,\va}_t\|_{L^2_{xy}}^2+C\lambda^2(\|\na M^{\va}\|_{L^2_{xy}}^2
+\va\|M^\va\|_{L^\infty_{xy}}^2\|\na C^{\va}\|_{L^2_{xy}}^2
+\|M^{\va}\|_{L^2_{xy}}^2
+\|\na C^{\va}\|_{L^2_{xy}}^2).
\end{split}
\enn
Substituting the above estimates for $J_1$-$J_3$ into \eqref{a33}, then applying the Gronwall's inequality to the resulting inequality and using the assumptions $\| M^\va\|_{L^\infty_TL^2_{xy}}<1$, $\| \vec{U}^{\,\va}\|_{L^\infty_TH^2_{xy}}<1$, \eqref{a30}, Proposition \ref{p2}, Lemma \ref{l1}, Lemma \ref{l3} and Lemma \ref{l4}, we conclude that
\be\label{a32}
\|\vec{U}^{\,\va}_t\|_{L^\infty_TL^2_{xy}}+\|\na\vec{U}^{\,\va}_t\|_{L^2_TL^2_{xy}}
\leq C_2\va^{\frac{1}{2}},
\ee
where the constant $C_2$ is independent of $\va$ and depending on $T$.

Moreover, it follows from the third equation of \eqref{e23}, the Sobolev embedding inequality, the assumption $\|\vec{U}^{\,\va}\|_{L^\infty_TH^2_{xy}}<1$, Proposition \ref{p2}, Lemma \ref{l3} and Lemma \ref{l4} that
\be\label{a40}
\begin{split}
&\|\na P^\va\|_{L^\infty_TL^2_{xy}}+\|\vec{U}^{\,\va}\|_{L^\infty_TH^2_{xy}}\\
\leq& \|\vec{U}^{\,\va}_t\|_{L^\infty_TL^2_{xy}}
+\va^{\frac{1}{2}}\|\vec{U}^{\,\va}\|_{L^\infty_TL^\infty_{xy}}\|\na\vec{U}^{\,\va}\|_{L^\infty_TL^2_{xy}}
+\|\vec{U}^{\,\va}\|_{L^\infty_TL^2_{xy}}\|\na\vec{u}^{\,0}\|_{L^\infty_TL^\infty_{xy}}\\
&+\|\vec{u}^{\,0}\|_{L^\infty_TL^\infty_{xy}}\|\na\vec{U}^{\,\va}\|_{L^\infty_TL^2_{xy}}
+\lambda\|M^{\va}\|_{L^\infty_TL^2_{xy}}+\va^{-\frac{1}{2}}\|\vec{h}^{\,\va}\|_{L^\infty_TL^2_{xy}}\\
\leq & \|\vec{U}^{\,\va}_t\|_{L^\infty_TL^2_{xy}}
+C\va^{\frac{1}{2}}\|\vec{U}^{\,\va}\|_{L^\infty_TH^2_{xy}}\|\na\vec{U}^{\,\va}\|_{L^\infty_TL^2_{xy}}
+C\|\vec{u}^{\,0}\|_{L^\infty_TH^3_{xy}}\|\vec{U}^{\,\va}\|_{L^\infty_TH^1_{xy}}\\
&
+\lambda\|M^{\va}\|_{L^\infty_TL^2_{xy}}+\va^{-\frac{1}{2}}\|\vec{h}^{\,\va}\|_{L^\infty_TL^2_{xy}}\\
\leq &C_3\va^{\frac{1}{2}}
\end{split}
\ee
and that
\be\label{a41}
\begin{split}
&\|\na P^\va\|_{L^2_TH^1_{xy}}+\|\vec{U}^{\,\va}\|_{L^2_TH^3_{xy}}\\
\leq & \|\vec{U}^{\,\va}_t\|_{L^2_TH^1_{xy}}
+C\va^{\frac{1}{2}}\|\vec{U}^{\,\va}\|_{L^\infty_TH^2_{xy}}\|\vec{U}^{\,\va}\|_{L^2_TH^2_{xy}}
+C\|\vec{u}^{\,0}\|_{L^2_TH^4_{xy}}\|\vec{U}^{\,\va}\|_{L^\infty_TH^2_{xy}}\\
&
+\lambda\|M^{\va}\|_{L^2_TH^1_{xy}}+\va^{-\frac{1}{2}}\|\vec{h}^{\,\va}\|_{L^2_TH^1_{xy}}\\
\leq &C_3\va^{\frac{1}{4}},
\end{split}
\ee
where the constant $C_3$ is independent of $\va$, depending on $T$. Collecting \eqref{a30}, \eqref{a32}-\eqref{a41} and denoting $C_4=C_1+C_2+C_3$, we derive the desired estimates. The proof is completed.

\end{proof}

\begin{lemma}\label{l6} Let the assumptions in Lemma \ref{l4} hold true. Then there exists a constant $C$ independent of $\va$ such that
\be\label{b74}
\begin{split}
\frac{d}{dt}&\|M^\va_t\|_{L^2_{xy}}^2+\|\na M^\va_t\|_{L^2_{xy}}^2\\
\leq& C(\|\vec{V}^{\,\va}\|_{L^2_{xy}}^2\|\vec{V}^{\,\va}\|_{H^1_{xy}}^2+1)\|M^\va_t\|_{L^2_{xy}}^2+C(\|\vec{U}^{\,\va}_t\|_{L^2_{xy}}^2
+\|\na M^\va\|_{L^2}^2)\\
&+C(\|\vec{V}^{\,\va}_t\|_{L^2_{xy}}^2
+\|\na\vec{U}^{\,\va}\|_{L^2_{xy}}^2+\|\vec{V}^{\,\va}\|_{H^1_{xy}}^2
+\va^{-1}\|f^\va_t\|_{L^2_{xy}}^2+1).
\end{split}
\ee
\end{lemma}
\begin{proof}
Differentiating the first equation of \eqref{e23} with respect to $t$
% we get
%\be\label{a48}
%\begin{split}
%M^\va_{tt}-\Delta M^\va_t=&-\va^{\frac{1}{2}}\vec{U}^{\,\va}_t\cdot \nabla M^\va-\va^{\frac{1}{2}}\vec{U}^{\,\va}\cdot \nabla M^\va_t-\vec{U}^{\,\va}_t\cdot \nabla M^a-\vec{U}^{\,\va}\cdot \nabla M^a_t-\vec{u}^{\,I,0}_t\cdot \nabla M^\va\\
%&-\vec{u}^{\,I,0}\cdot \nabla M^\va_t
%-\nabla\cdot [\va^{\frac{1}{2}}M^\va_t \nabla C^\va+M^\va_t \nabla C^a+M^a_t \nabla C^\va]\\
%&-\nabla\cdot [\va^{\frac{1}{2}}M^\va \nabla C^\va_t+M^\va \nabla C^a_t+M^a \nabla C^\va_t]
%+\va^{-\frac{1}{2}}f^\va_t.
%\end{split}
%\ee
and taking the $L^2_{xy}$ inner product of the resulting equation with $M^{\va}_t$, then using integration by parts and \eqref{b32}, one gets
\ben
\begin{split}
\frac{1}{2}&\frac{d}{dt}\|M^\va_t\|_{L^2_{xy}}^2+\|\na M^\va_t\|_{L^2_{xy}}^2\\
=&-\va^{\frac{1}{2}}\int_0^\infty\int_{-\infty}^\infty \vec{U}^{\,\va}_t\cdot \na M^\va M^\va_t dxdy
-\int_0^\infty\int_{-\infty}^\infty \vec{U}^{\,\va}_t\cdot \na M^a M^\va_t dxdy\\
&-\int_0^\infty\int_{-\infty}^\infty \vec{U}^{\,\va}\cdot \na M^a_t M^\va_t dxdy
-\int_0^\infty\int_{-\infty}^\infty \vec{u}^{\,0}_t\cdot \na M^\va M^\va_t dxdy
\\
&+\va^{\frac{1}{2}}\int_0^\infty\int_{-\infty}^\infty M^{\va}_t \na C^\va \cdot\na M^\va_t dxdy
+\va^{\frac{1}{2}}\int_0^\infty\int_{-\infty}^\infty M^{\va} \na C^\va_t \cdot\na M^\va_t dxdy\\
&+\int_0^\infty\int_{-\infty}^\infty M^{\va}_t \na C^a \cdot\na M^\va_t dxdy
+\int_0^\infty\int_{-\infty}^\infty M^{\va} \na C^a_t \cdot\na M^\va_t dxdy\\
&+\int_0^\infty\int_{-\infty}^\infty M^{a}_t \na C^\va \cdot\na M^\va_t dxdy
+\int_0^\infty\int_{-\infty}^\infty M^{a} \na C^\va_t \cdot\na M^\va_t dxdy
\\
&+\va^{\frac{1}{2}}\int_{-\infty}^\infty \pt_z\xi_t(x,0,t)\cdot M^\va_t (x,0,t) dx
+\va^{-\frac{1}{2}}\int_0^\infty\int_{-\infty}^\infty M^{\va}_t f^\va_t dxdy\\
:=&\sum_{i=1}^{12}R_i.
\end{split}
\enn
Integration by parts, the facts $\vec{U}^{\,\va}_t(x,0,t)=\mathbf{0}$ and $\na\cdot \vec{U}^{\,\va}_t=0$ yield
\ben
\begin{split}
R_1=\va^{\frac{1}{2}}\int_0^\infty\int_{-\infty}^\infty \vec{U}^{\,\va}_t\cdot \na M^\va_t M^\va dxdy
\leq& \frac{1}{20}\|\na M^{\va}_t\|_{L^2_{xy}}^2+C\va\|M^{\va}\|_{L^\infty_{xy}}^2\|\vec{U}^{\,\va}_t\|_{L^2_{xy}}^2.
\end{split}
\enn
We deduce from \eqref{a35} and the Sobolev embedding inequality that
\ben
\begin{split}
R_2\leq& \|\vec{U}^{\,\va}_t\|_{L^2_{xy}}\|\na M^{a}\|_{L^4_{xy}}\|M^{\va}_t\|_{L^4_{xy}}\\
\leq &\|\vec{U}^{\,\va}_t\|_{L^2_{xy}}\|M^{\va}_t\|_{L^2_{xy}}^{\frac{1}{2}}\|\na M^{\va}_t\|_{L^2_{xy}}^{\frac{1}{2}}\\
\leq&  \frac{1}{20}\|\na M^{\va}_t\|_{L^2_{xy}}^2+C\|\vec{U}^{\,\va}_t\|_{L^2_{xy}}^2+C\|M^{\va}_t\|_{L^2_{xy}}^2.
\end{split}
\enn
By the change of variables $z=\frac{y}{\sqrt{\va}}$, the assumption $0<\va<1$, Proposition \ref{p2}, Lemma \ref{l10} and Lemma \ref{l11}, we have
\be\label{a49}
\begin{split}
&\|\na M^a_t(t)\|_{L^2_{xy}}\\
\leq& \|\na m^{0}_t(t)\|_{L^2_{xy}}+C\va^{\frac{1}{4}}(\va^{\frac{1}{2}}\| \pt_x m^{B,1}_t(t)\|_{L^2_{xz}}+\| \pt_z m^{B,1}_t(t)\|_{L^2_{xz}}
+\va\| \pt_x m^{B,2}_t(t)\|_{L^2_{xz}})\\
&+C\va^{\frac{1}{4}}(\va^{\frac{1}{2}}\| \pt_z m^{B,2}_t(t)\|_{L^2_{xz}}
+\va^{\frac{3}{2}}\| \pt_x \xi_t(t)\|_{L^2_{xz}}+\va\| \pt_z \xi_t(t)\|_{L^2_{xz}})\\
\leq& C(\| m^{0}_t(t)\|_{H^1_{xy}}+\|m^{B,1}_t(t)\|_{H^1_{xz}}
+\| m^{B,2}_t(t)\|_{H^1_{xz}}
+\| \xi_t(t)\|_{H^1_{xz}})\\
\leq &C
\end{split}
\ee
and
\be\label{a55}
\begin{split}
\|\na C^a_t(t)\|_{L^2_{xy}}
\leq C(\| c^{0}_t(t)\|_{H^1_{xy}}+\|c^{B,1}_t(t)\|_{H^1_{xz}}
+\| c^{B,2}_t(t)\|_{H^1_{xz}}
)
\leq C
\end{split}
\ee
for each $t\in (0,T]$. Then \eqref{a49} along with the Sobolev embedding inequality and Poincaré inequality entails that
\ben
\begin{split}
R_3\leq  \|\vec{U}^{\,\va}\|_{L^4_{xy}}\|\na M^a_t\|_{L^2_{xy}}
\|M^{\va}_t\|_{L^4_{xy}}
\leq  C\|\na\vec{U}^{\,\va}\|_{L^2_{xy}}
\|M^{\va}_t\|_{H^1_{xy}}
\leq \frac{1}{20}\|\na M^{\va}_t\|_{L^2_{xy}}^2
+C\|M^{\va}_t\|_{L^2_{xy}}^2
+C\|\na\vec{U}^{\,\va}\|_{L^2_{xy}}^2.
\end{split}
\enn
The Soblolev embedding inequality yields that
\ben
\begin{split}
R_4\leq \|u^{\,0}_t\|_{L^\infty_{xy}}\|\na M^\va\|_{L^2_{xy}}\|M^\va_t\|_{L^2_{xy}}
\leq \|M^\va_t\|_{L^2_{xy}}^2
+C\|u^{\,0}_t\|_{H^2_{xy}}^2\|\na M^\va\|_{L^2_{xy}}^2
\end{split}
\enn
and that
\ben
\begin{split}
R_5\leq& \va^{\frac{1}{2}}\|M^\va_t\|_{L^4_{xy}}\|\na C^\va\|_{L^4_{xy}}\|\na M^\va_t\|_{L^2_{xy}}\\
\leq& C\va^{\frac{1}{2}}\|M^\va_t\|_{L^2_{xy}}^{\frac{1}{2}}\|\na M^\va_t\|_{L^2_{xy}}^{\frac{3}{2}}\|\na C^\va\|_{L^2_{xy}}^{\frac{1}{2}}\|\na C^\va\|_{H^1_{xy}}^{\frac{1}{2}}\\
\leq& \frac{1}{20}\|\na M^\va_t\|_{L^2_{xy}}^2+C\va^2\|\na C^\va\|_{L^2_{xy}}^2\|\na C^\va\|_{H^1_{xy}}^2\|M^\va_t\|_{L^2_{xy}}^2.
\end{split}
\enn
It follows from the Sobolev embedding inequality that
\ben
\begin{split}
R_6\leq  \va^{\frac{1}{2}}\|M^\va\|_{L^\infty_{xy}}\|\na C^\va_t\|_{L^2_{xy}}\|\na M^\va_t\|_{L^2_{xy}}
\leq \frac{1}{20}\|\na M^\va_t\|_{L^2_{xy}}^2+C\va\|M^\va\|_{L^\infty_{xy}}^2\|\na C^\va_t\|_{L^2_{xy}}^2.
\end{split}
\enn
\eqref{a34} entails that
\ben
\begin{split}
R_7\leq \|M^\va_t\|_{L^2_{xy}}\|\na C^a\|_{L^\infty_{xy}}\|\na M^\va_t\|_{L^2_{xy}}
\leq \frac{1}{20}\|\na M^\va_t\|_{L^2_{xy}}^2+C\|M^\va_t\|_{L^2_{xy}}^2.
\end{split}
\enn
It follows from \eqref{a55} that
\ben
\begin{split}
R_8\leq \|M^\va\|_{L^\infty_{xy}}\|\na C^a_t\|_{L^2_{xy}}\|\na M^\va_t\|_{L^2_{xy}}
\leq \frac{1}{20}\|\na M^\va_t\|_{L^2_{xy}}^2+C\|M^\va\|_{L^\infty_{xy}}^2.
\end{split}
\enn
\eqref{a49} gives
\ben
\begin{split}
R_9\leq \|M^a_t\|_{L^4_{xy}}\|\na C^\va\|_{L^4_{xy}}\|\na M^\va_t\|_{L^2_{xy}}
\leq \|M^a_t\|_{H^1_{xy}}\|\na C^\va\|_{H^1_{xy}}\|\na M^\va_t\|_{L^2_{xy}}
\leq \frac{1}{20}\|\na M^\va_t\|_{L^2_{xy}}^2+C\|\na C^\va\|_{H^1_{xy}}^2.
\end{split}
\enn
It follows from \eqref{b72} that
\ben
\begin{split}
R_{10}\leq \|M^a\|_{L^\infty_{xy}}\|\na C^\va_t\|_{L^2_{xy}}\|\na M^\va_t\|_{L^2_{xy}}
\leq \frac{1}{20}\|\na M^\va_t\|_{L^2_{xy}}^2+C\|\na C^\va_t\|_{L^2_{xy}}^2.
\end{split}
\enn
By a similar argument used in deriving \eqref{b34}, one deduces that
\be
\begin{split}
R_{11}
\leq \frac{1}{20}\|\na M^\va_t\|_{L^2_{xy}}^2+\|M^\va_t\|_{L^2_{xy}}^2+C\va\|\xi_t\|_{L^2_xH^2_{z}}^2.
\end{split}
\ee
The Cauchy-Schwarz inequality gives
\ben
R_{12}\leq \|M^\va_t\|_{L^2_{xy}}^2+\va^{-1}\|f^\va_t\|_{L^2_{xy}}^2.
\enn
Collecting the above estimates for $R_1$-$R_{12}$ and using the assumptions $\|M^\va\|_{L^\infty_{xy}}<1$, $0<\va<1$, the fact $\vec{V}^{\,\va}=\na C^\va$, Proposition \ref{p2} and Lemma \ref{l11}, we obtain the desired estimate. 
The proof is finished.

\end{proof}

\begin{lemma}\label{l7}
Suppose that the assumptions in Lemma \ref{l4} hold. Then there exists a constant $C$ independent of $\va$ such that
\be\label{b75}
\begin{split}
\frac{d}{dt}&\|\vec{V}^{\,\va}_t\|_{L^2_{xy}}^2+\va \|\na\vec{V}^{\,\va}_t\|_{L^2_{xy}}^2\\
\leq&\frac{1}{4}\|\na M^\va_t\|_{L^2_{xy}}^2
+C(\|\vec{U}^{\,\va}\|_{H^2_{xy}}^2\| C^{a}_t\|_{H^2_{xy}}^2
+\va^{-\frac{3}{4}}\|\vec{U}^{\,\va}_t\|_{H^1_{xy}}^2+\|\vec{V}^{\,\va}\|_{H^1_{xy}}^2+1)
\\
  &+C (\|M^\va_t\|_{L^2_{xy}}^2
 +\|C^\va_t\|_{L^2_{xy}}^2
 +\|\vec{U}^{\,\va}_t\|_{H^1_{xy}}^2\|\vec{V}^{\,\va}\|_{H^1_{xy}}^2
+\|\na M^\va\|_{L^2_{xy}}^2
+\va^{-1}\|\na g^\va_t\|_{L^2_{xy}}^2)
\\
&+C(
\|\vec{U}^{\,\va}\|_{H^3_{xy}}^2+\|\vec{V}^{\,\va}\|_{H^1_{xy}}^2+\| C^{a}_t\|_{L^\infty_{xy}}^2+1)\|\vec{V}^{\,\va}_t\|_{L^2_{xy}}^2
.
\end{split}
\ee
\end{lemma}
\begin{proof}
Differentiating \eqref{a10} with respect to $t$
%\ben
%\begin{split}
%&\vec{V}^{\,\va}_{tt}+\va^{\frac{1}{2}}\na (\vec{U}^{\,\va}_t\cdot\vec{V}^{\,\va})+\va^{\frac{1}{2}}\na (\vec{U}^{\,\va}\cdot\vec{V}^{\,\va}_t)+\na (\vec{U}^{\,\va}_t\cdot\na C^{a})
%+\na (\vec{U}^{\,\va}\cdot\na C^{a}_t)+\na (\vec{u}^{\,I,0}_t\cdot\vec{V}^{\,\va})
%\\
%=&\va \Delta \vec{V}^{\,\va}_t+\va^{-\frac{1}{2}}\na g^\va_t -M^\va_t \na C^a
%-M^\va \na C^a_t-\na M^\va_t C^a-\na M^\va C^a_t -\na M^{a}_t C^\va-\na M^{a} C^\va_t \\
%&-M^a_t\vec{V}^{\,\va} -M^a\vec{V}^{\,\va}_t+\na (\vec{u}^{\,I,0}\cdot\vec{V}^{\,\va}_t)+\va^{\frac{1}{2}}\na M^\va_t C^\va
%+\va^{\frac{1}{2}}\na M^\va C^\va_t +\va^{\frac{1}{2}}M^\va_t \vec{V}^{\,\va}
%+\va^{\frac{1}{2}}M^\va \vec{V}^{\,\va}_t.
%\end{split}
%\enn
%\ben
%\begin{split}
%&\vec{V}^{\,\va}_{tt}+\va^{\frac{1}{2}}\na (\vec{U}^{\,\va}\cdot\vec{V}^{\,\va})_t
%+\na (\vec{U}^{\,\va}\cdot\na C^{a})_t
%+\na (\vec{u}^{\,I,0}\cdot\vec{V}^{\,\va})_t
%+\va^{\frac{1}{2}}(\na M^\va C^\va)_t
%+\va^{\frac{1}{2}}(M^\va \vec{V}^{\,\va})_t 
%\\
%=&\va \Delta \vec{V}^{\,\va}_t+\va^{-\frac{1}{2}}\na g^\va_t-(M^\va \na C^a)_t
%-(\na M^\va C^a)_t -(\na M^{a} C^\va)_t
%-(M^a\vec{V}^{\,\va})_t 
%.
%\end{split}
%\enn
and taking the $L^2_{xy}$ inner product of the resulting equation with $\vec{V}^{\,\va}_{t}$, then using integration by parts to have
\ben
\begin{split}
&\frac{1}{2}\frac{d}{dt}\|\vec{V}^{\,\va}_t\|_{L^2_{xy}}^2+\va \|\na\vec{V}^{\,\va}_t\|_{L^2_{xy}}^2\\
=&-\va\int_{-\infty}^\infty \pt_y V^\va_{1t}(x,0,t)V^\va_{1t}(x,0,t) dx-\va\int_{-\infty}^\infty \pt_y V^\va_{2t}(x,0,t)V^\va_{2t}(x,0,t) dx\\
&-\va^{\frac{1}{2}}\int_{0}^\infty \int_{-\infty}^\infty\na (\vec{U}^{\,\va}\cdot\vec{V}^{\,\va})_t\cdot \vec{V}^{\,\va}_t\,dxdy
-\int_{0}^\infty \int_{-\infty}^\infty\na (\vec{U}^{\,\va}\cdot\na C^{a})_t\cdot \vec{V}^{\,\va}_t\,dxdy\\
&-\int_{0}^\infty \int_{-\infty}^\infty\na (\vec{u}^{\,0}\cdot\vec{V}^{\,\va})_t\cdot \vec{V}^{\,\va}_t\,dxdy
-\va^{\frac{1}{2}}\int_{0}^\infty \int_{-\infty}^\infty(\na M^\va C^\va)_t \cdot \vec{V}^{\,\va}_t \,dxdy\\
&-\va^{\frac{1}{2}}\int_{0}^\infty \int_{-\infty}^\infty (M^\va  \vec{V}^{\,\va})_t\cdot \vec{V}^{\,\va}_t\,dxdy
-\int_{0}^\infty \int_{-\infty}^\infty (M^\va  \na C^a)_t\cdot \vec{V}^{\,\va}_t\,dxdy\\
&-\int_{0}^\infty \int_{-\infty}^\infty (\na M^\va C^a)_t  \cdot \vec{V}^{\,\va}_t\,dxdy
-\int_{0}^\infty \int_{-\infty}^\infty (\na M^a C^\va)_t \cdot \vec{V}^{\,\va}_t\,dxdy\\
&-\int_{0}^\infty \int_{-\infty}^\infty (M^a  \vec{V}^{\,\va})_t\cdot \vec{V}^{\,\va}_t\,dxdy
+\va^{-\frac{1}{2}}\int_{0}^\infty \int_{-\infty}^\infty \na g^\va_t  \cdot \vec{V}^{\,\va}_t\,dxdy\\
&:=\sum_{i=1}^{12}L_i.
\end{split}
\enn
 By the definition of $\vec{V}^{\,\va}=\na C^\va$, the second equation and boundary conditions in \eqref{e23}, we have
 \ben
 \begin{split}
 L_1
 =-\va\int_{-\infty}^\infty \pt_x\pt_y C^\va_t(x,0,t)\pt_x C^\va_t(x,0,t) dx=0,\quad
 L_2
 =-\va\int_{-\infty}^\infty \pt_y^2 C^\va_t(x,0,t)\pt_y C^\va_t(x,0,t) dx=0.
 \end{split}
 \enn
It follows from integration by parts, the facts $\na\cdot \vec{U}^{\,\va}(x,y,t)=0$, $\vec{U}^{\,\va}(x,0,t)=\vec{U}^{\,\va}_t(x,0,t)=\mathbf{0}$ and the Sobolev embedding inequality that
\ben
\begin{split}
L_3=&\frac{1}{2}\va^{\frac{1}{2}}\int_{0}^\infty \int_{-\infty}^\infty (\na\cdot\vec{U}^{\,\va})\,|\vec{V}^{\,\va}_t|^2\,dxdy
-\va^{\frac{1}{2}}\int_{0}^\infty \int_{-\infty}^\infty (\vec{V}^{\,\va}_t\cdot\na\vec{U}^{\,\va})\cdot \vec{V}^{\,\va}_t\,dxdy\\
&+\va^{\frac{1}{2}}\int_{0}^\infty \int_{-\infty}^\infty (\vec{U}^{\,\va}_t\cdot\vec{V}^{\,\va})\, (\na\cdot\vec{V}^{\,\va}_t)\,dxdy\\
=&-\va^{\frac{1}{2}}\int_{0}^\infty \int_{-\infty}^\infty (\vec{V}^{\,\va}_t\cdot\na\vec{U}^{\,\va})\cdot \vec{V}^{\,\va}_t\,dxdy
+\va^{\frac{1}{2}}\int_{0}^\infty \int_{-\infty}^\infty (\vec{U}^{\,\va}_t\cdot\vec{V}^{\,\va})\, (\na\cdot\vec{V}^{\,\va}_t)\,dxdy\\
\leq&\va^{\frac{1}{2}} \|\na\vec{U}^{\,\va}\|_{L^\infty_{xy}}\|\vec{V}^{\,\va}_t\|_{L^2_{xy}}^2
+\va^{\frac{1}{2}} \|\vec{U}^{\,\va}_t\|_{L^4_{xy}}\|\vec{V}^{\,\va}\|_{L^4_{xy}}
\|\na\vec{V}^{\,\va}_t\|_{L^2_{xy}}\\
\leq& \frac{1}{4}\va\|\na\vec{V}^{\,\va}_t\|_{L^2_{xy}}^2
+C\|\vec{U}^{\,\va}_t\|_{H^1_{xy}}^2\|\vec{V}^{\,\va}\|_{H^1_{xy}}^2
+C\va^{\frac{1}{2}} \|\vec{U}^{\,\va}\|_{H^3_{xy}}\|\vec{V}^{\,\va}_t\|_{L^2_{xy}}^2.
\end{split}
\enn
By a similar argument used in deriving \eqref{a35},  the assumption $0<\va<1$, Proposition \ref{p2}, Lemma \ref{l10} and Lemma \ref{l11}, one gets
\ben
\begin{split}
\|\na^2 C^{a}(t)\|_{L^4_{xy}}\leq C(\|c^{0}(t)\|_{H^3_{xy}}+\va^{-\frac{3}{8}}\|c^{B,1}(t)\|_{H^3_{xz}}+\va^{\frac{1}{8}}\|c^{B,2}(t)\|_{H^3_{xz}}
)
\leq C \va^{-\frac{3}{8}}
\end{split}
\enn
for each $t\in [0,T]$, which along with
the Sobolev embedding inequality and \eqref{a34} entails that
\ben
\begin{split}
L_4=&-\int_{0}^\infty \int_{-\infty}^\infty \na \vec{U}^{\,\va}\cdot\na C^{a}_t\cdot \vec{V}^{\,\va}_t\,dxdy
-\int_{0}^\infty \int_{-\infty}^\infty\vec{U}^{\,\va}\cdot\na^2 C^{a}_t\cdot \vec{V}^{\,\va}_t\,dxdy\\
&-\int_{0}^\infty \int_{-\infty}^\infty\na \vec{U}^{\,\va}_t\cdot\na C^{a}\cdot \vec{V}^{\,\va}_t\,dxdy
-\int_{0}^\infty \int_{-\infty}^\infty\vec{U}^{\,\va}_t\cdot\na^2 C^{a}\cdot \vec{V}^{\,\va}_t\,dxdy\\
\leq &\|\na \vec{U}^{\,\va}\|_{L^4_{xy}}\|\na C^{a}_t\|_{L^4_{xy}}\|\vec{V}^{\,\va}_t\|_{L^2_{xy}}
+\|\vec{U}^{\,\va}\|_{L^\infty_{xy}}\|\na^2 C^{a}_t\|_{L^2_{xy}}\|\vec{V}^{\,\va}_t\|_{L^2_{xy}}\\
+&\|\na \vec{U}^{\,\va}_t\|_{L^2_{xy}}\|\na C^{a}\|_{L^\infty_{xy}}\|\vec{V}^{\,\va}_t\|_{L^2_{xy}}
+\|\vec{U}^{\,\va}_t\|_{L^4_{xy}}\|\na^2 C^{a}\|_{L^4_{xy}}\|\vec{V}^{\,\va}_t\|_{L^2_{xy}}\\
\leq &C(\|\vec{U}^{\,\va}\|_{H^2_{xy}}^2\| C^{a}_t\|_{H^2_{xy}}^2
+\va^{-\frac{3}{4}}\|\vec{U}^{\,\va}_t\|_{H^1_{xy}}^2+\|\vec{U}^{\,\va}_t\|_{H^1_{xy}}^2)+\|\vec{V}^{\,\va}_t\|_{L^2_{xy}}^2.
\end{split}
\enn
From integration by parts, the fact $\na\cdot\vec{u}^{\,0}(x,y,t)=0$, $\vec{u}^{\,0}(x,0,t)=\mathbf{0}$ and the Sobolev embedding inequality, one deduces that
\ben
\begin{split}
L_5
=&-\int_{0}^\infty \int_{-\infty}^\infty(\vec{V}^{\,\va}_t\cdot\na\vec{u}^{\,0})\cdot \vec{V}^{\,\va}_t\,dxdy
-\int_{0}^\infty \int_{-\infty}^\infty (\vec{V}^{\,\va}\cdot\na\vec{u}^{\,0}_t)\cdot \vec{V}^{\,\va}_t\,dxdy\\
&-\int_{0}^\infty \int_{-\infty}^\infty (\vec{u}^{\,0}_t\cdot\na\vec{V}^{\,\va})\cdot \vec{V}^{\,\va}_t\,dxdy\\
\leq & \|\na\vec{u}^{\,0}\|_{L^\infty_{xy}}\|\vec{V}^{\,\va}_t\|_{L^2_{xy}}^2
+\|\vec{V}^{\,\va}\|_{L^2_{xy}}\|\na\vec{u}^{\,0}_t\|_{L^\infty_{xy}}\|\vec{V}^{\,\va}_t\|_{L^2_{xy}}
+\|\vec{u}^{\,0}_t\|_{L^\infty_{xy}}\|\na\vec{V}^{\,\va}\|_{L^2_{xy}}\|\vec{V}^{\,\va}_t\|_{L^2_{xy}}\\
\leq & C(\|\vec{u}^{\,0}\|_{H^3_{xy}}+\|\vec{u}^{\,0}_t\|_{H^3_{xy}}^2+1)\|\vec{V}^{\,\va}_t\|_{L^2_{xy}}^2
+\|\vec{V}^{\,\va}\|_{H^1_{xy}}^2.
\end{split}
\enn
Integration by parts, along with the Sobolev embedding inequality leads to
\ben
\begin{split}
L_6=&-\va^{\frac{1}{2}}\int_{0}^\infty \int_{-\infty}^\infty\na M^\va_t C^\va \cdot \vec{V}^{\,\va}_t \,dxdy
+\va^{\frac{1}{2}}\int_{0}^\infty \int_{-\infty}^\infty M^\va \na C^\va_t \cdot \vec{V}^{\,\va}_t \,dxdy\\
&+\va^{\frac{1}{2}}\int_{0}^\infty \int_{-\infty}^\infty M^\va C^\va_t \cdot \na\cdot\vec{V}^{\,\va}_t \,dxdy\\
\leq&
\va^{\frac{1}{2}}\|\na M^\va_t\|_{L^2_{xy}} \|C^\va\|_{L^\infty_{xy}}\|\vec{V}^{\,\va}_t\|_{L^2_{xy}}
+\va^{\frac{1}{2}}\|M^\va\|_{L^\infty_{xy}}(\|\na C^\va_t\|_{L^2_{xy}}\|\vec{V}^{\,\va}_t\|_{L^2_{xy}}
+\| C^\va_t\|_{L^2_{xy}}\|\na\vec{V}^{\,\va}_t\|_{L^2_{xy}})\\
\leq &\frac{1}{4}\va\|\na \vec{V}^{\,\va}_t\|_{L^2_{xy}}^2+\frac{1}{12}\va\|\na M^\va_t\|_{L^2_{xy}}^2
+C(\va^{\frac{1}{2}}\|M^\va\|_{L^\infty_{xy}}+\|C^\va\|_{L^\infty_{xy}}^2)\|\vec{V}^{\,\va}_t\|_{L^2_{xy}}^2
+\| M^\va\|_{L^\infty_{xy}}^2\|C^\va_t\|_{L^2_{xy}}^2.
\end{split}
\enn
It follows from the Sobolev embedding inequality that
\ben
\begin{split}
L_7\leq& \va^{\frac{1}{2}}\|M^\va\|_{L^\infty_{xy}}\|\vec{V}^{\,\va}_t\|_{L^2_{xy}}^2
+\va^{\frac{1}{2}} \|M^\va_t\|_{L^4_{xy}}\|\vec{V}^{\,\va}\|_{L^4_{xy}}\|\vec{V}^{\,\va}_t\|_{L^2_{xy}}\\
\leq&\frac{1}{12}\|M^\va_t\|_{H^1_{xy}}^2 +C(\va^{\frac{1}{2}}\|M^\va\|_{L^\infty_{xy}}+\va \|\vec{V}^{\,\va}\|_{H^1_{xy}}^2)\|\vec{V}^{\,\va}_t\|_{L^2_{xy}}^2
.
\end{split}
\enn
\eqref{a34} and \eqref{a55} yield
\ben
\begin{split}
L_8
\leq &
\|M^\va_t\|_{L^2_{xy}}\|\na C^a\|_{L^\infty_{xy}} \|\vec{V}^{\,\va}_t\|_{L^2_{xy}}
+\|M^\va\|_{L^\infty_{xy}}\|\na C^a_t\|_{L^2_{xy}} \|\vec{V}^{\,\va}_t\|_{L^2_{xy}}\\
\leq & C(\|M^\va_t\|_{L^2_{xy}}^2+\|\vec{V}^{\,\va}_t\|_{L^2_{xy}}^2 +\|M^\va\|_{L^\infty_{xy}}^2).
\end{split}
\enn
The Sobolev embedding inequality and \eqref{b72} entail that
%\be\label{b37}
%\begin{split}
%\|C^a_t(t)\|_{L^\infty_{xy}}\leq \|C^a_t(t)\|_{L^\infty_{xz}}
%\leq C(\|c^{I,0}_t\|_{L^\infty_TH^2_{xy}}+\|c^{B,1}_t\|_{L^\infty_TH^2_{xz}}
%+\|c^{B,2}_t\|_{L^\infty_TH^2_{xz}})\leq C
%\end{split}
%\ee
%for each $t\in [0,T]$, which along with \eqref{b72} entails that
\ben
\begin{split}
L_9\leq&
\|\na M^\va_t\|_{L^2_{xy}} \|C^a\|_{L^\infty_{xy}} \|\vec{V}^{\,\va}_t\|_{L^2_{xy}}
+\|\na M^\va\|_{L^2_{xy}} \|C^a_t\|_{L^\infty_{xy}} \|\vec{V}^{\,\va}_t\|_{L^2_{xy}}\\
\leq& \frac{1}{12}\|\na M^\va_t\|_{L^2_{xy}}^2 +\|\na M^\va\|_{L^2_{xy}}^2
+C(1+\|C^a_t\|_{L^\infty_{xy}}^2) \|\vec{V}^{\,\va}_t\|_{L^2_{xy}}^2.
\end{split}
\enn
 \eqref{a49} and \eqref{a35} lead to
\ben
\begin{split}
L_{10}\leq &
\|\na M^a_t\|_{L^2_{xy}} \|C^\va\|_{L^\infty_{xy}} \|\vec{V}^{\,\va}_t\|_{L^2_{xy}}
+\|\na M^a\|_{L^4_{xy}} \|C^\va_t\|_{L^4_{xy}} \|\vec{V}^{\,\va}_t\|_{L^2_{xy}}\\
\leq & C(\|C^\va\|_{L^\infty_{xy}} \|\vec{V}^{\,\va}_t\|_{L^2_{xy}}
+\|C^\va_t\|_{H^1_{xy}} \|\vec{V}^{\,\va}_t\|_{L^2_{xy}})\\
\leq & C (\|C^\va\|_{L^\infty_{xy}}^2+ \|\vec{V}^{\,\va}_t\|_{L^2_{xy}}^2
+\|C^\va_t\|_{L^2_{xy}}^2).
\end{split}
\enn
\eqref{a49}, along with \eqref{b72} gives
\ben
\begin{split}
L_{11}\leq&
\|M^a_t\|_{L^4_{xy}}  \|\vec{V}^{\,\va}\|_{L^4_{xy}} \|\vec{V}^{\,\va}_t\|_{L^2_{xy}}
+\|M^a\|_{L^\infty_{xy}}  \|\vec{V}^{\,\va}_t\|_{L^2_{xy}}^2 \\
\leq & C\|M^a_t\|_{H^1_{xy}}  \|\vec{V}^{\,\va}\|_{H^1_{xy}} \|\vec{V}^{\,\va}_t\|_{L^2_{xy}}
+C  \|\vec{V}^{\,\va}_t\|_{L^2_{xy}}^2 \\
\leq &C(\|\vec{V}^{\,\va}\|_{H^1_{xy}}^2+ \|\vec{V}^{\,\va}_t\|_{L^2_{xy}}^2).
\end{split}
\enn
The Cauchy-Schwarz inequality entails that
\ben
L_{12}\leq
\va^{-1}\|\na g^\va_t\|_{L^2_{xy}}^2+\|\vec{V}^{\,\va}_t\|_{L^2_{xy}}^2.
\enn
Collecting the above estimates for $L_1$-$L_{12}$, and using the assumption $0<\va<1$, $\| M^\va\|_{L^\infty_TL^\infty_{xy}}<1$, $\| C^\va\|_{L^\infty_TL^\infty_{xy}}<1$ and Proposition \ref{p2}, one gets \eqref{b75}. 
The proof is finished.

\end{proof}

\begin{lemma}\label{l14} Let the assumptions in Lemma \ref{l4} hold true. Then there exists a constant C independent of $\va$, such that
\ben
\|M^{\va}_t\|_{L^\infty_TL^2_{xy}}^2+\|\na M^{\va}_t\|_{L^2_TL^2_{xy}}^2+
\|\vec{V}^{\,\va}_t\|_{L^\infty_TL^2_{xy}}^2+\va \|\na\vec{V}^{\,\va}_t\|_{L^2_TL^2_{xy}}^2
+\|\vec{U}^{\,\va}_t\|_{L^\infty_TH^1_{xy}}+\|\vec{U}^{\,\va}_{tt}\|_{L^2_TL^2_{xy}}
\leq C.
\enn
\end{lemma}
\begin{proof}
Adding \eqref{b74} to \eqref{b75}, we arrive at
\be\label{b76}
\begin{split}
\frac{d}{dt}&(\|M^\va_t\|_{L^2_{xy}}^2+\|\vec{V}^{\,\va}_t\|_{L^2_{xy}}^2)+\|\na M^\va_t\|_{L^2_{xy}}^2+\va \|\na\vec{V}^{\,\va}_t\|_{L^2_{xy}}^2\\
\leq& C(\|\vec{V}^{\,\va}\|_{L^2_{xy}}^2\|\vec{V}^{\,\va}\|_{H^1_{xy}}^2+\|\vec{U}^{\,\va}\|_{H^3_{xy}}^2+\|\vec{V}^{\,\va}\|_{H^1_{xy}}^2+\| C^{a}_t\|_{L^\infty_{xy}}^2+1)(\|M^\va_t\|_{L^2_{xy}}^2+\|\vec{V}^{\,\va}_t\|_{L^2_{xy}}^2)\\
&+C(
\|\na\vec{U}^{\,\va}\|_{L^2_{xy}}^2+\|\vec{V}^{\,\va}\|_{H^1_{xy}}^2
+\|\vec{U}^{\,\va}\|_{H^2_{xy}}^2\| C^{a}_t\|_{H^2_{xy}}^2
+\va^{-\frac{3}{4}}\|\vec{U}^{\,\va}_t\|_{H^1_{xy}}^2+1)
\\
  &+C (
 \|C^\va_t\|_{L^2_{xy}}^2
 +\|\vec{U}^{\,\va}_t\|_{H^1_{xy}}^2\|\vec{V}^{\,\va}\|_{H^1_{xy}}^2
+\|\na M^\va\|_{L^2_{xy}}^2
+\va^{-1}\|f^\va_t\|_{L^2_{xy}}^2
+\va^{-1}\|\na g^\va_t\|_{L^2_{xy}}^2)
.
\end{split}
\ee
For fixed $\va>0$, it follows from the change of variables $z=\frac{y}{\sqrt{\va}}$ that
\ben
\begin{split}
\|\pt^2_y c^{B,1}_t(x,\frac{y}{\sqrt{\va}},t)\|_{L^2_TL^2_{xy}}
=\va^{-\frac{3}{4}}\|\pt^2_z c^{B,1}_t(x,z,t)\|_{L^2_TL^2_{xz}}
\leq \va^{-\frac{3}{4}}\| c^{B,1}_t(x,z,t)\|_{L^2_TH^2_{xz}}.
\end{split}
\enn
By a similar argument used above one can estimate the other terms in $\| \na^2 c^{B,1}_t(x,\frac{y}{\sqrt{\va}},t)\|_{L^2_TL^2_{xy}}$ to deduce that
\ben
\begin{split}
\| \na^2 c^{B,1}_t(x,\frac{y}{\sqrt{\va}},t)\|_{L^2_TL^2_{xy}}
\leq C(1+\va^{-\frac{1}{4}}+\va^{-\frac{3}{4}})\| c^{B,1}_t(x,z,t)\|_{L^2_TH^2_{xz}}
\end{split}
\enn
and thus
\be\label{a64}
\begin{split}
\| c^{B,1}_t(x,\frac{y}{\sqrt{\va}},t)\|_{L^2_TH^2_{xy}}
\leq C\va^{-\frac{3}{4}}\| c^{B,1}_t\|_{L^2_TH^2_{xz}}
\end{split}
\ee
due to the fact $0<\va<1$. In a similar fashion as obtaining \eqref{a64}, one gets
\be\label{b80}
\| c^{B,2}_t(x,\frac{y}{\sqrt{\va}},t)\|_{L^2_TH^2_{xy}}
\leq C\va^{-\frac{3}{4}}\| c^{B,2}_t\|_{L^2_TH^2_{xz}}.
\ee
Then it follows from \eqref{a64}, \eqref{b80}, Proposition \ref{p2}, Lemma \ref{l10}, Lemma \ref{l11} and the assumption $0<\va<1$ that
\be\label{a65}
\begin{split}
\|C^a_t\|_{L^2_TH^2_{xy}}
\leq& C(\|c^0_t\|_{L^2_TH^2_{xy}}+\va^{\frac{1}{2}}\|c^{B,1}_t\|_{L^2_TH^2_{xy}}+\va\|c^{B,2}_t\|_{L^2_TH^2_{xy}})\\
\leq& C(\|c^0_t\|_{L^2_TH^2_{xy}}+\va^{-\frac{1}{4}}\|c^{B,1}_t\|_{L^2_TH^2_{xy}}+\va^{\frac{1}{4}}\|c^{B,2}_t\|_{L^2_TH^2_{xy}})\\
\leq& C\va^{-\frac{1}{4}}.
\end{split}
\ee
 On the other hand, the change of variables $z=\frac{y}{\sqrt{\va}}$, Sobolev embedding inequality and assumption $0<\va<1$ lead to
\be\label{b78}
\begin{split}
\|C^a_t\|_{L^2_TL^\infty_{xy}}
\leq & C(\|c^0_t\|_{L^2_TL^\infty_{xy}}+\va^{\frac{1}{2}}\|c^{B,1}_t\|_{L^2_TL^\infty_{xz}}+\va\|c^{B,2}_t\|_{L^2_TL^\infty_{xz}})\\
\leq& C
(\|c^{0}_t\|_{L^2_TH^2_{xy}}
+\va^{\frac{1}{2}}\|c^{B,1}_t\|_{L^2_TH^2_{xz}}
+\va\|c^{B,2}_t\|_{L^2_TH^2_{xz}}
)\\
\leq& C.
\end{split}
\ee
 Applying the Gronwall's inequality to \eqref{b76} and using \eqref{a65}, \eqref{b78}, Lemma \ref{l1}, Lemma \ref{l2}, Lemma \ref{l4} and Lemma \ref{l12}, we obtain
\be\label{b79}
\|M^{\va}_t\|_{L^\infty_TL^2_{xy}}^2+\|\na M^{\va}_t\|_{L^2_TL^2_{xy}}^2+
\|\vec{V}^{\,\va}_t\|_{L^\infty_TL^2_{xy}}^2+\va \|\na\vec{V}^{\,\va}_t\|_{L^2_TL^2_{xy}}^2
\leq C.
\ee

Differentiating the third equation in \eqref{e23} with respect to $t$ and testing the resulting equation with $\vec{U}^{\,\va}_{tt}$ in $L^2_{xy}$, then using integration by parts to have
%\be\label{a32}
%\begin{split}
%\vec{U}^{\,\va}_{tt}+\va^{\frac{1}{2}}\vec{U}^{\,\va}\cdot \nabla \vec{U}^{\,\va}_t
%=&\Delta \vec{U}^{\,\va}_t+\va^{-\frac{1}{2}}\vec{h}^{\,\va}_t
%-\va^{\frac{1}{2}}\vec{U}^{\,\va}_t\cdot \nabla \vec{U}^{\,\va}-\vec{U}^{\,\va}\cdot \nabla \vec{u}^{\,I,0}_t-\vec{U}^{\,\va}_t\cdot \nabla \vec{u}^{\,I,0}\\
%&-\vec{u}^{\,I,0}\cdot \nabla \vec{U}^{\,\va}_t-\vec{u}^{\,I,0}_t\cdot \nabla \vec{U}^{\,\va}
%-\nabla P^\va_t-M^\va_t (0,\lambda).
%\end{split}
%\ee
\be\label{b77}
\begin{split}
&\frac{1}{2}\frac{d}{dt}\|\na\vec{U}^{\,\va}_t\|_{L^2_{xy}}^2+\|\vec{U}^{\,\va}_{tt}\|_{L^2_{xy}}^2\\
=&-\int_0^\infty\int_{-\infty}^\infty (\va^{\frac{1}{2}}\vec{U}^{\,\va}_t\cdot\na \vec{U}^{\,\va}
+\vec{U}^{\,\va}\cdot\na \vec{u}^{\,0}_t+\vec{U}^{\,\va}_t\cdot\na \vec{u}^{\,0}+\vec{u}^{\,0}_t\cdot\na \vec{U}^{\,\va})\cdot\vec{U}^{\,\va}_{tt}dxdy\\
&-\int_0^\infty\int_{-\infty}^\infty (\va^{\frac{1}{2}}\vec{U}^{\,\va}\cdot\na \vec{U}^{\,\va}_t
+\vec{u}^{\,0}\cdot\na \vec{U}^{\,\va}_t-\va^{-\frac{1}{2}}\vec{h}^{\,\va}_t)\cdot\vec{U}^{\,\va}_{tt}dxdy\\
&-\lambda\int_0^\infty\int_{-\infty}^\infty M^\va_t U^{\va}_{2tt} dxdy
\\
:=&\sum_{i=1}^3 Q_i.
\end{split}
\ee
It follows from the Sobolev embedding inequality that
\ben
\begin{split}
Q_1\leq &(\va^{\frac{1}{2}}\|\vec{U}^{\,\va}_t\|_{L^4_{xy}} \|\na \vec{U}^{\,\va}\|_{L^4_{xy}}
+\|\vec{U}^{\,\va}\|_{L^2_{xy}}\|\na \vec{u}^{\,0}_t\|_{L^\infty_{xy}}
)\|\vec{U}^{\,\va}_{tt}\|_{L^2_{xy}}\\
 &+(
\|\vec{U}^{\,\va}_t\|_{L^2_{xy}}\|\na \vec{u}^{\,0}\|_{L^\infty_{xy}}
+\|\vec{u}^{\,0}_t\|_{L^\infty_{xy}}\|\na \vec{U}^{\,\va}\|_{L^2_{xy}})\|\vec{U}^{\,\va}_{tt}\|_{L^2_{xy}}\\
\leq &\frac{1}{4}\|\vec{U}^{\,\va}_{tt}\|_{L^2_{xy}}^2+C(\|\vec{U}^{\,\va}_t\|_{H^1_{xy}}^2 \|\vec{U}^{\,\va}\|_{H^2_{xy}}^2
+\|\vec{U}^{\,\va}\|_{H^1_{xy}}^2\|\vec{u}^{\,0}_t\|_{H^3_{xy}}^2
+\|\vec{U}^{\,\va}_t\|_{L^2_{xy}}^2\|\vec{u}^{\,0}\|_{H^3_{xy}}^2
)
\end{split}
\enn
and that
\ben
\begin{split}
Q_2+Q_3\leq &
(\va^{\frac{1}{2}}\|\vec{U}^{\,\va}\|_{L^\infty_{xy}}\|\na \vec{U}^{\,\va}_t\|_{L^2_{xy}}
+\|\vec{u}^{\,0}\|_{L^\infty_{xy}}\|\na \vec{U}^{\,\va}_t\|_{L^2_{xy}}+\va^{-\frac{1}{2}}\|\vec{h}^{\,\va}_t\|_{L^2_{xy}}
+\lambda \|M^\va_t\|_{L^2_{xy}})\|\vec{U}^{\,\va}_{tt}\|_{L^2_{xy}}\\
\leq &\frac{1}{4}\|\vec{U}^{\,\va}_{tt}\|_{L^2_{xy}}^2+C(\va\|\vec{U}^{\,\va}\|_{H^2_{xy}}^2+\|\vec{u}^{\,0}\|_{H^2_{xy}}^2)
\|\na\vec{U}^{\,\va}_t\|_{L^2_{xy}}
+\va^{-1}\|\vec{h}^{\,\va}_t\|_{L^2_{xy}}^2
+\lambda^2 \|M^\va_t\|_{L^2_{xy}}^2.
\end{split}
\enn
Inserting the above estimates for $Q_1$-$Q_3$ into \eqref{b77} and using Proposition \ref{p2}, Lemma \ref{l3}, Lemma \ref{l12} and \eqref{b79} , one derives
\ben
\|\na\vec{U}^{\,\va}_t\|_{L^\infty_TL^2_{xy}}+\|\vec{U}^{\,\va}_{tt}\|_{L^2_TL^2_{xy}}\leq C,
\enn
which, in conjunction with \eqref{b79} gives the desired estimates. The proof is finished.

\end{proof}

\begin{lemma}\label{l8}
Let the assumptions in Proposition \ref{p3} hold true. Assume further that
\be\label{a000}
\|M^\va\|_{L^\infty_TL^\infty_{xy}}+\|C^\va\|_{L^\infty_TL^\infty_{xy}}+\|\vec{U}^{\,\va}\|_{L^\infty_TH^2_{xy}}<1.
\ee
Then 
 there exists constants $C_7$ and $C$, independent of $\va$, depending on $T$, such that
\be\label{b89}
\|M^\va\|_{L^\infty_TL^\infty_{xy}}+\|C^\va\|_{L^\infty_TL^\infty_{xy}}+\|\vec{U}^{\,\va}\|_{L^\infty_TH^2_{xy}}
\leq C_7\va^{\frac{1}{8}}
\ee
and 
\be\label{b90}
\|\vec{U}^{\,\va}\|_{L^\infty_TL^\infty_{xy}}\leq C\va^{\frac{1}{2}},\qquad 
\|\na C^\va\|_{L^\infty_TL^\infty_{xy}}\leq C\va^{-\frac{3}{8}}, \qquad \|\na\vec{U}^{\,\va}\|_{L^\infty_TL^\infty_{xy}}\leq C\va^{\frac{1}{4}}.
\ee
\end{lemma}
\begin{proof} By Lemma \ref{l4}, Lemma \ref{l14} and the fact
\ben
\frac{d}{dt}\|\vec{V}^{\,\va}(t)\|^2_{H^1_{xy}}=2(\vec{V}^{\,\va}(t),\vec{V}_t^{\,\va}(t))_{H^1_{xy}},
\enn
where $(\vec{V}^{\,\va}(t),\vec{V}_t^{\,\va}(t))_{H^1_{xy}}$ denotes the ${H^1_{xy}}$ inner product of $\vec{V}^{\,\va}(t)$ and $\vec{V}_t^{\,\va}(t)$, one deduces that
\ben
\begin{split}
\|\vec{V}^{\,\va}(t)\|^2_{H^1_{xy}}
=2\int_0^t (\vec{V}^{\,\va}(s),\vec{V}_s^{\,\va}(s))_{H^1_{xy}}ds
\leq
2\|\vec{V}^{\,\va}\|_{L^2_TH^1_{xy}}\|\vec{V}_t^{\,\va}\|_{L^2_TH^1_{xy}}
\leq C\va^{-\frac{1}{2}}
\end{split}
\enn
for each $t\in (0,T]$, that is
\be\label{b35}
\|\vec{V}^{\,\va}\|_{L^\infty_TH^1_{xy}}\leq C\va^{-\frac{1}{4}}.
\ee
A similar argument used in deriving \eqref{b35} along with Lemma \ref{l4} and Lemma \ref{l14} leads to
\be\label{b36}
\|M^{\va}\|^2_{L^\infty_TH^1_{xy}}\leq C\|M^{\va}\|_{L^2_TH^1_{xy}}\|M^{\va}_t\|_{L^2_TH^1_{xy}}
\leq C\va^{\frac{1}{2}}.
\ee
By an analogous argument used in attaining \eqref{a65}, one gets
\be\label{b73}
\begin{split}
\|C^a\|_{L^\infty_TH^2_{xy}}
\leq C(\|c^0\|_{L^\infty_TH^2_{xy}}+\va^{-\frac{1}{4}}\|c^{B,1}\|_{L^\infty_TH^2_{xz}}+\va^{\frac{1}{4}}\|c^{B,2}\|_{L^\infty_TH^2_{xz}})
\leq C\va^{-\frac{1}{4}},
\end{split}
\ee
which, along with \eqref{a10}, \eqref{a35}-\eqref{b72}, \eqref{b35}, \eqref{b36}, Proposition \ref{p2}, Lemma \ref{l2}, Lemma \ref{l4}, Lemma \ref{l14} and the assumption \eqref{a000}, gives that
\ben
\begin{split}
\va \|\vec{V}^{\,\va}\|_{L^\infty_TH^2_{xy}}
\leq &
\va^{\frac{1}{2}}
 \|\vec{U}^{\,\va}\|_{L^\infty_TH^2_{xy}}\|\vec{V}^{\,\va}\|_{L^\infty_TH^1_{xy}}
 +C\|C^{a}\|_{L^\infty_TH^2_{xy}} \|\vec{U}^{\,\va}\|_{L^\infty_TH^2_{xy}} 
+\|\vec{u}^{\,0}\|_{L^\infty_TH^2_{xy}}\|\vec{V}^{\,\va}\|_{L^\infty_TH^1_{xy}}\\
&+\va^{\frac{1}{2}}\|\na M^\va\|_{L^\infty_TL^2_{xy}} \|C^\va\|_{L^\infty_TL^\infty_{xy}} +\va^{\frac{1}{2}}\|M^\va\|_{L^\infty_TL^\infty_{xy}} \|\vec{V}^{\,\va}\|_{L^\infty_TL^2_{xy}}
+\|M^\va\|_{L^\infty_TL^2_{xy}} \|\na C^a\|_{L^\infty_TL^\infty_{xy}}\\
&+\|\na M^\va\|_{L^\infty_TL^2_{xy}} \|C^a\|_{L^\infty_TL^\infty_{xy}} +C\|\na M^{a}\|_{L^\infty_TL^4_{xy}} \|C^\va\|_{L^\infty_TH^1_{xy}} \\ &+\|M^a\|_{L^\infty_TL^\infty_{xy}}\|\vec{V}^{\,\va}\|_{L^\infty_TL^2_{xy}}
+\|\vec{V}^{\,\va}_t\|_{L^\infty_TL^2_{xy}}
+\va^{-\frac{1}{2}}\|\na g^\va\|_{L^\infty_TL^2_{xy}}\\
\leq &C\va^{-\frac{1}{4}}.
\end{split}
\enn
Thus 
\ben
\|\vec{V}^{\,\va}\|_{L^\infty_TH^2_{xy}}\leq C\va^{-\frac{5}{4}},
\enn
which, along with the Gagliardo-Nirenberg interpolation inequality and Lemma \ref{l4} yields
\be\label{b81}
\|\vec{V}^{\,\va}\|_{L^\infty_TL^\infty_{xy}}\leq
 \|\vec{V}^{\,\va}\|_{L^\infty_TH^2_{xy}}^{\frac{1}{2}}\|\vec{V}^{\,\va}\|_{L^\infty_TL^2_{xy}}^{\frac{1}{2}}
 +\|\vec{V}^{\,\va}\|_{L^\infty_TL^2_{xy}}
 \leq C \va^{-\frac{3}{8}}.
\ee

From the first equation of \eqref{e23}, Sobolev embedding inequality and the fact $\vec{V}^{\,\va}=\na C^\va$, one gets
\be\label{a68}
\begin{split}
\|M^\va\|_{L^\infty_TH^2_{xy}}
\leq&
\va^{\frac{1}{2}}\|\na M^\va \cdot\vec{V}^{\,\va}\|_{L^\infty_TL^2_{xy}}
+\va^{\frac{1}{2}}\| M^\va\|_{L^\infty_TL^\infty_{xy}}\|\na\vec{V}^{\,\va}\|_{L^\infty_TL^2_{xy}}\\
&+\|M^\va\|_{L^\infty_TL^\infty_{xy}} \|\nabla^2 C^a\|_{L^\infty_TL^2_{xy}}
+\|\na M^\va\|_{L^\infty_TL^2_{xy}} \|\nabla C^a\|_{L^\infty_TL^\infty_{xy}}\\
&+\|\na M^a\|_{L^\infty_TL^4_{xy}} \|\vec{V}^{\,\va}\|_{L^\infty_TL^4_{xy}}
+\| M^a\|_{L^\infty_TL^\infty_{xy}} \|\na\vec{V}^{\,\va}\|_{L^\infty_TL^2_{xy}}\\
&
+\|M^\va_t\|_{L^\infty_TL^2_{xy}}+C\va^{\frac{1}{2}}\|\vec{U}^{\,\va}\|_{L^\infty_TH^2_{xy}}\|\na M^\va\|_{L^\infty_TL^2_{xy}}+\va^{-\frac{1}{2}}\|f^\va\|_{L^\infty_TL^2_{xy}}\\
&+\|\vec{U}^{\,\va}\|_{L^\infty_TL^4_{xy}}\|\nabla M^a\|_{L^\infty_TL^4_{xy}}
+\|\vec{u}^{\,0}\|_{L^\infty_TL^\infty_{xy}}\|\nabla M^\va\|_{L^\infty_TL^2_{xy}}.
\end{split}
\ee
The Gagliardo-Nirenberg interpolation inequality leads to
\ben
\begin{split}
\va^{\frac{1}{2}}\|\na M^\va \cdot \vec{V}^{\,\va}\|_{L^\infty_TL^2_{xy}}
\leq& 
\va^{\frac{1}{2}}\|\na M^\va\|_{L^\infty_TL^4_{xy}}\|\vec{V}^{\,\va}\|_{L^\infty_TL^4_{xy}}
\\
\leq &C\va^{\frac{1}{2}}(\|M^\va\|_{L^\infty_TH^2_{xy}}^{\frac{1}{2}}\|\na M^\va\|_{L^\infty_TL^2_{xy}}^{\frac{1}{2}}
+\|\na M^\va\|_{L^\infty_TL^2_{xy}})
\|\vec{V}^{\,\va}\|_{L^\infty_TH^1_{xy}}\\
\leq & \frac{1}{2}\|M^\va\|_{L^\infty_TH^2_{xy}}+C\va\|\na M^\va\|_{L^\infty_TL^2_{xy}}\|\vec{V}^{\,\va}\|_{L^\infty_TH^1_{xy}}^2+C\va^{\frac{1}{2}}\|\na M^\va\|_{L^\infty_TL^2_{xy}}\|\vec{V}^{\,\va}\|_{L^\infty_TH^1_{xy}}.
\end{split}
\enn
Substituting the above estimate into \eqref{a68} and using the assumption $\| M^\va\|_{L^\infty_TL^\infty_{xy}}<1$, \eqref{b35}-\eqref{b73}, \eqref{a35}-\eqref{b72}, Proposition \ref{p2}, Lemma \ref{l1} and Lemma \ref{l14}, we have
\be\label{a69}
\|M^\va\|_{L^\infty_TH^2_{xy}}\leq C\va^{-\frac{1}{4}}.
\ee
It follows from the Galiardo-Nirenberg interpolation inequality, \eqref{a69} and Lemma \ref{l4} that
\be\label{b82}
\|M^\va\|_{L^\infty_TL^\infty_{xy}}
\leq C(\|M^\va\|_{L^\infty_TH^2_{xy}}^{\frac{1}{2}}\|M^\va\|_{L^\infty_TL^2_{xy}}^{\frac{1}{2}}+\|M^\va\|_{L^\infty_TL^2_{xy}})
\leq C_5\va^{\frac{1}{8}},
\ee
where the constant $C_5$ is independent of $\va$, depending on $T$.
The Gagliardo-Nirenberg interpolation inequality, the fact $\vec{V}^{\,\va}=\na C^\va$, \eqref{b35} and Lemma \ref{l4}, yield
\be\label{b83}
\begin{split}
\|C^\va\|_{L^\infty_TL^\infty_{xy}}
\leq C(\|\na^2 C^\va\|_{L^\infty_TL^2_{xy}}^{\frac{1}{2}}\|C^\va\|_{L^\infty_TL^2_{xy}}^{\frac{1}{2}}+\|C^\va\|_{L^\infty_TL^2_{xy}})
\leq C_6\va^{\frac{1}{8}}
\end{split}
\ee
with the constant $C_6$ independent of $\va$ and depending on $T$. Denoting $C_7=C_4+C_5+C_6$, one deduces \eqref{b89} from Lemma \ref{l12}, \eqref{b82}, \eqref{b83} and the assumption $0<\va<1$.

Lemma \ref{l4}, Lemma \ref{l12} and the Gagliardo-Nirenberg interpolation inequality further lead to
\be\label{b84}
\begin{split}
\|\vec{U}^{\,\va}\|_{L^\infty_TL^\infty_{xy}}
\leq C(\|\vec{U}^{\,\va}\|_{L^\infty_TH^2_{xy}}^{\frac{1}{2}}\|\vec{U}^{\,\va}\|_{L^\infty_TL^2_{xy}}^{\frac{1}{2}}
+\|\vec{U}^{\,\va}\|_{L^\infty_TL^2_{xy}})
\leq C\va^{\frac{1}{2}}.
\end{split}
\ee
Similarly to the derivation of \eqref{a40}, one deduces from the third equation of \eqref{e23}, \eqref{b36}, Proposition \ref{p2}, Lemma \ref{l3}, Lemma \ref{l4} and Lemma \ref{l12} that
\ben
\begin{split}
&\|\na P^\va\|_{L^\infty_TH^1_{xy}}+\|\vec{U}^{\,\va}\|_{L^\infty_TH^3_{xy}}\\
\leq & \|\vec{U}^{\,\va}_t\|_{L^\infty_TH^1_{xy}}
+C\va^{\frac{1}{2}}\|\vec{U}^{\,\va}\|_{L^\infty_TH^2_{xy}}^2
+C\|\vec{u}^{\,0}\|_{L^\infty_TH^3_{xy}}\|\vec{U}^{\,\va}\|_{L^\infty_TH^2_{xy}}
+\lambda\|M^{\va}\|_{L^\infty_TH^1_{xy}}+\va^{-\frac{1}{2}}\|\vec{h}^{\,\va}\|_{L^\infty_TH^1_{xy}}\\
\leq &C,
\end{split}
\enn
which, along with Lemma \ref{l12} and the Gagliardo-Nirenberg inequality leads
\be\label{b85}
\begin{split}
\|\na\vec{U}^{\,\va}\|_{L^\infty_TL^\infty_{xy}}
\leq C(\|\na\vec{U}^{\,\va}\|_{L^\infty_TH^2_{xy}}^{\frac{1}{2}}\|\na\vec{U}^{\,\va}\|_{L^\infty_TL^2_{xy}}^{\frac{1}{2}}
+\|\na\vec{U}^{\,\va}\|_{L^\infty_TL^2_{xy}})
\leq C\va^{\frac{1}{4}}.
\end{split}
\ee
Collecting \eqref{b81}, \eqref{b84} and \eqref{b85} and using the assumption $0<\va<1$, we obtain \eqref{b90}. The proof is completed.

\end{proof}
  Based on the results derived in Lemma \ref{l8}, we next prove Proposition \ref{p3}.
\newline
\emph{Proof of Proposition \ref{p3}.}
Let 
\be\label{b48}
\va_T=\min\{(2C_7)^{-8},1\}.
\ee
 Then it follows from \eqref{b89} that \ben
 \|M^\va\|_{L^\infty_TL^\infty_{xy}}+\|C^\va\|_{L^\infty_TL^\infty_{xy}}+\|\vec{U}^{\,\va}\|_{L^\infty_TH^2_{xy}}<\frac{1}{2},
  \enn
 for each $\va\in (0,\va_T]$, which, along with Lemma \ref{l8} and the \emph{bootstrap principle} (see \cite[page 21, Proposition 1.21]{tao2006nonlinear}) gives \eqref{b86}. \eqref{b87} follows from \eqref{b90}.
 
\endProof
\subsection{Proof of Theorem \ref{t1}.} 
%First, from Lemma - Lemma and Proposition, one deduces for fixed $\va>0$ that
%\ben
%(M^a, C^a,\vec{u}^{\,0}, \na p^0)\in C([0,T];H^2_{xy}\times H^3_{xy}\times H^3_{xy}\times H^1_{xy}), 
%\enn
%which, along with \eqref{e22} and Proposition \ref{p3} leads to
%\be\label{b92}
%(m^\va,c^\va,\vec{u}^{\,\va}, \na p^\va)\in C([0,T];H^2_{xy}\times H^3_{xy}\times H^3_{xy}\times H^1_{xy})
%\ee
%for fixed $0<\va<\va_T$.
First, it follows from \eqref{e22}, Proposition \ref{p3}, the change of variables $z=\frac{y}{\sqrt{\va}}$, the Sobolev embedding inequality and Lemma \ref{l10} - Lemma \ref{l11} that
\be\label{b93}
\begin{split}
\|m^\va-m^0\|_{L^\infty_TL^\infty_{xy}}\leq& \va^{\frac{1}{2}}\|M^\va\|_{L^\infty_TL^\infty_{xy}}
+\va^{\frac{1}{2}}\|m^{B,1}\|_{L^\infty_TL^\infty_{xy}}+\va\|m^{B,2}\|_{L^\infty_TL^\infty_{xy}}
+\va^{\frac{3}{2}}\|\xi\|_{L^\infty_TL^\infty_{xy}}\\
\leq& \va^{\frac{1}{2}}\|M^\va\|_{L^\infty_TL^\infty_{xy}}
+\va^{\frac{1}{2}}\|m^{B,1}\|_{L^\infty_TL^\infty_{xz}}+\va\|m^{B,2}\|_{L^\infty_TL^\infty_{xz}}
+\va^{\frac{3}{2}}\|\xi\|_{L^\infty_TL^\infty_{xz}}\\
\leq& \va^{\frac{1}{2}}\|M^\va\|_{L^\infty_TL^\infty_{xy}}
+C(\va^{\frac{1}{2}}\|m^{B,1}\|_{L^\infty_TH^2_{xz}}+\va\|m^{B,2}\|_{L^\infty_TH^2_{xz}}
+\va^{\frac{3}{2}}\|\xi\|_{L^\infty_TH^2_{xz}})\\
\leq &C\va^{\frac{1}{2}}.
\end{split}
\ee
By a similar argument used in deriving \eqref{b93}, one deduces from \eqref{e22}, Proposition \ref{p3} and Lemma \ref{l10} - Lemma \ref{l11} that
\be\label{b94}
\begin{split}
\|c^\va-c^0\|_{L^\infty_TL^\infty_{xy}}
\leq& \va^{\frac{1}{2}}\|C^\va\|_{L^\infty_TL^\infty_{xy}}
+C(\va^{\frac{1}{2}}\|c^{B,1}\|_{L^\infty_TH^2_{xz}}+\va\|c^{B,2}\|_{L^\infty_TH^2_{xz}}
)
\leq C\va^{\frac{1}{2}},\\
\|\pt_xc^\va-\pt_xc^0\|_{L^\infty_TL^\infty_{xy}}
\leq& \va^{\frac{1}{2}}\|\na C^\va\|_{L^\infty_TL^\infty_{xy}}
+C(\va^{\frac{1}{2}}\|\pt_xc^{B,1}\|_{L^\infty_TH^2_{xz}}+\va\|\pt_xc^{B,2}\|_{L^\infty_TH^2_{xz}}
)
\leq C\va^{\frac{1}{8}}
\end{split}
\ee
and 
\be\label{b95}
\begin{split}
\|\vec{u}^{\,\va}-\vec{u}^{\,0}\|_{L^\infty_TL^\infty_{xy}}
\leq& \va^{\frac{1}{2}}\|\vec{U}^{\,\va}\|_{L^\infty_TL^\infty_{xy}}\leq C\va,\\
\|\na\vec{u}^{\,\va}-\na\vec{u}^{\,0}\|_{L^\infty_TL^\infty_{xy}}
\leq& \va^{\frac{1}{2}}\|\na\vec{U}^{\,\va}\|_{L^\infty_TL^\infty_{xy}}\leq C\va^{\frac{3}{4}}.
\end{split}
\ee
\eqref{e22} along with Proposition \ref{p3}, Lemma \ref{l11} and the change of variables $z=\frac{y}{\sqrt{\va}}$ further gives
\be\label{b96}
\begin{split}
&\|\pt_yc^\va(x,y,t)-[\pt_yc^0(x,y,t)+\pt_zc^{B,1}(x,\frac{y}{\sqrt{\va}},t)]\|_{L^\infty_TL^\infty_{xy}}\\
\leq& \va^{\frac{1}{2}}\|\na C^{\va}(x,y,t)\|_{L^\infty_TL^\infty_{xy}}
+\va\|\pt_y c^{B,2}(x,\frac{y}{\sqrt{\va}},t)\|_{L^\infty_TL^\infty_{xy}}\\
\leq& \va^{\frac{1}{2}}\|\na C^{\va}(x,y,t)\|_{L^\infty_TL^\infty_{xy}}
+\va^{\frac{1}{2}}\|\pt_z c^{B,2}(x,z,t)\|_{L^\infty_TL^\infty_{xz}}\\
\leq& \va^{\frac{1}{2}}\|\na C^{\va}(x,y,t)\|_{L^\infty_TL^\infty_{xy}}
+C\va^{\frac{1}{2}}\|\pt_z c^{B,2}(x,z,t)\|_{L^\infty_TH^2_{xz}}\\
\leq &C\va^{\frac{1}{8}}.
\end{split}
\ee
Collecting \eqref{b93}-\eqref{b96}, we obtain the desired estimate and complete the proof.

\endProof

\section{Appendix}
\noindent \textbf{Step 1. Initial and boundary conditions.}
  Inserting \eqref{b2} into the initial conditions in \eqref{e1}, one gets
%Letting $\eta\rightarrow -\infty$ in \eqref{a35}, we obtain $p^{b,j},\,q^{b,j}\rightarrow 0$ due to the assumption that $p^{b,j},\,q^{b,j}$ decay to zero exponentially as $\eta\rightarrow -\infty$.
\be\label{b3}
\begin{split}
m^{I,0}(x,y,0)&=m_{0}(x,y),\quad m^{B,0}(x,z,0)=0,\\
c^{I,0}(x,y,0)&=c_{0}(x,y),\quad c^{B,0}(x,z,0)=0,\\
\vec{u}^{\, I,0}(x,y,0)&=\vec{u}_{0}(x,y),\quad \vec{u}^{\,B,0}(x,z,0)=\mathbf{0}
\end{split}
\ee
and for $j\geq 1$
\be\label{b4}
\begin{split}
m^{I,j}(x,y,0)&=m^{B,j}(x,z,0)=0,\\
c^{I,j}(x,y,0)&=c^{B,j}(x,z,0)=0,\\
\vec{u}^{\,I,j}(x,y,0)&=\vec{u}^{\,B,j}(x,z,0)=\mathbf{0}.
\end{split}
\ee
%Similarly, sending $z\rightarrow \infty$ in \eqref{a35}, we infer that for $j\geq 0$
%\be\label{a32}
%p^{b,j}(\eta,0)=q^{b,j}(\eta,0)=0.
%\ee
For boundary conditions, one substitutes \eqref{b2} into \eqref{e2} and gets for $j\in\mathbb{N}$ that
\ben
\begin{split}
 0=&\sum_{j=0}^\infty \va^{\frac{j}{2}}\partial_y m^{I,j}(x,0,t)
 +\sum_{j=-1}^\infty \va^{\frac{j}{2}}\partial_z m^{B,j+1}(x,0,t)
 -\sum_{j=0}^\infty \va^{\frac{j}{2}}\sum_{l=0}^j m^{I,l}(x,0,t)\partial_y c^{I,j-l}(x,0,t)\\
 &-\sum_{j=-1}^\infty \va^{\frac{j}{2}}\sum_{l=0}^{j+1} m^{I,l}(x,0,t)\partial_z c^{B,j+1-l}(x,0,t)
 -\sum_{j=0}^\infty \va^{\frac{j}{2}}\sum_{l=0}^{j} m^{B,l}(x,0,t)\partial_y c^{I,j-l}(x,0,t)
 \\
 &-\sum_{j=-1}^\infty \va^{\frac{j}{2}}\sum_{l=0}^{j+1} m^{B,l}(x,0,t)\partial_z c^{B,j+1-l}(x,0,t),\\
&\qquad\qquad \qquad 
\sum_{j=0}^\infty \va^{\frac{j}{2}}\partial_y c^{I,j}(x,0,t)
+\sum_{j=-1}^\infty \va^{\frac{j}{2}}\partial_z c^{B,j+1}(x,0,t)=0
\end{split}
\enn
and
\ben
\begin{split}
\sum_{j=0}^\infty \va^{\frac{j}{2}}
[\vec{u}^{\,I,j}(x,0,t)+\vec{u}^{\,B,j}(x,0,t)]=\mathbf{0}.
\end{split}
\enn
To fulfill the above boundary conditions for all small $\va>0$, it is required that
\be\label{b6}
\begin{split}
&\ \ \ \ \ \ \ \ 0=\partial_z m^{B,0}(x,0,t)-m^{I,0}(x,0,t)
\partial_z c^{B,0}(x,0,t)
-m^{B,0}(x,0,t)
\partial_z c^{B,0}(x,0,t),\\
&\partial_z c^{B,0}(x,0,t)=0,\quad\partial_z c^{B,1}(x,0,t)
=-\partial_y c^{I,0}(x,0,t),\quad\partial_z c^{B,2}(x,0,t)
=-\partial_y c^{I,1}(x,0,t)
\end{split}
\ee
and 
\be\label{b5}
\begin{split}
0=&\partial_y m^{I,k}(x,0,t)+\partial_z m^{B,k+1}(x,0,t)
-\sum_{l=0}^{k}m^{I,l}(x,0,t)\partial_y c^{I,k-l}(x,0,t)\\
&-\sum_{l=0}^{k+1}m^{I,l}(x,0,t)\partial_z c^{B,k+1-l}(x,0,t)
-\sum_{l=0}^{k}m^{B,l}(x,0,t)\partial_y c^{I,k-l}(x,0,t)\\
&-\sum_{l=0}^{k+1}m^{B,l}(x,0,t)\partial_z c^{B,k+1-l}(x,0,t)
\end{split}
\ee
and
\be\label{b7}
\begin{split}
\vec{u}^{\,I,k}(x,0,t)+\vec{u}^{\,B,k}(x,0,t)=\mathbf{0}
\end{split}
\ee
for $k\geq 0$.
\\
\textbf{Step 2. Equations for $\vec{u}^{\,I,j}$ and $\vec{u}^{\,B,j}$.} For equations of outer layer profiles $\vec{u}^{\,I,j}$, we omit the boundary layer profiles $(m^{B,j},c^{B,j},\vec{u}^{\,B,j},p^{B,j})$ and substitute \eqref{b2} into the third and fourth equations in \eqref{e1} to deduce that
\be\label{b8}
\begin{split}
\vec{u}^{\,I,j}_t+\sum_{l=0}^{j}\vec{u}^{\,I,l}&\cdot \nabla \vec{u}^{\,I,j-l}+\nabla p^{I,j}+m^{I,j}(0,\lambda)=\Delta \vec{u}^{\,I,j}
\end{split}
\ee
and
\be\label{b10}
\begin{split}
\nabla\cdot \vec{u}^{\,I,j}=0
\end{split}
\ee
for $j\geq 0$.
 To find the equations for boundary layer profiles $\vec{u}^{\,B,j}$, by a similar argument in \cite[Step 2, Appendix]{HLWW}, namely inserting \eqref{b2} into the third equation of \eqref{e1} and subtracting \eqref{b8} from the resulting equations then expanding $\vec{u}^{\,I,j}(x,y,t)=\vec{u}^{\,I,j}(x,\va^{1/2}z,t)$
 in $\va$ by the Taylor expansion , we end up with
 \be\label{b9}
\sum_{j= -2}^\infty\va^{j/2}\vec{G}^{\,j}(x,z,t)=\mathbf{0},
\ee
where
\ben
\left\{\aligned
\vec{G}^{\,-2}=&-\partial^2_z\vec{u}^{\,B,0},\\
\vec{G}^{\,-1}=&[\overline{u^{\,I,0}_2}+u^{\,B,0}_2]\partial_z\vec{u}^{\,B,0}
+(0,\partial_z p^{B,0})-\partial_z^2 \vec{u}^{\,B,1},\\
\vec{G}^{\,0}=&\vec{u}^{\,B,0}_t
+\vec{u}^{\,B,0}\cdot \overline{\nabla\vec{u}^{\,I,0}}
+[\ol{u^{\,I,0}_1}+u^{\,B,0}_1]\partial_x \vec{u}^{\,B,0}
+[\ol{u^{\,I,0}_2}+u^{\,B,0}_2]\partial_z \vec{u}^{\,B,1}
+[\ol{u^{\,I,1}_2}+u^{\,B,1}_2]\partial_z \vec{u}^{\,B,0}\\
&+(0,\lambda m^{B,0})
+(\partial_x p^{B,0},\partial_z p^{B,1})
+z\ol{\partial_y u^{I,0}_2}\partial_z \vec{u}^{\,B,0}
-\partial_x^2 \vec{u}^{\,B,0}
-\partial_z^2 \vec{u}^{\,B,2},
\\
\vec{G}^{\,1}=&\vec{u}^{\,B,1}_t
+\vec{u}^{\,B,0}\cdot \ol{\nabla\vec{u}^{\,I,1}}
+\vec{u}^{\,B,1}\cdot\ol{\nabla \vec{u}^{\,I,0}}
+[\ol{u^{I,0}_1}+u^{B,0}_1]\partial_x \vec{u}^{\,B,1}
+[\overline{u^{I,1}_1}+u^{B,1}_1]\partial_x \vec{u}^{\,B,0}\\
&+[\overline{u^{I,0}_2}+u^{B,0}_2]\partial_z \vec{u}^{\,B,2}
+[\overline{u^{I,1}_2}+u^{B,1}_2]\partial_z \vec{u}^{\,B,1}
+[\overline{u^{I,2}_2}+u^{B,2}_2]\partial_z \vec{u}^{\,B,0}
+(\partial_x p^{B,1},\partial_z p^{B,2})\\
&+(0,\lambda m^{B,1})
+z[\vec{u}^{\,B,0}\cdot \overline{\nabla\partial_y\vec{u}^{\,I,0}}
+\ol{\partial_yu^{\,I,0}_1}\partial_x \vec{u}^{\,B,0}
+\ol{\partial_yu^{\,I,0}_2}\partial_z \vec{u}^{\,B,1}
+\ol{\partial_yu^{\,I,1}_2}\partial_z \vec{u}^{\,B,0}]\\
&+z^2\overline{\partial^2_y u^{I,0}_2}\partial_z \vec{u}^{\,B,0}
-\partial_x^2\vec{u}^{\,B,1}
-\partial_z^2\vec{u}^{\,B,3},
\\
\cdots&\,\cdots
\endaligned\right.
\enn
 with $\vec{G}^{j}=\mathbf{0}$ for $j\geq -2$. Moreover, substituting \eqref{b2} into the fourth equation in \eqref{e1} and using \eqref{b10}, one gets
 \be\label{b12}
 \partial_z u^{B,0}_2=0
 \ee
 and
 \be\label{b11}
 \partial_x u^{B,j}_1+\partial_z u^{B,j+1}_2=0
 \ee
with $j\geq 0$.
Using assumption (H), we deduce from \eqref{b7}, \eqref{b9}-\eqref{b11} that
\be\label{b13}
\vec{u}^{\,B,0}=\vec{u}^{\,B,1}
=\vec{u}^{\,B,2}=\vec{u}^{\,B,3}=\mathbf{0}
\ee
and that
\be\label{b14}
p^{B,0}=p^{B,1}=0,\qquad p^{B,2}(x,z,t)=\lambda \int_z^\infty m^{B,1}(x,s,t)ds.
\ee
\textbf{Step 3. Equations for $m^{I,j}$ and $m^{B,j}$.} Similarly to Step 2, we omit the boundary layer profiles $(m^{B,j},c^{B,j},\vec{u}^{\,B,j},p^{B,j})$ and insert \eqref{b2} into the first equation of \eqref{e1} to get
\be\label{b15}
m^{I,j}_t+\sum_{l=0}^{j}\vec{u}^{\,I,l}\cdot\nabla m^{I,j-l}+\sum_{l=0}^{j}\nabla\cdot(m^{I,l}\nabla c^{I,j-l})=\Delta m^{I,j},
\ee
with $j\geq 0$. Plugging \eqref{b2} into the first equation of \eqref{e1} and subtracting \eqref{b15} from the resulting equality then applying the Taylor expansion to $(m^{I,j},c^{I,j},\vec{u}^{\,I,j})(x,\va^{1/2}z,t)$ in $\va$ and using \eqref{b13}, one gets
\be\label{b16}
\sum_{j= -2}^\infty\va^{j/2}F^{\,j}(x,z,t)=0,
\ee
where
\ben
\left\{\aligned
F^{-2}=&\ol{m^{I,0}}\partial_z^2 c^{B,0}+\partial_z (m^{B,0}\partial_z c^{B,0})
-\partial_z^2 m^{B,0},\\
F^{-1}=&\ol{u^{I,0}_2}\partial_z m^{B,0}+\ol{\partial_y m^{I,0}}\partial_z c^{B,0}+\partial_z m^{B,0}\ol{\partial_y c^{I,0}}+\ol{m^{I,0}}\partial^2_z c^{B,1}+\ol{m^{I,1}}\partial^2_z c^{B,0}
+\partial_z(m^{B,0}\partial_z c^{B,1})\\
&+\partial_z(m^{B,1}\partial_z c^{B,0})
+z\ol{\partial_y m^{I,0}}\partial^2_z c^{B,0}-\partial^2_z m^{B,1},\\
F^{0}=&m^{B,0}_t
+\ol{u^{\,I,0}_1}\partial_x m^{B,0}
+\ol{u^{I,0}_2}\partial_z m^{B,1}
+\ol{u^{I,1}_2}\partial_z m^{B,0}
+\ol{\partial_y m^{I,0}}\partial_z c^{B,1}
+\ol{\partial_y m^{I,1}}\partial_z c^{B,0}\\
&+\partial_zm^{B,0}\ol{\partial_y c^{I,1}}
+\partial_zm^{B,1}\ol{\partial_y c^{I,0}}+\partial_x[\ol{m^{I,0}}\partial_x c^{B,0}
+m^{B,0}\ol{\partial_x c^{I,0}}+m^{B,0}\partial_x c^{B,0}]
\\
&+\partial_z[m^{B,0}\partial_z c^{B,2}+m^{B,1}\partial_z c^{B,1}
+m^{B,2}\partial_z c^{B,0}]
+\ol{m^{I,0}}\partial_z^2 c^{B,2}
+\ol{m^{I,1}}\partial_z^2 c^{B,1}\\
&+\ol{m^{I,2}}\partial_z^2 c^{B,0}
+m^{B,0}\ol{\partial_y^2 c^{I,0}}
+z[\ol{\partial_y u^{I,0}_2}\partial_z m^{B,0}+\ol{\partial_y^2 m^{I,0}}\partial_z c^{B,0}+\partial_z m^{B,0}\ol{\partial_y^2 c^{I,0}}]\\
&+z[\ol{\partial_y m^{I,0}}\partial^2_z c^{B,1}
+\ol{\partial_y m^{I,1}}\partial^2_z c^{B,0}]
+\frac{z^2}{2}\ol{\partial_y^2 m^{I,0}}\partial^2_z c^{B,0}-\partial^2_x m^{B,0}-\partial^2_z m^{B,2},
\\
F^{1}=&m^{B,1}_t
+\ol{u^{I,0}_1}\partial_x m^{B,1}
+\ol{u^{I,1}_1}\partial_x m^{B,0}
+\ol{u^{I,0}_2}\partial_z m^{B,2}
+\ol{u^{I,1}_2}\partial_z m^{B,1}
+\ol{u^{I,2}_2}\partial_z m^{B,0}+\partial_x(m^{B,0}\partial_x c^{B,1})\\
&+\partial_x[m^{B,1}\ol{\partial_xc^{I,0}}
+m^{B,1}\partial_x c^{B,0}
+\ol{m^{I,0}}\partial_x c^{B,1}+m^{B,0}\ol{\partial_x c^{I,1}}
+\ol{m^{I,1}}\partial_x c^{B,0}]
+\ol{\partial_ym^{I,0}}\partial_zc^{B,2}\\
&+\ol{\partial_ym^{I,1}}\partial_zc^{B,1}
+\ol{\partial_ym^{I,2}}\partial_zc^{B,0}
+\partial_zm^{B,2}\ol{\partial_yc^{I,0}}
+\partial_zm^{B,1}\ol{\partial_yc^{I,1}}
+\partial_zm^{B,0}\ol{\partial_yc^{I,2}}
\\
&
+\partial_z[m^{B,0}\partial_zc^{B,3}
+m^{B,1}\partial_zc^{B,2}+m^{B,2}\partial_zc^{B,1}
+m^{B,3}\partial_zc^{B,0}]
+m^{B,0}\ol{\partial_y^2 c^{I,1}}
+m^{B,1}\ol{\partial_y^2 c^{I,0}}\\
&+\ol{m^{I,0}}\partial^2_zc^{B,3}
+\ol{m^{I,1}}\partial^2_zc^{B,2}
+\ol{m^{I,2}}\partial^2_zc^{B,1}
+\ol{m^{I,3}}\partial^2_zc^{B,0}
+z\ol{\partial_y u^{\,I,0}_1}\partial_x m^{B,0}
+z
\ol{\partial_y u^{I,0}_2}\partial_z m^{B,1}\\
&+z[\ol{\partial_y u^{I,1}_2}\partial_z m^{B,0}
+\ol{\partial_y^2 m^{I,0}}\partial_z c^{B,1}
+\ol{\partial_y^2 m^{I,1}}\partial_z c^{B,0}
+\partial_zm^{B,0}\ol{\partial_y^2 c^{I,1}}
+\partial_zm^{B,1}\ol{\partial_y^2 c^{I,0}}]\\
&+z\partial_x[\ol{\partial_y m^{I,0}}\partial_x c^{B,0}
+m^{B,0}\ol{\partial_y \partial_x c^{I,0}}]
+z[\ol{\partial_y m^{I,0}}\partial_z^2 c^{B,2}+
\ol{\partial_y m^{I,1}}\partial_z^2 c^{B,1}+\ol{\partial_ym^{I,2}}\partial_z^2 c^{B,0}]\\
&
+zm^{B,0}\ol{\partial_y^3 c^{I,0}}+\frac{z^2}{2}[\ol{\partial_y^2 u^{I,0}_2}\partial_z m^{B,0}+\ol{\partial_y^3 m^{I,0}}\partial_z c^{B,0}+\partial_z m^{B,0}\ol{\partial_y^3 c^{I,0}}+\ol{\partial_y^2 m^{I,0}}\partial^2_z c^{B,1}]\\
&+\frac{z^2}{2}\ol{\partial_y^2 m^{I,1}}\partial^2_z c^{B,0}
+\frac{z^3}{3!}\ol{\partial_y^3 m^{I,0}}\partial^2_z c^{B,0}
-\partial^2_xm^{B,1}
-\partial^2_zm^{B,3},\\
\cdots&\,\cdots
\endaligned\right.
\enn
with $F^{j}=0$ for $j\geq -2.$\\
\textbf{Step 4. Equations for $c^{I,j}$ and $c^{B,j}$.} We omit the boundary layer profiles and insert \eqref{b2} into the second equation of \eqref{e1} to have
\be\label{b17}
c^{I,0}_t+\vec{u}^{\,I,0}\cdot\nabla c^{I,0}
+m^{I,0}c^{I,0}=0,
\ee
\be\label{b24}
c^{I,1}_t+\vec{u}^{\,I,0}\cdot\nabla c^{I,1}
+\vec{u}^{\,I,1}\cdot\nabla c^{I,0}
+m^{I,0}c^{I,1}+m^{I,1}c^{I,0}=0
\ee
and
\be\label{b18}
c^{I,j}_t+\sum_{l=0}^j\vec{u}^{\,I,l}\cdot\nabla c^{I,j-l}+\sum_{l=0}^j m^{I,l}c^{I,j-l}=\Delta c^{I,j-2}
\ee
with $j\geq2$. Plugging \eqref{b2} into the second equation of \eqref{e1} and subtracting \eqref{b17}-\eqref{b18} from the resulting equation then applying the Taylor expansion to $(m^{I,j},c^{I,j},\vec{u}^{\,I,j})(x,\va^{1/2}z,t)$ in $\va$ and using \eqref{b13}, we obtain
\be\label{b19}
\sum_{j= -2}^\infty\va^{j/2}H^{j}(x,z,t)=0,
\ee
where
\ben
\left\{\aligned
H^{-1}=&\ol{u^{I,0}_2}\partial_zc^{B,0},\\
H^{0}=&c^{B,0}_t
+\ol{u^{I,0}_1}\partial_xc^{B,0}
+\ol{u^{I,0}_2}\partial_zc^{B,1}+\ol{u^{I,1}_2}\partial_zc^{B,0}
+\ol{m^{I,0}}c^{B,0}
+m^{B,0}[\ol{c^{I,0}}+c^{B,0}]+z \ol{\partial_yu^{I,0}_2}\partial_zc^{B,0}\\
&-\partial^2_z c^{B,0},\\
H^1=&c^{B,1}_t
+\ol{u^{I,0}_1}\partial_xc^{B,1}+\ol{u^{I,1}_1}\partial_xc^{B,0}
+\ol{u^{I,0}_2}\partial_zc^{B,2}+\ol{u^{I,1}_2}\partial_zc^{B,1}
+\ol{u^{I,2}_2}\partial_zc^{B,0}
+\ol{m^{I,0}}c^{B,1}+\ol{m^{I,1}}c^{B,0}\\
&+m^{B,0}[\ol{c^{I,1}}+c^{B,1}]
+m^{B,1}[\ol{c^{I,0}}+c^{B,0}]
+z[\ol{\partial_y {u^{I,0}_1}}\partial_x c^{B,0}+\ol{\pt_y u^{I,0}_2}\pt_z c^{B,1}
+\ol{\pt_y u^{I,1}_2}\pt_z c^{B,0}]\\
&+z [\ol{\pt_ym^{I,0}}c^{B,0}+m^{B,0}\ol{\pt_y c^{I,0}}]+\frac{z^2}{2}\ol{\partial_y^2 u^{I,0}_2}\pt_z c^{B,0}
-\pt^2_z c^{B,1},\\
H^2=& c^{B,2}_t
+\ol{u^{I,0}_1}\pt_x c^{B,2}+\ol{u^{I,1}_1}\pt_x c^{B,1}
+\ol{u^{I,2}_1}\pt_x c^{B,0}+\ol{u^{I,0}_2}\pt_z c^{B,3}
+\ol{u^{I,1}_2}\pt_z c^{B,2}
+\ol{u^{I,2}_2}\pt_z c^{B,1}+\ol{u^{I,3}_2}\pt_z c^{B,0}
\\
&+\ol{m^{I,0}}c^{B,2}+\ol{m^{I,1}}c^{B,1}+\ol{m^{I,2}}c^{B,0}+m^{B,0}[\ol{c^{I,2}}+c^{B,2}]
+m^{B,1}[\ol{c^{I,1}}+c^{B,1}]+m^{B,2}[\ol{c^{I,0}}+c^{B,0}]
\\
&+z[\ol{\pt_yu^{I,0}_1}\partial_xc^{B,1}+\ol{\pt_yu^{I,1}_1}\partial_xc^{B,0}
+\ol{\pt_yu^{I,0}_2}\partial_zc^{B,2}+\ol{\pt_yu^{I,1}_2}\partial_zc^{B,1}
+\ol{\pt_yu^{I,2}_2}\partial_zc^{B,0}+\ol{\pt_ym^{I,0}}c^{B,1}]\\
&+z[
\ol{\pt_ym^{I,1}}c^{B,0}+m^{B,0}\ol{\pt_yc^{I,1}}+m^{B,1}\ol{\pt_yc^{I,0}}]
+\frac{z^2}{2}[\ol{\partial_y^2 {u^{I,0}_1}}\partial_x c^{B,0}+\ol{\pt_y^2 u^{I,0}_2}\pt_z c^{B,1}
+\ol{\pt_y^2 u^{I,1}_2}\pt_z c^{B,0}]\\
&+\frac{z^2}{2}[\ol{\pt_y^2m^{I,0}}c^{B,0}+m^{B,0}\ol{\pt_y^2 c^{I,0}}]
+\frac{z^3}{3!}\ol{\partial_y^3 u^{I,0}_2}\pt_z c^{B,0}
-\pt^2_xc^{B,0}-\pt^2_zc^{B,2},
\\
\cdots&\,\cdots
\endaligned\right.
\enn
 with $H^{j}=0$ for $j\geq -1$.
\\
\textbf{Step 5. Derivation of \eqref{e3}-\eqref{e18}.}
The first equality in \eqref{b16}, along with assumption (H) leads to
 \ben
 \ol{m^{I,0}}\partial_z c^{B,0}+m^{B,0}\partial_z c^{B,0}
-\partial_z m^{B,0}=0\qquad \text{for}\ \ (x,z,t)\in \mathbb{R}^2_{+}\times(0,\infty),
\enn
 which, in conjunction with assumption (H) further gives
 \be\label{b31}
 m^{B,0}=\ol{m^{I,0}}(e^{c^{B,0}}-1)\qquad \text{for}\ \ (x,z,t)\in \mathbb{R}^2_{+}\times(0,\infty).
 \ee
 The second identity of \eqref{b19}, \eqref{b13} and \eqref{b7} gives
 \ben
 c^{B,0}_t+z \ol{\partial_y u^{I,0}_2} \partial_z c^{B,0}+\ol{m^{I,0}}c^{B,0}
+m^{B,0}[\ol{c^{I,0}}+c^{B,0}]=\partial_z^2 c^{B,0}.
\enn
Inserting \eqref{b31} into the above equation and using \eqref{b3} and \eqref{b6}, we have
\ben
\left\{
\begin{array}{lll}
c^{B,0}_t+z \ol{\partial_y u^{I,0}_2} \partial_z c^{B,0}+\ol{m^{I,0}}c^{B,0}
+\ol{m^{I,0}}(e^{c^{B,0}}-1)[\ol{c^{I,0}}+c^{B,0}]=\partial_z^2 c^{B,0},\ \  (x,z,t)\in \mathbb{R}^2_{+}\times(0,\infty),\\
c^{B,0}(x,z,0)=0,\\
\partial_z c^{B,0}(x,0,t)=0.
\end{array}
\right.
\enn
By the uniqueness of solutions, we deduce from the above initial-boundary value problem that
\be\label{b20}
c^{B,0}(x,z,t)=0\qquad \text{for}\ \ (x,z,t)\in \mathbb{R}^2_{+}\times(0,\infty).
\ee
 Inserting \eqref{b20} into \eqref{b31}, we obtain
 \be\label{b21}
m^{B,0}(x,z,t)=0\qquad \text{for}\ \ (x,z,t)\in \mathbb{R}^2_{+}\times(0,\infty).
\ee
By \eqref{b7}, \eqref{b13}, \eqref{b20} and \eqref{b21}, we deduce from the second equality of \eqref{b16} that
\be\label{b29}
\pt_z^2 m^{B,1}=\ol{m^{I,0}}\pt_z^2 c^{B,1}.
\ee
Integrating the above equality over $(z,\infty)$ and using assumption (H) to get
\be\label{b33}
\pt_z m^{B,1}=\ol{m^{I,0}} \pt_zc^{B,1} \qquad \text{for}\ \ (x,z,t)\in \mathbb{R}^2_{+}\times(0,\infty),
\ee
which, along with \eqref{b5} with $k=0$, \eqref{b20} and \eqref{b21} leads to
\be\label{b23}
\pt_y m^{I,0}(x,0,t)= m^{I,0}(x,0,t)\pt_y c^{I,0}(x,0,t).
\ee
 Then \eqref{e3} follows from \eqref{b8}, \eqref{b10}, \eqref{b15}, \eqref{b17}, \eqref{b13}, \eqref{b3}, \eqref{b7} and \eqref{b23}. \eqref{e4} follows from \eqref{b13}, \eqref{b14}, \eqref{b20} and \eqref{b21}.

By \eqref{b13}, \eqref{b20} and \eqref{b21}, we deduce from the third equality of \eqref{b16} that
\be\label{b30}
 \begin{split}
 \pt_z^2 m^{B,2}=\ol{m^{I,0}}\pt_z^2c^{B,2}+\ol{m^{I,1}}\pt_z^2c^{B,1}
 +\pt_zm^{B,1}\ol{\pt_yc^{I,0}}+\pt_z(m^{B,1}\pt_zc^{B,1})
 +\pt_z (z\ol{\pt_ym^{I,0}}\pt_z c^{B,1}),
 \end{split}
 \ee
 which, along with integration over $(z,\infty)$ and assumption (H) gives
 \be\label{b25}
 \pt_z m^{B,2}=\ol{m^{I,0}}\pt_zc^{B,2}+\ol{m^{I,1}}\pt_zc^{B,1}
 +m^{B,1}\ol{\pt_yc^{I,0}}+m^{B,1}\pt_zc^{B,1}
 +z\ol{\pt_ym^{I,0}}\pt_z c^{B,1}.
 \ee
 From \eqref{b5} with $k=1$, \eqref{b20} and \eqref{b21}, we deduce that
 \be\label{b26}
 \begin{split}
 \pt_z m^{B,2}(x,0,t)=&-\ol{\pt_ym^{I,1}}+\ol{m^{I,0}}\ol{\pt_ym^{I,1}}
 +\ol{m^{I,1}}\ol{\pt_ym^{I,0}}+\ol{m^{I,0}}\pt_zc^{B,2}(x,0,t)\\
 &+\ol{m^{I,1}}\pt_zc^{B,1}(x,0,t)
 +m^{B,1}(x,0,t)\ol{\pt_yc^{I,0}}+m^{B,1}(x,0,t)\pt_zc^{B,1}(x,0,t).
 \end{split}
 \ee
 Setting $z=0$ in \eqref{b25} and using \eqref{b26}, one immediately gets
 \be\label{b27}
 0=\ol{\pt_ym^{I,1}}-\ol{m^{I,0}}\ol{\pt_ym^{I,1}}
 -\ol{m^{I,1}}\ol{\pt_ym^{I,0}}.
 \ee
  Then \eqref{e5} follows from \eqref{b8}, \eqref{b10}, \eqref{b15}, \eqref{b24}, \eqref{b4}, \eqref{b7}, \eqref{b13} and \eqref{b27}. 
  Integrating \eqref{b33} over $(z,\infty)$ and using assumption (H) to get
\be\label{b22}
m^{B,1}=\ol{m^{I,0}} c^{B,1} \qquad \text{for}\ \ (x,z,t)\in \mathbb{R}^2_{+}\times(0,\infty).
\ee 
 Then 
\eqref{e9}-\eqref{e11} follow from \eqref{b4}, \eqref{b6}, \eqref{b13}, \eqref{b14}, \eqref{b20}, \eqref{b21}, \eqref{b22} and the third equality in \eqref{b19}.

Integrating \eqref{b25} over $(z,\infty)$ and using assumption (H) and \eqref{e6}, one gets \eqref{e13}. Inserting \eqref{e13} into the fourth equality of \eqref{b19} and using \eqref{b20}, \eqref{b21}, \eqref{b13}, \eqref{b7}, and \eqref{e6} we derive \eqref{e12}. Omitting the terms containing $m^{I,2}$, $c^{I,2}$, $\vec{u}^{\,I,2}$ or $c^{B,3}$ or their derivatives in the fourth equality of \eqref{b16} and replacing $\partial^2_zm^{B,3}$ with $\pt^2_z\xi$ and using \eqref{b20}, \eqref{b21} and \eqref{e6}, yields
\be\label{b28}
\begin{split}
\partial^2_z\xi=&m^{B,1}_t
+\pt_x[\partial_xm^{B,1}\ol{\partial_xc^{I,0}}
+\partial_x\ol{m^{I,0}}\partial_x c^{B,1}]
+\frac{z^2}{2}\ol{\partial^2_ym^{I,0}}
\partial^2_zc^{B,1}
+\ol{\partial_ym^{I,0}}\partial_zc^{B,2}\\
&
+m^{B,1}\ol{\partial^2_yc^{I,0}}
+z[\ol{\partial_y u^{I,0}_2}\partial_z m^{B,1}
+\ol{\partial_ym^{I,0}}\partial_z^2c^{B,2}
]
+\partial_zm^{B,2}\ol{\partial_yc^{I,0}}
\\
&
+\partial_z[
m^{B,1}\partial_zc^{B,2}+m^{B,2}\partial_zc^{B,1}
]+z[
\partial_zm^{B,1}\overline{\partial^2_yc^{I,0}}+\overline{\partial^2_ym^{I,0}}\partial_zc^{B,1}]-\partial^2_xm^{B,1}
.
\end{split}
\ee
Integrating the above equality twice over $(z,\infty)$ and using assumption (H) yield \eqref{e18}. 
~\\
~\\
\noindent \textbf{Acknowledgement}. This work is supported by National Natural Science Foundation of China (No. 12471195) and Heilongjiang Provincial Natural Science Foundation of China (No. YQ2024A001).

\bibliography{rf}
\bibliographystyle{plain}

\end{document}